\documentclass[12pt]{amsart}  
\usepackage{amssymb}
\usepackage{amscd}
\usepackage{xypic}

\newtheorem{thm}{Theorem}[section]
\newtheorem{cor}[thm]{Corollary}
\newtheorem{prop}[thm]{Proposition}
\newtheorem{lemma}[thm]{Lemma}

\theoremstyle{remark}
\newtheorem{remark}[thm]{Remark}
\newtheorem{example}[thm]{Example}

\theoremstyle{definition}
\newtheorem{definition}[thm]{Definition}

\numberwithin{equation}{subsection}

\newcommand{\supp}{\operatorname{Supp}}

\newcommand{\Diff}{\operatorname{Diff}}

\newcommand{\R}{\mathbf{R}}
\newcommand{\Hom}{\operatorname{Hom}}
\newcommand{\End}{\operatorname{End}}
\newcommand{\Isom}{\operatorname{Isom}}

\newcommand{\isomoto}{\overset{\sim}{\to}}
\newcommand{\id}{{\mathtt{Id}}}

\newcommand{\shHom}{\underline{\operatorname{Hom}}}
\newcommand{\shEnd}{\underline{\operatorname{End}}}

\newcommand{\shAut}{{\underline{\operatorname{Aut}}}}
\newcommand{\shIsom}{{\underline{\operatorname{Isom}}}}

\newcommand{\pr}{{\mathtt{pr}}}

\newcommand{\DR}{\mathtt{DR}}

\newcommand{\vac}{{\mathbf{1}}}

\newcommand{\Ad}{\operatorname{Ad}}
\newcommand{\one}{{1\!\!1}}

\newcommand{\Triv}{\operatorname{Triv}}
\newcommand{\MC}{\operatorname{MC}}

\newcommand{\Def}{\operatorname{Def}}

\newcommand{\Mat}{{\mathtt{Mat}}}
\newcommand{\Tot}{{\mathtt{Tot}}}
\newcommand{\Der}{{\mathtt{Der}}}

\newcommand{\Alg}{\operatorname{ALG}}
\newcommand\real[1]{{\vert{#1}\vert}}
\newcommand{\Cat}{\mathbf{Cat}}

 \DeclareMathOperator{\cotr}{cotr}
\DeclareMathOperator{\ad}{\mathtt{ad}}

\newcommand{\AuxData}{\mathtt{AD}}
\DeclareMathOperator{\Sing}{Sing}
\newcommand{\liminv}{\mathop{\varprojlim}\limits}
\newcommand\s[1]{{(#1)}}
\newcommand{\stack}{\mathtt{st}}

\DeclareMathOperator{\AlgStack}{AlgStack}
\DeclareMathOperator{\Algd}{Algd}
\DeclareMathOperator{\CosMatAlg}{CMA}

\newcommand{\cma}{\mathtt{cma}}
\newcommand{\conv}{\mathtt{conv}}
\newcommand{\triv}{\mathtt{triv}}
\newcommand{\mat}{\mathtt{mat}}

\DeclareMathOperator{\Sh}{Sh} \DeclareMathOperator{\ShAb}{ShAb}
\DeclareMathOperator{\ArtAlg}{ArtAlg}

\renewcommand{\beth}{sd}

\makeatletter
\renewcommand{\subsubsection}{\@startsection
{subsubsection}%
{2}%
{0mm}%
{-\baselineskip}%
{-0.5\baselineskip}%
{\normalfont\normalsize\bfseries }}%
\makeatother

\begin{document}

\title{Deformations of algebroid stacks}

\author[P.Bressler]{Paul Bressler}
\address{Max-Planck-Institut f\"{u}r Mathematik,
Vivatsgasse 7, 53111 Bonn, Germany} \email{bressler@gmail.com}

\author[A.Gorokhovsky]{Alexander Gorokhovsky}
\address{Department of Mathematics, UCB 395,
University of Colorado, Boulder, CO~80309-0395, USA}
\email{Alexander.Gorokhovsky@colorado.edu}

\author[R.Nest]{Ryszard Nest}
\address{Department of Mathematics,
Copenhagen University, Universitetsparken 5, 2100 Copenhagen,
Denmark}
 \email{rnest@math.ku.dk}

\author[B.Tsygan]{Boris Tsygan}
\address{Department of
Mathematics, Northwestern University, Evanston, IL 60208-2730, USA}
\email{tsygan@math.northwestern.edu}

\begin{abstract}

In this paper we consider deformations of an algebroid stack on an \'etale groupoid.  We construct a differential
graded Lie algebra (DGLA) which controls this deformation theory.  In the case when the algebroid is a twisted form of
functions we show that this DGLA is quasiisomorphic to the twist  of the DGLA of Hochschild cochains on the algebra of
functions on the groupoid  by the characteristic class   of the corresponding gerbe.
\end{abstract}

\thanks{A. Gorokhovsky was partially supported by NSF grant DMS-0400342. B. Tsygan
was partially supported by NSF grant DMS-0605030}
\maketitle \vspace{-0.5cm} \tableofcontents

\section{Introduction}

The two main results of the paper are the following.
\begin{enumerate}
\item We classify deformations of an algebroid stack on an \'{e}tale groupoid by Maurer-Cartan elements of a
    differential graded Lie algebra (DGLA) canonically associated to the algebroid stack, see Theorem \ref{mt1}.

\item In the particular case, let the algebroid stack be a twisted form of the structure sheaf (i.e. is associated
    to a gerbe on the groupoid). We construct a quasiisomorphism of the DGLA alluded to above with the DGLA of
Hochschild cochains on the algebra of functions on the groupoid, twisted by the class of the gerbe), see Theorems
    \ref{equiv of 2-gpds} and \ref{mt3}
\end{enumerate}
In the case when the \'etale groupoid is a manifold these results were established in \cite{bgnt1}.

Recall that a deformation of an algebraic structure, say, over ${\mathbb C}$ is a structure over ${\mathbb C}[[\hbar]]$
whose reduction modulo $\hbar$ is the original one. Two deformations are said to be isomorphic if there is an
isomorphism of the two structures over ${\mathbb C}[[\hbar]]$ that is identity modulo $\hbar$. The algebra ${\mathbb
C}[[\hbar]]$ may be replaced by any commutative (pro)artinian algebra ${\mathfrak a}$ with the maximal ideal
${\mathfrak m}$ such that ${\mathfrak a}/{\mathfrak m}\isomoto {\mathbb C}.$

It has been discovered in \cite{GM}, \cite{S}, \cite{SS} that deformations of many types of objects are controlled by a
differential graded Lie algebra, or a DGLA, in the following sense. Let ${\mathfrak g}$ be a DGLA. A Maurer-Cartan
element over ${\mathfrak a}$ is an element $\gamma \in {\mathfrak g}^1\otimes {\mathfrak m}$ satisfying
$$d\gamma + {\frac{1}{2}}[\gamma,\gamma]=0.$$
One can define the notion of equivalence of two Maurer-Cartan elements (essentially, as a gauge equivalence over the
group ${\rm{exp}}({\mathfrak g}^0\otimes {\mathfrak m})$). The set of isomorphism classes of deformations over
${\mathfrak a}$ is in a bijection with the set of equivalence classes of Maurer-Cartan elements of ${\mathfrak g}$ over
${\mathfrak a}.$ This is true for such objects as flat connections in a bundle on a manifold (not surprisingly), but
also for associative and Lie algebras, complex structures on a manifold, etc. For an associative algebra, the DGLA
controlling its deformations is the Hochschild cochain complex shifted by one, $C^\bullet(A,A)[1]$ (cf. \cite{Ge}).

Now let us pass to sheaves of algebras. It turns out that their deformations are still controlled by DGLAs. Two new
points appear:  there is a technical issue of defining this DGLA and the most natural DGLA of this sort actually
controls deformations of a sheaf within a bigger class of objects, not as a sheaf of algebras but as an algebroid
stack.

For a sheaf ${\mathcal A}$  of algebras on a space $X$, one can define the DGLA which is, essentially, the complex of
cochains of $X$ with coefficients in the (sheafification of the presheaf) $C^{\bullet}({\mathcal A},{\mathcal A})[1]$,
in the sense which we now describe (compare \cite{hi1, hi2}).

If by cochains with coefficients in a sheaf of DGLAs one means \v{C}ech cochains then it is not clear how to define on
them a DGLA structure. Indeed, one can multiply \v{C}ech cochains, but the product is no longer commutative. Therefore
there is no natural bracket on Lie algebra valued \v{C}ech cochains, for the same reason as there is no bracket on a
tensor product of a Lie algebra and an associative algebra. The problem is resolved if one replaces \v{C}ech cochains
by another type of cochains that have a (skew)commutative product. This is possible only in characteristic zero (which
is well within the scope of our work). In fact there are several ways of doing this: De Rham-Sullivan forms on a
simplicial set; jets with the canonical connection on a smooth manifold (real, complex, or algebraic); and Dolbeault
forms on a complex manifold. The first method works for any space and for any sheaf (the simplicial set is the nerve of
an open cover; one has to pass to a limit over the covers to get the right answer). In order be able to write a complex
of jets, or a Dolbeault complex, with coefficients in a sheaf, one has to somewhat restrict the class of sheaves. The
sheaf of Hochschild complexes is within this restricted class for a lot of naturally arising sheaves of algebras.

The DGLA of cochains with coefficients in the Hochschild complex controls deformations of ${\mathcal A}$ as an
algebroid stack. An algebroid stack is a natural generalization of a sheaf of algebras. It is a sheaf of categories
with some additional properties; a sheaf of algebras gives rise to a stack which is the sheaf of categories of modules
over this sheaves of algebras. In more down to earth terms, an algebroid stack can be described by a descent datum,
i.e. a collection of locally defined sheaves of algebras glued together with a twist by a two-cocycle (cf. below). The
role of algebroid stacks in deformation theory was first emphasized in \cite{Ka}, \cite{ko1}.

For a complex manifold with a (holomorphic) symplectic structure the canonical deformation quantization is an algebroid
stack. The first obstruction for this algebroid stack to be (equivalent to) a sheaf of algebras is the first
Rozansky-Witten invariant \cite{bgnt1}.

In light of this, it is very natural to study deformation theory of algebroid stacks. In \cite{bgnt1} we showed that it
is still controlled by a DGLA. This DGLA is the complex of De Rham-Sullivan forms on the first barycentric subdivision
of the nerve of an open cover with coefficients in a sheaf of Hochschild complexes of the sheaf of twisted matrix
algebras constructed from a descent datum. To get the right answer one has to pass to a limit over all the covers.

An important special case of an algebroid stack is a gerbe, or a twisted form of the structure sheaf on a manifold.
Gerbes on $X$ are classified by the second cohomology group $H^2(X,{\mathcal O}^*_X).$ For a class $c$ in this group,
one can pass to its image $\partial c$ in $H^3 (X, 2\pi i{\mathbb Z})$ or to the projection ${\widetilde{c}}$ from
$H^2(X,{\mathcal O}^*_X )= H^2(X,{\mathcal O}_X/2\pi i{\mathbb Z})$ to $H^2(X,{\mathcal O}_X/{\mathbb C})$. One can
define also $d{\rm{log}}c\in H^2(X, \Omega^{1,{\rm{closed}}})$. In \cite{bgnt1}, we proved that the DGLA controlling
deformations of a gerbe as an algebroid stack is equivalent to the similar DGLA for the trivial gerbe (i.e. cochains
with values in the Hochschild complex) twisted by the class ${\widetilde{c}}$. In a forthcoming work, we prove the
formality theorem, namely that the latter DGLA is equivalent to the DGLA of multivector fields twisted by $\partial c$
(in the real case) or by $d{\rm{log}}c$ (complex analytic case).

In this paper, we generalize these results from manifolds to \'{e}tale groupoids. The DGLA whose Maurer-Cartan elements
classify deformations of an algebroid stack is constructed as follows. From an algebroid stack on Hausdorff \'etale
groupoid one passes to a cosimplicial matrix algebra on the nerve of the groupoid. If the groupoid is non-Hausdorff one
has to replace the groupoid by its embedding category, cf. \cite{ieke95}.

For a cosimplicial matrix algebra we form the Hochschild cochain complex which happens to be a cosimplicial sheaf of
DGLAs not on the nerve itself, but on its first subdivision. From this we pass to a cosimplicial DGLA and to its
totalization which is an ordinary DGLA.

For a gerbe on an \'etale groupoid this DGLA can be replaced by another of a more familiar geometric nature, leading to
the Theorems \ref{equiv of 2-gpds} and \ref{mt3}.

Let us say a few more words about motivations behind this work. Twisted modules over algebroid stack deformations of
the structure sheaf, or DQ modules, are being extensively studied in \cite{KS1}, \cite{KS2}. This study, together with
the direction of \cite{bgnt1}, \cite{bgnt2} and the present work, includes or should eventually include the Hochschild
homology and cohomology theory, the cyclic homology, characteristic classes, Riemann-Roch theorems. The context of a
deformation of a gerbe on an \'etale groupoid provides a natural generality for all this.  Similarly, the groupoid with
a symplectic structure seems to be a natural context and for the Rozansky-Witten model of 3-dimensional topological
quantum field theory. Note that this theory is naturally related do deformation quantization of the sheaf of ${\mathcal
O}$-modules as a symmetric monoidal category \cite{KRS}.

As another example, a Riemann-Roch theorem for deformation quantizations of \'etale groupoids should imply a (higher)
index theorem for Fourier elliptic operators given by kernels whose wave front is the graph of a characteristic
foliation, in the same way as its partial case for symplectic manifolds implies the index theorem for elliptic pairs
\cite{BNT}.

Let us describe this situation in more detail. Let $\Sigma$ be a coisotropic submanifold of a symplectic manifold $M$.
The holonomy groupoid of the characteristic foliation on $\Sigma$ is an \'etale groupoid with a symplectic structure.
The canonical deformation quantization of this \'etale groupoid is an algebroid stack, similarly to the case of
deformations of complex symplectic manifolds that was discussed above. The Rozhansky-Witten class can be defined in
this situation as well.

The canonical deformation quantization of the symplectic \'etale groupoid associated to a coisotropic submanifold
naturally arises in the study of the following question motivated by problems in microlocal analysis. For a coisotropic
submanifold $\Sigma$ of a symplectic manifold $M$ consider the graph $\Lambda$ of the characteristic foliation which is
a Lagrangian submanifold of the product $M\times M^{op}$. When $M = T^*X$ for a manifold $X$, and $\Sigma$ is conic,
the Lagrangian $\Lambda$ is conic as well, hence determines a class of Fourier integral operators given by kernels
whose wave fronts are contained in $\Lambda$, compare \cite{gs79}. These operators form an algebra under composition
since the composition $\Lambda\circ\Lambda$ coincides with $\Lambda$. An asymptotic version of the operator product
gives rise to a deformation of the foliation  algebra (compare \cite{bg}) which we will discuss in a subsequent work.
The canonical deformation quantization of the holonomy groupoid is, in a suitable sense, Morita equivalent to this
algebra.

This paper is organized as follows. In the section \ref{preliminaries} we give overview of the preliminaries from the
category theory and the theory of (co)simplicial spaces. We also describe I.~Moerdijk's construction of embedding
category and stacks on \'etale categories.

In the section \ref{Deformations and DGLA} we review relation between the deformation and differential graded Lie
algebras as well as discuss in this context deformations of matrix Azumaya algebras.

In the section \ref{Deformations of cosimplicial matrix Azumaya algebras} we introduce the notion of cosimplicial
matrix algebra and construct a DGLA governing deformations of matrix algebras (cf. Theorem \ref{thm:mapgroup equiv}).
We then specialize to the case of cosimplicial matrix Azumaya algebras. In this case using the differential geometry of
the infinite jet bundle we are able to show that the deformation theory of a cosimplicial matrix Azumaya algebra
$\mathcal{A}$ is controlled by DGLA of jets of Hochschild cochains twisted by the cohomology class of the gerbe
associated with $\mathcal{A}$.

Finally in the section \ref{Applications to etale groupoids} we apply the results of the previous section to study the
deformation theory of stacks on \'etale groupoids. We also use these methods to study the deformation theory of a
twisted convolution algebra of \'etale  groupoid.
\section{Preliminaries}\label{preliminaries}

\subsection{Categorical preliminaries}

\subsubsection{The simplicial category}
For $n=0,1,2,\ldots$ we denote by $[n]$ the category with objects
$0,\ldots,n$ freely generated by the graph
\[
0\to 1\to\cdots\to n \ .
\]
For $0\leq i\leq j\leq n$ we denote by $(ij)$ the unique arrow $i\to
j$ in $[n]$. We denote by $\Delta$ the full subcategory of $\Cat$
with objects the categories $[n]$ for $n = 0,1,2,\ldots$.

For a category $\mathcal{C}$ we refer to a functor $\lambda \colon
[n] \to \mathcal{C}$ as \emph{ a(n $n$-)simplex in $\mathcal{C}$}.
For a morphism $f \colon  [m] \to [n]$ in $\Delta$ and a simplex
$\lambda \colon  [n] \to \mathcal{C}$ we denote by $f^*(\lambda)$
the simplex $\lambda\circ f$.

Suppose that $f \colon  [m] \to [n]$ is a morphism in $\Delta$ and
$\lambda \colon  [n] \to \mathcal{C}$ is a simplex. Let
$\mu=f^*(\lambda)$ for short. The morphism $(0n)$ in $[n]$ factors
uniquely into the composition $0\to f(0)\to f(m)\to n$ which, under
$\lambda$, gives the factorization of $\lambda(0n)\colon
\lambda(0)\to\lambda(n)$ in $\mathcal{C}$ into the composition
\begin{equation}\label{factorization}
\lambda(0)\xrightarrow{a} \mu(0)
\xrightarrow{b}\mu(m)\xrightarrow{c} \lambda(n)\ ,
\end{equation}
where $b=\mu((0m))$.

For $0\leq i\leq n+1$ we denote by $\partial_i = \partial^n_i \colon
[n]\to[n+1]$ the $i^{\text{th}}$ face map, i.e. the unique injective
map whose image does not contain the object $i\in[n+1]$.

For $0\leq i\leq n-1$ we denote by $s_i=s^n_i\colon [n]\to[n-1]$ the
$i^{\text{th}}$ degeneracy map, i.e. the unique surjective map such
that $s_i(i)=s_i(i+1)$.

\subsubsection{Subdivision}\label{subsubsection: subdivision}
For $\lambda \colon  [n] \to \Delta$ define $\lambda_k$ by
$\lambda(k) = [\lambda_k]$, $k=0, 1,\ldots, n$. Let $\beth^\lambda
\colon  [n] \to \lambda(n)$ be a morphism in $\Delta$ defined by
\begin{equation}
\beth^\lambda(k)=\lambda(kn)(\lambda_k) .
\end{equation}
For a morphism $f \colon  [m] \to [n]$ in $\Delta$ let
$\beth(f)^\lambda = c$ in the notations of \eqref{factorization}.
Then, the diagram
\[
\begin{CD}
[m] @>{f}>> [n] \\
@V{\beth^{f^*(\lambda)}}VV @VV{\beth^\lambda}V \\
f^*(\lambda)(m) @>{\beth(f)^\lambda}>> \lambda(n)
\end{CD}
\]
is commutative.

\subsubsection{(Co)simplicial objets}
A simplicial object in a category $\mathcal C$ is a functor
$\Delta^{op}\to\mathcal{C}$. For a simplicial object $X$ we denote
$X([n])$ by $X_n$.

A cosimplicial object in a category $\mathcal C$ is a functor
$\Delta\to\mathcal{C}$. For a cosimplicial object $Y$ we denote
$Y([n])$ by $Y^n$.

Simplicial (respectively, cosimplicial) objects in $\mathcal C$ form
a category in the standard way.

\subsubsection{Nerve}
For $n=0,1,2,\ldots$ let $N_n\mathcal{C} \colon
=\Hom([n],\mathcal{C})$. The assignment $n\mapsto N_n\mathcal{C}$,
$([m]\xrightarrow{f}[n])\mapsto N\mathcal{C}(f)\colon
=(\lambda\mapsto\lambda\circ f)$ defines the simplicial set
$N\mathcal{C}$ called \emph{the nerve of $\mathcal C$}.

The effect of the face map $\partial^n_i$ (respectively, the
degeneracy map $s^n_i$) will be denoted by $d^i = d_n^i$
(respectively, $\varsigma_i = \varsigma^n_i$).

\subsubsection{Subdivision of (co)simplicial objects}\label{geometric realization}
Assume that coproducts in $\mathcal C$ are represented. Let
$X\in\mathcal{C}^{\Delta^{op}}$.

For $\lambda\colon [n]\to\Delta$ let $\real{X}_\lambda \colon
=X_{\lambda(n)}$, $\real{X}_n\colon =\coprod_{\lambda\colon
[n]\to\Delta}\real{X}_\lambda$.

For a morphism $f \colon  [m] \to [n]$ in $\Delta$ let $\real{X}(f)
\colon \real{X}_n\to\real{X}_m$ denote the map whose restriction to
$\real{X}_\lambda$ is the map $X(\beth(f)^\lambda)$.

The assignment $[n]\mapsto \real{X}_n$, $f\mapsto \real{X}(f)$
defines the simplicial object $\real{X}$ called \emph{the
subdivision} of $X$.

Let $\beth(X)_n \colon  \real{X}_n \to X_n$ denote the map whose
restriction to $\real{X}_\lambda$ is the map $X(\beth^\lambda)$. The
assignment $[n]\mapsto \beth(X)_n$ defines the canonical morphism of
simplicial objects
\[
\beth(X) \colon  \real{X} \to X .
\]

Suppose that $\mathcal{C}$ has products. For
$V\in\mathcal{C}^\Delta$, $\lambda \colon  [n] \to \Delta$ let
$\real{V}^{\lambda} =V^{\lambda(n)}$, $\real{V}^n =
\prod_{[n]\xrightarrow{\lambda}\Delta} \real{V}^{\lambda}$. For a
morphism $f \colon  [m] \to [n]$ in $\Delta$ let $\real{V}(f) \colon
\real{V}_m \to \real{V}_n$ denote the map  such that
$\pr_\lambda\circ\real{V}(f) =
V(\beth(f)^\lambda)\circ\pr_{f^*(\lambda)}$. The assignment
$[n]\mapsto \real{V}_n$, $f\mapsto \real{V}(f)$ defines the
cosimplicial object $\real{V}$ called \emph{the subdivision} of $V$.

Let $\beth(V)^n \colon  V^n \to \real{V}^n$ denote the map such that
$\pr_\lambda\circ\beth(V)^n = V(\beth^\lambda)$. The assignment
$[n]\mapsto \beth(V)_n$ defines the canonical morphism of simplicial
objects
\[
\beth(V) \colon  V \to \real{V} .
\]

\subsubsection{Totalization of cosimplicial vector spaces} Next, we recall the
definition of the functor $\Tot$ which assigns to a cosimplicial
vector space a complex of vector spaces.

For $n= 0, 1, 2, \ldots$ let $\Omega_n$ denote the polynomial de
Rham complex of the $n$-dimensional simplex. In other words,
$\Omega_n$ is the DGCA generated by $t_0,\ldots,t_n$ of degree zero
and $dt_0,\ldots,dt_n$ of degree one subject to the relations
$t_0+\dots+t_n=1$ and $dt_0+\dots+dt_n=0$; the differential
$d_\Delta$ is determined by the Leibniz rule and $t_i\mapsto dt_i$.
The assignment $\Delta\ni[n]\mapsto\Omega_n$ extends in a natural
way to a simplicial DGCA.

Suppose that $V$ is a cosimplicial vector space. For each morphism
$f \colon  [p]\to [q]$ in $\Delta$ we have the morphisms $V(f)
\colon  V^p\to V^q$ and $\Omega(f) \colon  \Omega_q\to \Omega_p$.
Let $\Tot(V)^k\subset \prod_n \Omega^k_n\otimes V^n$  denote the
subspace which consists of those $a = (a_n)$ which satisfy the
conditions $(\id\otimes V(f))(a_p)= (\Omega(f)\otimes\id)(a_q)$ for
all $f \colon  [p]\to [q]$. The de Rham differential $d_\Delta
\colon  \Omega^k \to \Omega^{k+1}$ induces the differential in
$\Tot(V)$. It is clear that the assignment $V \mapsto \Tot(V)$ is a
functor on the category of cosimplicial vectors spaces with values
in the category of complexes of vector spaces.

%
%
%

\subsubsection{Cohomology of cosimplicial vector spaces}
For a cosimplicial vector space $V$ we denote by $C(V)$
(respectively, $N(V)$), the associated complex (respectively,
normalized complex). These are given by $C^i(V) = N^i(V) = 0$ for $i
< 0$ and, otherwise, by
\[
C^n(V) = V^n, \ \ \ \ N^n(V) = \bigcap_i\ker(s_i \colon  V^n\to
V^{n-1})
\]
In either case the differential $\partial = \partial_V$ is given by
$\partial^n = \sum_i (-1)^i\partial^n_i$. The natural inclusion
$N(V)\hookrightarrow C(V)$ is a quasiisomorphism (see e.g.
\cite{ez}). Let $H^\bullet(V) \colon = H^\bullet(N(V)) =
H^\bullet(C(V))$.

We will also need the following result, see e.g.\cite{seg}:
\begin{lemma}\label{lemma: subdiv quism}
Suppose that $V$ is a cosimplicial vector space. The map
$H^\bullet(\beth(V)) \colon  H^\bullet(V) \to H^\bullet(\real{V})$
is an isomorphism.
\end{lemma}

\subsubsection{}
There is a natural quasiisomorphism
\[
\int_V \colon  \Tot(V) \to N(V)
\]
(see e.g. \cite{dupont}) such that the composition
\[
H^0(V) \xrightarrow{1\otimes\id} \Tot(V) \xrightarrow{\int_V} N(V)
\]
coincides with the inclusion $H^0(V)\hookrightarrow N(V)$.

Suppose that $(V, d)$ is a cosimplicial complex of vector spaces.
Then, for $i\in\mathbb{Z}$, we have the cosimplicial vector space
$V^{\bullet,i} \colon  [n] \mapsto (V^n)^i$. Applying $\Tot(\ )$,
$N(\ )$ and $C(\ )$ componentwise we obtain double complexes whose
associated total complexes will be denoted, respectively, by
$\Tot(V)$, $N(V)$ and $C(V)$. The maps
\[
H^0(V^{\bullet,i}) \to \Tot(V^{\bullet,i})
\]
give rise to the map of complexes
\begin{equation}\label{ker to Tot cxs}
\ker(V^0 \rightrightarrows V^1) \to \Tot(V)
\end{equation}

\begin{lemma}\label{lemma: ker to Tot cxs}
Suppose that $V$ is a cosimplicial complex of vector spaces such
that
\begin{enumerate}
\item there exists $M\in\mathbb{Z}$ such that, for all $i < M$, $n = 0,1,\ldots$, $V^{n,i} = 0$ (i.e. the complexes
    $V^n$ are bounded below uniformly in $n$)
\item for all $i\in\mathbb{Z}$ and $j\neq 0$, $H^j(V^{\bullet,i}) = 0$.
\end{enumerate}
Then, the map \eqref{ker to Tot cxs} is a quasiisomorphism.
\end{lemma}
\begin{proof}
It suffices to show that the composition
\[
\ker(V^0 \rightrightarrows V^1) \to \Tot(V) \xrightarrow{\int_V}
N(V)
\]
is a quasiisomorphism. Note that, for $i\in\mathbb{Z}$, the
acyclicity assumption implies that the composition
\[
\ker(V^{0,i} \rightrightarrows V^{1,i}) = H^0(V^{\bullet,i}) \to
\Tot(V^{\bullet,i}) \to N(V^{\bullet,i}))
\]
is a quasiisomorphism. Since the second map is a quasiisomorphism so
is the first one. Since, by assumption, $V$ is uniformly bounded
below the claim follows.
\end{proof}

\subsection{Cosimplicial sheaves and stacks}
We will denote by $\Sh_R(X)$ the category of sheaves of $R$-modules
on $X$.

\subsubsection{Cosimplicial sheaves} Suppose that
$X$ is a simplicial space.

A cosimplicial sheaf of vector spaces (or, more generally, a
cosimplicial complex of sheaves) $\mathcal F$ on $X$ is given by the
following data:
\begin{enumerate}
\item for each $p=0,1,2,\ldots$ a sheaf ${\mathcal F}^p$ on $X_p$ and
\item for each morphism $f\colon [p]\to[q]$ in $\Delta$ a morphism $f_*\colon  X(f)^{-1}{\mathcal F}^p\to{\mathcal
    F}^q$.
\end{enumerate}
These are subject to the condition: for each pair of composable
arrows $[p]\stackrel{f}{\to}[g]\stackrel{g}{\to}[r]$ the diagram map
$(g\circ f)_*\colon  X(g\circ f)^{-1}\mathcal{F}^p\to \mathcal{F}^r$
is equal to the composition $X(g\circ f)^{-1}\mathcal{F}^p\cong
X(g)^{-1}X(f)^{-1}\stackrel{X(g)^{-1}f_*}{\longrightarrow}
X(g)^{-1}\mathcal{F}^q\stackrel{g_*}{\to}\mathcal{F}^r$.

\begin{definition}\label{def special sheaf}
A cosimplicial sheaf $\mathcal{F}$ is \emph{special} if the
structure morphisms $f_*\colon  X(f)^{-1}{\mathcal F}^p\to{\mathcal
F}^q$ are isomorphisms for all $f$.
\end{definition}

\subsubsection{Cohomology of cosimplicial sheaves}
For a cosimplicial sheaf of vector spaces $\mathcal{F}$ on $X$ let
$\Gamma(X;\mathcal{F})^n = \Gamma(X_n;\mathcal{F}^n)$. For a
morphism $f\colon  [p] \to [q]$ in $\Delta$ let $f_* =
\Gamma(X;\mathcal{F})(f) \colon \Gamma(X;\mathcal{F})^p \to
\Gamma(X;\mathcal{F})^q$ denote the composition
\[
\Gamma(X_p;\mathcal{F}^p) \xrightarrow{X(f)^*}
\Gamma(X_q;X(f)^{-1}\mathcal{F}^p) \xrightarrow{\Gamma(X_q;f_*)}
\Gamma(X_q;\mathcal{F}^q)
\]
The assignments $[p]\mapsto\Gamma(X;\mathcal{F})^p$, $f\mapsto f_*$
define a cosimplicial vector space denoted $\Gamma(X;\mathcal{F})$.

The functor $\mathcal{F} \mapsto H^0(\Gamma(X;\mathcal{F}))$ from
the (abelian) category of cosimplicial sheaves of vector spaces on
$X$ to the category of vector spaces is left exact. Let
$\mathbf{R}\Gamma(X;\mathcal{F})) =
\mathbf{R}H^0(\Gamma(X;\mathcal{F})$, $H^i(X;\mathcal{F}) =
\mathbf{R}^i\Gamma(X;\mathcal{F}))$.

Assume that ${\mathcal F}$ satisfies $H^i(X_j;{\mathcal F}^j)=0$ for
$i\neq 0$. For complexes of sheaves the assumption on ${\mathcal
F}^j$ is that the canonical morphism in the derived category
$\Gamma(X_j;{\mathcal F}^j)\to\R\Gamma(X_j;{\mathcal F}^j)$ is an
isomorphism. Under the acyclicity assumption on $\mathcal{F}$ we
have
\[
H^j(X;\mathcal{F})=H^j(\Gamma(X;\mathcal{F}))
\]

\subsubsection{Sheaves on the subdivision}\label{sheaves on
subdivision} For a cosimplicial sheaf $\mathcal{F}$ on $X$ let
\[
\real{\mathcal{F}} \colon = \beth(X)^{-1}\mathcal{F} \ .
\]

Thus, $\real{\mathcal{F}}$ is a cosimplicial sheaf on $\real{X}$ and
\[
\Gamma(\real{X};\real{\mathcal{F}})^n =
\Gamma(\real{X}_n;\real{\mathcal{F}}_n) = \prod_{\lambda \colon
[n]\to\Delta}
\Gamma(X_{\lambda(n)};\mathcal{\mathcal{F}}^{\lambda(n)}) =
\real{\Gamma(X;\mathcal{F})}_n
\]
i.e. the canonical isomorphism of cosimplicial vector spaces
\[
\Gamma(\real{X};\real{\mathcal{F}}) = \real{\Gamma(X;\mathcal{F})}
\]

\begin{lemma}
The map $\beth(X)^* \colon  H^\bullet(X;\mathcal{F}) \to
H^\bullet(\real{X};\real{\mathcal{F}})$ is an isomorphism.
\end{lemma}
\begin{proof}
Follows from Lemma \ref{lemma: subdiv quism}
\end{proof}

\subsubsection{}\label{realization for special sheaves}
We proceed with the notations introduced above. Suppose that
$\mathcal{F}$ is a special cosimplicial sheaf on $X$. Then,
$\real{\mathcal{F}}$ admits an equivalent description better suited
for our purposes.

Let $\real{\mathcal{F}}^{\prime n}$ denote the sheaf on $\real{X}_n$
whose restriction to $\real{X}_\lambda$ is given by
$\real{\mathcal{F}}^\prime_\lambda \colon =
X(\lambda(0n))^{-1}\mathcal{F}^{\lambda(0)}$. For a morphism of
simplices $f \colon  \mu\to\lambda$ the corresponding (component of
the) structure map $f^\prime_*$ is defined as the unique map making
the diagram
\[
\begin{CD}
X(c)^{-1}X(\mu((0m)))^{-1}\mathcal{F}^{\mu(0)} @>{f^\prime_*}>>
X(\lambda(0n))^{-1}\mathcal{F}^{\lambda(0)} \\
@V{X(c)^{-1}\mu((0m))_*}VV @VV{\lambda(0n)_*}V \\
 X(c)^{-1}\mathcal{F}^{\mu(m)} @>{c_*}>>
\mathcal{F}^{\lambda(n)}
\end{CD}
\]
commutative. Note that $f^\prime_*$ exists and is unique since the
vertical maps are isomorphisms as $\mathcal{F}$ is special.

It is clear that $\real{F}^\prime$ is a cosimplicial sheaf on
$\real{X}$; moreover there is a canonical isomorphism
$\real{F}^\prime\to\real{F}$ whose restriction to $\real{X}_\lambda$
is given by the structure map $\lambda(0n)_*$.

\subsubsection{Stacks}
We refer the reader to \cite{SGA1} and \cite{Vist} for basic
definitions. We will use the notion of fibered category
interchangeably with that of a pseudo-functor. A \emph{prestack}
$\mathcal C$ on a space $Y$ is a category fibered over the category
of open subsets of $Y$, equivalently, a pseudo-functor $U\mapsto
\mathcal{C}(U)$, satisfying the following additional requirement.
For an open subset $U$ of $Y$ and two objects $A,B\in\mathcal{C}(U)$
we have the presheaf $\shHom_\mathcal{C}(A,B)$ on $U$ defined by
$U\supseteq V\mapsto \Hom_{\mathcal{C}(V)}(A\vert_V, B\vert_B)$. The
fibered category $\mathcal{C}$ is a prestack if for any $U$,
$A,B\in\mathcal{C}(U)$, the presheaf $\shHom_\mathcal{C}(A,B)$ is a
sheaf. A prestack is a \emph{stack} if, in addition, it satisfies
the condition of effective descent for objects. For a prestack
$\mathcal{C}$ we denote the associated stack by
$\widetilde{\mathcal{C}}$.

A stack in groupoids $\mathcal{G}$ is a gerbe if it is locally nonempty and locally connected, i.e. it satisfies
\begin{enumerate}
\item any point $y\in Y$ has a neighborhood $U$ such that $\mathcal{G}(U)$ is nonempty;

\item for any $U\subseteq Y$, $y\in U$, $A,B\in\mathcal{G}(U)$ there exists a neighborhood $V\subseteq U$ of $y$
    and an isomorphism $A\vert_V \cong B\vert_V$.
\end{enumerate}

Let $\mathcal{A}$ be a sheaf of abelian groups on $Y$. An $\mathcal{A}$-gerbe is a gerbe $\mathcal{G}$  with the
following additional data: an isomorphism $\lambda_A \colon \shEnd(A) \to \mathcal{A}|_U$ for every open $U \subset Y$
and $A\in \mathcal{G}(U)$. These isomorphisms are required to satisfy the following compatibility condition. Note that
if $B\in \mathcal{G}(U)$ and $B\cong A$ then there exists an isomorphism $\lambda_{AB} \colon \shEnd(B) \to \shEnd(A)$
(canonical since $\mathcal{A}$ is abelian). Then the identity $\lambda_B=\lambda_A \circ \lambda_{AB}$ should hold.

\subsubsection{Cosimplicial stacks}
Suppose that $X$ is an \'etale simplicial space.

A cosimplicial stack $\mathcal{C}$ on $X$ is given by the following
data:
\begin{enumerate}
\item for each $[p] \in \Delta$ as stack $\mathcal{C}^p$ on $X_p$ ;

\item for each morphism $f \colon [p] \to [q]$ in $\Delta$ a $1$-morphism of stacks $\mathcal{C}_f \colon
    X(f)^{-1}\mathcal{C}^p \to \mathcal{C}^q$ ;

\item for any pair of morphisms $[p] \xrightarrow{f} [q] \xrightarrow{g} [r]$ a $2$-morphism $\mathcal{C}_{f,g}
    \colon \mathcal{C}_g \circ X(g)^*(\mathcal{C}_f) \to \mathcal{C}_{g\circ f}$
\end{enumerate}
These are subject to the associativity condition: for a triple of
composable arrows $[p]\stackrel{f}{\longrightarrow}[q]
\stackrel{g}{\longrightarrow}[r]\stackrel{h}{\longrightarrow}[s]$
the equality of $2$-morphisms
\[
\mathcal{C}_{g\circ f, h}\circ
(X(h)^{-1}\mathcal{C}_{f,g}\otimes\id_{\mathcal{C}_h}) =
\mathcal{C}_{f, h\circ g}\circ(\id_{X(h\circ
g)^{-1}\mathcal{C}_f}\otimes\mathcal{C}_{g,h})
\]
holds. Here and below we use $\otimes$ to denote the horizontal
composition of $2$-morphisms.

Suppose that $\mathcal{C}$ and $\mathcal{D}$ are cosimplicial stacks
on $X$. A $1$-morphism $\phi \colon \mathcal{C} \to \mathcal{D}$ is
given by the following data:
\begin{enumerate}
\item for each $[p] \in \Delta$ a $1$-morphism $\phi^p \colon \mathcal{C}^p \to \mathcal{D}^p$

\item for each morphism $f \colon [p] \to [q]$ in $\Delta$ a $2$-isomorphism $\phi_f \colon
    \phi^q\circ\mathcal{C}_f \to \mathcal{D}_f\circ X(f)^*(\phi^p)$
\end{enumerate}
which, for every pair of morphisms $[p] \xrightarrow{f} [q]
\xrightarrow{g} [r]$ satisfy
\begin{multline*}
(\mathcal{D}_{f,g} \otimes \id_{X(gf)^*(\phi^p)})
\circ(X(g)^*(\phi_f)\otimes \id_{\mathcal{D}_g})\circ
(\id_{X(g)^*(\mathcal{C}_f)} \otimes \phi_g) \\= \phi_{g\circ f}
\circ (\id_{\phi^r}\otimes \mathcal{C}_{f,g})
\end{multline*}

Suppose that $\phi$ and $\psi$ are $1$-morphisms of cosimplicial
stacks $\mathcal{C} \to \mathcal{D}$. A $2$-morphism $b \colon  \phi
\to \psi$ is given by a collection of $2$-morphisms $b^p \colon
\phi^p \to \psi^p$, $p = 0, 1, 2, \ldots$, which satisfy
\[
\psi_f \circ (b^q\otimes\id_{\mathcal{C}_f}) =
(\id_{\mathcal{D}_f}\otimes X(f)^*(b^p)) \circ \phi_f
\]
for all $f \colon [p] \to [q]$ in $\Delta$.

\subsubsection{Cosimplicial gerbes}\label{subsubsection: cosimp gerbe}
Suppose that $A$ is an abelian cosimplicial sheaf on $X$. A
cosimplicial $A$-gerbe $\mathcal{G}$ on $X$ is a cosimplicial stack
on $X$ such that
\begin{enumerate}
\item for each $[p]\in\Delta$, $\mathcal{G}^p$ is a $A^p$-gerbe on $X_p$.

\item For each morphism $f\colon [p]\to[q]$ in $\Delta$ the $1$-morphism
\[\mathcal{G}_f\colon
    X(f)^{-1}\mathcal{G}^p \to\mathcal{G}^q
    \]
of stacks compatible with the map $A_f \colon  X(f)^{-1}A^p \to
A^q$.
\end{enumerate}

\subsection{Sheaves and stacks on \'etale categories}\label{subsection: dedushka}

\subsubsection{\'Etale categories}
In what follows ($C^\infty$-)manifolds are not assumed to be
Hausdorff. Note, however, that, by the very definition a manifold is
a locally Hausdorff space. An \emph{\'etale} map of manifolds is a
local diffeomorphism.

An \emph{\'etale category} $G$ is a category in manifolds with the manifold of objects $N_0G$, the manifold of
morphisms $N_1G$ and \'etale structure maps. Forgetting the manifold structures on the sets of objects and morphisms
one obtains the underlying category.

If the underlying category of an \'etale category $G$ is a groupoid
and the inversion map $N_1G \to N_1G$ is $C^\infty$ one says that
$G$ is an \emph{\'etale groupoid}.

We will identify a manifold $X$ with the category with the space of
objects $X$ and only the identity morphisms. In particular, for any
\'etale category $G$ there is a canonical embedding $N_0G \to G$
which is identity on objects.

\subsubsection{Notation}\label{notational scheme}
We will make extensive use of the following notational scheme.
Suppose that $X$ is a simplicial space (such as the nerve of a
topological category) and $f \colon  [p] \to [q]$ is a morphism in
$\Delta$. The latter induced the map $X(f) \colon  X_q \to X_p$ of
spaces. Note that $f$ is determined by the number $q$ and the
sequence $\vec{f} = (f(0),\ldots,f(p-1))$. For an object $A$
associated to (or, rather, ``on'') $X_p$ (such as a function, a
sheaf, a stack) for which the inverse image under $X(f)$ is define
we will denote the resulting object on $X_q$ by $A^\s{q}_{\vec{f}}$.

\subsubsection{Sheaves on \'etale categories}
Suppose that $G$ is an \'etale category. A sheaf (of sets) on $G$ is a pair $\underline{F} = (F, F_{01})$, where
\begin{itemize}
\item $F$ is a sheaf on the space of objects $N_0G$ and

\item $F_{01} \colon F^\s{1}_1 \to F^\s{1}_0$ is an isomorphisms of sheaves on the space of morphisms $N_1G$
\end{itemize}
which is multiplicative, i.e. satisfies the ``cocycle'' condition
\[
F^\s{2}_{02} = F^\s{2}_{01}\circ F^\s{2}_{12} ,
\]
(on $N_2G$) and unit preserving, i.e.
\[
F^\s{0}_{00} = \id_F .
\]
We denote by $\Sh(G)$ the category of sheaves on $G$.

A morphism $f \colon  \underline{F} \to \underline{F}'$ of sheaves
on $\mathcal{E}$ is a morphism of sheaves $f \colon F \to F'$ on
$N_0\mathcal{E}$ which satisfies the ``equivariance'' condition
\[
f^\s{1}_0 \circ F_{01} = F_{01} \circ f^\s{1}_ .
\]

A morphism of \'etale categories $\phi \colon  G \to G'$ induces the
functor of inverse image (or pull-back)
\[
\phi^{-1} \colon  \Sh(G') \to \Sh(G) .
\]
For $\underline{F} = (F,F_{01}) \in \Sh(G')$, the sheaf on $N_0G$
underlying $\phi^{-1}\underline{F}$ is given by $(N_0\phi)^{-1}F$.
The image of $F_{01}$ under the map
\begin{multline*}
(N_1\phi)^* \colon  \Hom(F^\s{1}_1, F^\s{1}_0) \\
\to
\Hom((N_1\phi)^{-1}F^\s{1}_1, (N_1\phi)^{-1}F^\s{1}_0) \\
\cong \Hom(((N_0\phi)^{-1}F)^\s{1}_1, ((N_0\phi)^{-1}F)^\s{1}_0)
\end{multline*}

The category $\Sh(G)$ has a final object which we will denote by
$*$. For $\underline{F}\in\Sh(G)$ let $\Gamma(G;\underline{F})
\colon = \Hom_{\Sh(G)}(*, \underline{F})$. The set
$\Gamma(G;\underline{F})$ is easily identified with the subset of
$G$-invariant sections of $\Gamma(N_0G;F)$.

A sheaf $\underline{F}$ on $G$ gives rise to a cosimplicial sheaf on
$NG$ which we denote $\underline{F}_\Delta$. The latter is defined
as follows. For $n = 0, 1, 2, \ldots$ let $\underline{F}_\Delta^n =
F^\s{n}_0 \in \Sh(N_nG)$. For $f \colon  [p] \to [q]$ in $\Delta$
the corresponding structure map $f_*$ is defined as
\[
NG(f)^{-1}\underline{F}_\Delta^p = F^\s{q}_{f(0)}
\xrightarrow{F^\s{q}_{0 f(0)}} F^\s{q}_0 = \underline{F}_\Delta^q .
\]

\subsubsection{Stacks on \'etale categories}\label{subsubsection: stacks on et cat}
\begin{definition}
A stack on $G$ is a triple $\underline{\mathcal{C}} =
(\mathcal{C},\mathcal{C}_{01},\mathcal{C}_{012})$ which consists of
\begin{enumerate}
\item a stack $\mathcal{C}$ on $N_0G$
\item an \emph{equivalence} $\mathcal{C}_{01} \colon \mathcal{C}^\s{1}_1 \to \mathcal{C}^\s{1}_0$
\item an \emph{isomorphism} $\mathcal{C}_{012} \colon \mathcal{C}^\s{2}_{01}\circ\mathcal{C}^\s{2}_{12} \to
    \mathcal{C}^\s{2}_{02}$ of $1$-morphisms $\mathcal{C}^\s{2}_2 \to \mathcal{C}^\s{2}_0$.
\end{enumerate}
which satisfy
\begin{itemize}
\item $\mathcal{C}^\s{3}_{023}\circ(\mathcal{C}^\s{3}_{012}\otimes\id) =
    \mathcal{C}^\s{3}_{013}\circ(\id\otimes\mathcal{C}^\s{3}_{123})$
\item $\mathcal{C}^\s{0}_{00} = \id$
\end{itemize}
\end{definition}

Suppose that $\underline{\mathcal{C}}$ and $\underline{\mathcal{D}}$
are stacks on $G$.

\begin{definition}
A $1$-morphism $\underline{\phi} \colon \underline{\mathcal{C}} \to
\underline{\mathcal{D}}$ is a pair $\underline{\phi} = (\phi_0,
\phi_{01})$ which consists of
\begin{enumerate}
\item a $1$-morphism $\phi_0 \colon \mathcal{C} \to \mathcal{D}$ of stacks on $N_0G$
\item a $2$-isomorphism $\phi_{01} \colon \phi^\s{1}_0\circ\mathcal{C}_{01} \to \mathcal{D}_{01}\circ\phi^\s{1}_1$
    of $1$-morphisms $\mathcal{C}^\s{1}_1 \to \mathcal{D}^\s{1}_0$
\end{enumerate}
which satisfy
\begin{itemize}
\item $(\mathcal{D}_{012} \otimes \id_{\phi^\s{2}_2}) \circ (\id_{\mathcal{D}^\s{2}_{01}} \otimes
    \phi^\s{2}_{12})\circ(\phi^\s{2}_{01}\otimes \id_{\mathcal{C}^\s{2}_{12}})=\phi^\s{2}_{02} \circ
    (\id_{\phi^\s{2}_0}\otimes \mathcal{C}_{012}) $
\item $\phi^\s{0}_{00} = \id$
\end{itemize}
\end{definition}

Suppose that $\underline{\phi}$ and $\underline{\psi}$ are
$1$-morphisms $\underline{\mathcal{C}} \to \underline{\mathcal{D}}$.

\begin{definition}
A $2$-morphism $b \colon  \underline{\phi} \to \underline{\psi}$ is
a $2$-morphism $b \colon  \phi_0 \to \psi_0$ which satisfies
$\psi_{01}\circ (b^\s{1}_0\otimes\id_{\mathcal{C}_{01}}) =
(\id_{\mathcal{D}_{01}}\otimes b^\s{1}_1)\circ\phi_{01}$
\end{definition}

Suppose that $\phi ; G \to G'$ is a morphism of \'etale categories
and $\underline{\mathcal{C}} = (\mathcal{C}, \mathcal{C}_{01},
\mathcal{C}_{012})$ is a stack on $G'$. The inverse image
$\phi^{-1}\underline{\mathcal{C}}$ is the stack on $G$ given by the
triple $(\mathcal{D}, \mathcal{D}_{01}, \mathcal{D}_{012})$ with
$\mathcal{D} = (N_0\phi)^{-1}\mathcal{C}$, $\mathcal{D}_{01}$ equal
to the image of $\mathcal{C}_{01}$ under the map
\begin{multline*}
(N_1\phi)^* \colon  \Hom(\mathcal{C}^\s{1}_1, \mathcal{C}^\s{1}_0) \\
\to
\Hom((N_1\phi)^{-1}\mathcal{C}^\s{1}_1, (N_1\phi)^{-1}\mathcal{C}^\s{1}_0) \\
\cong \Hom(\mathcal{D}^\s{1}_1, \mathcal{D}^\s{1}_0)
\end{multline*}
and $\mathcal{D}_{012}$ induced by $\mathcal{C}_{012}$ in a similar
fashion.

A stack $\underline{\mathcal{C}}$ on $G$ gives rise to a
cosimplicial stack on $NG$ which we denote
$\underline{\mathcal{C}}_\Delta$. The latter is defined as follows.
For $n = 0, 1, 2, \ldots$ let $\underline{\mathcal{C}}_\Delta^n =
\mathcal{C}^\s{n}_0$. For $f \colon  [p] \to [q]$ in $\Delta$ the
corresponding structure $1$-morphism
$\underline{\mathcal{C}}_{\Delta f}$ is defined as
\[
NG(f)^{-1}\underline{\mathcal{C}}_\Delta^p =
\mathcal{C}^\s{p}_{f(0)} \xrightarrow{\mathcal{C}^\s{q}_{0 f(0)}}
\mathcal{C}^\s{q}_0 = \underline{\mathcal{C}}_\Delta^q .
\]
For a pair of morphisms $[p] \xrightarrow{f} [q] \xrightarrow{g}
[r]$ let $\underline{\mathcal{C}}_{\Delta f,g} = \mathcal{C}_{0 g(0)
g(f(0))}$.

\subsubsection{Gerbes on \'etale categories}
Suppose that $G$ is an \'etale category, $\underline{A}$ is an
abelian sheaf on $G$.

An $\underline{A}$-gerbe $\underline{\mathcal{G}}$ on $G$ is a stack
on $G$ such that
\begin{enumerate}
\item $\mathcal{G}$ is an $A^0$-gerbe on $N_0G$ ;

\item the $1$-morphism $\mathcal{C}_{01}$ is compatible with the morphism $A_{01}$.
\end{enumerate}

If $\underline{\mathcal{G}}$ is an $\underline{A}$-gerbe $G$, then
$\underline{\mathcal{G}}_\Delta$ is a cosimplicial
$\underline{A}_\Delta$-gerbe on $NG$.

\subsubsection{The category of embeddings}\label{cat of emb}
Below we recall the basics of the ``category of $G$-embeddings"
associated with an \'etale category $G$ and a basis of the topology
on the space of objects of $G$, which was introduced by I.~Moerdijk
in \cite{ieke95}.

For an \'etale category $G$ and a basis $\mathbb{B}$ for the
topology on $N_0G$ we denote by $\mathcal{E}_\mathbb{B}(G)$ or,
simply by $\mathcal{E}$, the following category.

The space of objects is given by $N_0\mathcal{E} =
\coprod_{U\in\mathbb{B}} U$, the disjoint union of the elements of
$\mathbb{B}$. Thus, the space of morphism decomposes as
\[
N_1\mathcal{E} = \coprod_{(U,V)\in\mathbb{B}\times\mathbb{B}}
(N_1\mathcal{E})_{(U,V)} ,
\]
where $(N_1\mathcal{E})_{(U,V)}$ is defined by the pull-back square
\[
\begin{CD}
(N_1\mathcal{E})_{(U,V)} @>>> N_1\mathcal{E} \\
@VVV @VV{(d_0^1, d_0^0)}V \\
U \times V @>>> N_0\mathcal{E} \times N_0\mathcal{E} .
\end{CD}
\]
(the bottom arrow being the inclusion). Now,
$(N_1\mathcal{E})_{(U,V)} \subset (N_1G)_{(U,V)}$ (the latter
defined in the same manner as the former replacing
$\mathcal{E}_\mathbb{B}(G)$ by $G$) is given by
$(N_1\mathcal{E})_{(U,V)} = \coprod_{\sigma} \sigma(U)$, where
$\sigma \colon  U \to (N_1G)_{(U,V)}$ is a section of the ``source''
projection $d_0^1 \colon  (N_1G)_{(U,V)} \to U$ such that the
composition $U \xrightarrow{\sigma} \sigma(U) \xrightarrow{d_0^0} V$
is an embedding.

With the structure (source, target, composition, ``identity'') maps
induced from those of $G$, $\mathcal{E}_\mathbb{B}(G)$ is an \'etale
category equipped with the canonical functor
\[
\lambda_\mathbb{B}(G) \colon  \mathcal{E}_\mathbb{B}(G) \to G .
\]
Note that the maps $N_i\lambda_\mathbb{B}(G) \colon
N_i\mathcal{E}_\mathbb{B}(G) \to N_iG$ are \'etale surjections.

The canonical map $i \colon N_0G \to G$ induces the map of the
respective embedding categories $\mathcal{E}_\mathbb{B}(i)\colon
\mathcal{E}_\mathbb{B}(N_0G) \to \mathcal{E}_\mathbb{B}(G)$ and the
diagram
\begin{equation}\label{comm diag of Es}
\begin{CD}
\mathcal{E}_\mathbb{B}(N_0G) @>{\mathcal{E}_\mathbb{B}(i)}>> \mathcal{E}_\mathbb{B}(G) \\
@V{\lambda_\mathbb{B}(N_0G)}VV @VV{\lambda_\mathbb{B}(G)}V \\
N_0G @>{i}>> G
\end{CD}
\end{equation}
is commutative.

%
%

\subsubsection{}\label{subsubsection: adjoint top space}
First consider the particular case when $G = X$ is a space and
$\mathbb{B}$ is a basis for the topology on $X$. For an open subset
$V\subseteq X$ let $\mathbb{B}\cap V \subseteq \mathbb{B} = \left\{
U\in\mathbb{B} \vert U\subseteq V \right\}$. There is an obvious
embedding $\mathcal{E}_{\mathbb{B}\cap V}(V) \to
\mathcal{E}_\mathbb{B}(X)$.

For $\underline{F} = (F,F_{01}) \in \Sh(\mathcal{E}_\mathbb{B}(X))$,
$V\subseteq X$ let $\lambda_!\underline{F}$ denote the presheaf on
$X$ defined by
\[
V \mapsto \Gamma(\mathcal{E}_{\mathbb{B}\cap V}(V) ; \underline{F})
= \varprojlim_{U\in\mathbb{B}\cap V} F(U) ,
\]
where $\mathbb{B}\cap V$ is partially ordered by inclusion. The
presheaf $\lambda_!\underline{F}$ is in fact a sheaf. It is
characterized by the following property:
$\lambda_!\underline{F}\vert_U = F\vert_U$ for any $U\in\mathbb{B}$.

Let $\lambda^{-1}\lambda_!\underline{F} = (H,H_{01})$. For $V$ an
open subset of $N_0\mathcal{E}_\mathbb{B}(X)$ we have
\[
H(V) = (\lambda_!\underline{F})(V) = \varprojlim_{U\in\mathbb{B}\cap
V} F(U) = F(V)
\]
naturally in $V$. We leave it to the reader to check that this
extends to an isomorphism $\lambda^{-1}\lambda_!\underline{F} =
\underline{F}$ natural in $\underline{F}$, i.e. to an isomorphism of
functors $\lambda^{-1}\lambda_! = \id$.

On the other hand, for $H\in\Sh(X)$, $V$ an open subset of $X$, put
$\lambda^{-1}H = (F, F_{01})$; we have
\[
(\lambda_!\lambda^{-1}H)(V) = \varprojlim_{U\in\mathbb{B}\cap V}
F(U) = \varprojlim_{U\in\mathbb{B}\cap V} H(U) = H(V)
\]
naturally in $V$. We leave it to the reader to check that this
extends to an isomorphism $H = \lambda_!\lambda^{-1}H$ natural in
$H$, i.e. to an isomorphism of functors $\id =
\lambda_!\lambda^{-1}$.

\subsubsection{}\label{subsubsection: general case}
We now consider the general case. To simplify notations we put
$\mathcal{E} \colon = \mathcal{E}_\mathbb{B}(G)$, $\mathcal{E}'
\colon = \mathcal{E}_\mathbb{B}(N_0G)$, $\lambda \colon =
\lambda_\mathbb{B}(G)$, $\lambda' \colon =
\lambda_\mathbb{B}(N_0G)$. Let $\underline{F} = (F, F_{01}) \in
\Sh(\mathcal{E})$ and let $\underline{F}' = (F', F'_{01}) \colon =
\mathcal{E}_\mathbb{B}(i)^{-1}\underline{F}$.

Applying the construction of \ref{subsubsection: adjoint top space}
to $\underline{F}'$ we obtain the sheaf $\lambda'_!\underline{F}'$
on $N_0G$. Note that $(N_0\lambda)^{-1}\lambda'_!\underline{F}' =
F$. The properties of the map $N_1\lambda$ imply that the pull-back
map
\begin{multline*}
(N_1\lambda)^*\colon \Hom((\lambda'_!\underline{F}')^\s{1}_1, (\lambda'_!\underline{F}')^\s{1}_0)  \\
=
\Gamma(N_1G; \shHom((\lambda'_!\underline{F}')^\s{1}_1, (\lambda'_!\underline{F}')^\s{1}_0) \\
\to \Gamma(N_1\mathcal{E};
(N_1\lambda)^{-1}\shHom((\lambda'_!\underline{F}')^\s{1}_1,
(\lambda'_!\underline{F}')^\s{1}_0)
\end{multline*}
is injective. Combining the latter with the canonical isomorphisms
\begin{eqnarray*}
(N_1\lambda)^{-1}\shHom((\lambda'_!\underline{F}')^\s{1}_1,
(\lambda'_!\underline{F}')^\s{1}_0)
 & = & \shHom((N_1\lambda)^{-1}(\lambda'_!\underline{F}')^\s{1}_1, (N_1\lambda)^{-1}(\lambda'_!\underline{F}')^\s{1}_0) \\
 & = & \shHom(((N_0\lambda)^{-1}\lambda'_!\underline{F}')^\s{1}_1, ((N_0\lambda)^{-1}\lambda'_!\underline{F}')^\s{1}_0) \\
 & = & \shHom(F^\s{1}_1, F^\s{1}_0)
\end{eqnarray*}
we obtain the injective map
\begin{equation}\label{injection of Homs}
\Hom((\lambda'_!\underline{F}')^\s{1}_1,
(\lambda'_!\underline{F}')^\s{1}_0) \to \Hom(F^\s{1}_1, F^\s{1}_0) .
\end{equation}
We leave it to the reader to verify that the map $F_{01}$ lies in
the image of \eqref{injection of Homs}; let
$(\lambda'_!\underline{F}')_{01}$ denote the corresponding element
of $\Hom(t^{-1}\lambda'_!\underline{F}',
s^{-1}\lambda'_!\underline{F}'))$. The pair
$(\lambda'_!\underline{F}', (\lambda'_!\underline{F}')_{01})$ is
easily seen to determine a sheaf on $G$ henceforth denoted
$\lambda_!\underline{F}$. The assignment $\underline{F} \mapsto
\lambda_!\underline{F}$ extends to a functor, denoted
\[
\lambda_! \colon \Sh(\mathcal{E}) \to \Sh(G) .
\]
quasi-inverse to the inverse image functor $\lambda^{-1}$. Hence,
$\lambda^{-1}(*) = *$ and, for $\underline{F} \in \Sh(G)$ the map
$\Gamma(G;\underline{F}) \to \Gamma(\mathcal{E};
\lambda^{-1}\underline{F})$ is an isomorphism. Similarly,
$\lambda_!(*) = *$ and, for $\underline{H} \in \Sh(\mathcal{E})$ the
map $\Gamma(\mathcal{E};\underline{H}) \to \Gamma(G;
\lambda_!\underline{H})$ is an isomorphism.

\subsubsection{}\label{subsubsection: equiv for ab sh}
The functors $\lambda^{-1}$ and $\lambda_!$ restrict to mutually
quasi-inverse exact equivalences of abelian categories
\[
\lambda^{-1}\colon \ShAb(G) \rightleftarrows \ShAb(\mathcal{E})
\colon  \lambda_!
\]
The morphism $\lambda^*\colon\R\Gamma(G;\underline{F}) \to
\R\Gamma(\mathcal{E};\lambda^{-1}\underline{F})$ is an isomorphism
in the derived category.

Suppose that $\underline{F}\in\ShAb(G)$ is $\mathbb{B}$-acyclic,
i.e., for any $U\in\mathbb{B}$, $i\neq 0$, $H^i(U; F) = 0$. Then,
the composition $C(\Gamma(N\mathcal{E};\lambda^{-1}\underline{F}))
\to \R\Gamma(\mathcal{E};\lambda^{-1}\underline{F})\cong
\R\Gamma(G;\underline{F})$ is an isomorphism in the derived
category.

\subsubsection{}\label{subsubsection: inv im stack equiv}
Suppose given $G$, $\mathbb{B}$ as in \ref{cat of emb}; we proceed
in the notations introduced in \ref{subsubsection: general case}.
The functor of inverse image under $\lambda$ establishes an
equivalence of (2-)categories of stack on $G$ and those on
$\mathcal{E}$. Below we sketch the construction of the quasi-inverse
along the lines of \ref{subsubsection: adjoint top space} and
\ref{subsubsection: general case}.

First consider the case $G=X$ a space. Let $\mathcal{C}$ be a stack
on $\mathcal{E}$. For an open subset $V \subseteq X$ let
\[
\lambda_!\mathcal{C}(V) \colon = \varprojlim_{U\in\mathbb{B}\cap V}
\mathcal{C}(U) ,
\]
where the latter is described in \cite{KS}, Definition 19.1.6.
Briefly, an object of $\lambda_!\mathcal{C}(V)$ is a pair
$(A,\varphi)$  which consists of a function $A\colon \mathbb{B}\cap
V\ni U\mapsto A_U\in\mathcal{C}(U)$ and a function $\varphi \colon
(U \subseteq U') \mapsto (\varphi_{UU'} \colon  A_U'\vert_U
\xrightarrow{\sim} A_U)$; the latter is required to satisfy a kind
of a cocycle condition with respect to compositions of inclusions of
basic open sets. For $(A,\varphi)$, $(A',\varphi')$ as above the
assignment $\mathbb{B}\cap V\ni U \mapsto
\Hom_\mathcal{C}(A_U,A'_U)$ extends to a presheaf on $\mathbb{B}$.
By definition,
\[
\Hom_{\lambda_!\mathcal{C}(V)}((A,\varphi),(A',\varphi')) =
\varprojlim_{U\in\mathbb{B}\cap V}\Hom_\mathcal{C}(A_U,A'_U) .
\]
The assignment $V \mapsto \lambda_!\mathcal{C}(V)$ extends to a
stack on $X$ denoted $\lambda_!\mathcal{C}$. We have natural
equivalences
\[
\lambda^{-1}\lambda_!\underline{\mathcal{D}} \cong
\underline{\mathcal{D}},\ \ \lambda_!\lambda^{-1}\mathcal{C} \cong
\mathcal{C} .
\]

Continuing with the general case, let $\underline{\mathcal{C}}$ be a
stack on $\mathcal{E}$. The stack $\lambda_!\underline{\mathcal{C}}$
on $G$ is given by the triple $(\mathcal{D}, \mathcal{D}_{01},
\mathcal{D}_{012})$ with $\mathcal{D} =
\lambda'_!\mathcal{E}_\mathbb{B}(i)^{-1}\mathcal{C}$. The morphisms
$\mathcal{D}_{01}$ and $\mathcal{D}_{012}$ are induced,
respectively, by $\mathcal{C}_{01}$ and $\mathcal{C}_{012}$. We omit
the details.


\subsection{Jet bundle}
Let $X$ be a smooth manifold. Let  $\pr_i\colon X\times X\to X$, $i = 1,2$ denote  the projection on the
$i^{\text{th}}$ factor and let $\Delta_X \colon  X\to X\times X$ denote the diagonal embedding.

Let
\[
\mathcal{I}_{X} \colon = \ker(\Delta_X^*) .
\]
The sheaf $\mathcal{I}_{X}$ plays the role of the defining ideal of the ``diagonal embedding $X \to X\times X$'': there
is a short exact sequence of sheaves on $X\times X$
\[
0 \to \mathcal{I}_{X} \to \mathcal{O}_{X\times X} \to
(\Delta_X)_*\mathcal{O}_{X} \to 0 .
\]

For a locally-free ${\mathcal O}_{X }$-module of finite rank
${\mathcal E}$ let
\begin{eqnarray*}
\mathcal{J}_{X }^k({\mathcal E}) & \colon = &
(\pr_1)_*\left({\mathcal O}_{X\times X}/{\mathcal
I}_{X}^{k+1}\otimes_{\pr_2^{-1}{\mathcal
O}_{X}}\pr_2^{-1}{\mathcal E}\right) \ , \\
\mathcal{J}^k_{X}& \colon = & \mathcal{J}_{X }^k(\mathcal{O}_{X }) \
.
\end{eqnarray*}
It is clear from the above definition that $\mathcal{J}^k_{X }$ is,
in a natural way, a commutative algebra and $\mathcal {J}_{X
}^k({\mathcal E})$ is a $\mathcal {J}^k_{X}$-module.

Let
\[
\vac^{(k)} \colon  \mathcal{O}_{X }\to \mathcal{J}^k_{X }
\]
denote the composition
\[
\mathcal{O}_{X } \xrightarrow{\pr_1^*} (\pr_1)_*\mathcal{O}_{X\times
X  } \to \mathcal{J}^k_{X }
\]
In what follows, unless stated explicitly otherwise, we regard
$\mathcal{J}_{X }^k({\mathcal E})$ as a $\mathcal {O}_{X}$-module
via the map $\vac^{(k)}$.

Let
\[
j^k\colon  \mathcal{E} \to \mathcal{J}_{X }^k(\mathcal{E})
\]
denote the composition
\[
\mathcal{E} \xrightarrow{e\mapsto 1\otimes e}
(\pr_1)_*\mathcal{O}_{X\times X  } \otimes_\mathbb{C} \mathcal{E}
\to \mathcal{J}_{X }^k(\mathcal{E})
\]
The map $j^k$ is not $\mathcal{O}_{X }$-linear unless $k=0$.

For $0\leq k\leq l$ the inclusion ${\mathcal I}_{X
}^{l+1}\to{\mathcal I}_{X}^{k+1}$ induces the surjective map
$\pi_{l,k}\colon {\mathcal J}^l_{X }({\mathcal E}) \to {\mathcal
J}^k_{X }({\mathcal E})$. The sheaves ${\mathcal J}^k_{X }({\mathcal
E})$, $k=0,1,\ldots$ together with the maps $\pi_{l,k}$, $k\leq l$
form an inverse system. Let ${\mathcal J}_{X }({\mathcal E}) =
{\mathcal J}^\infty_{X }({\mathcal E})\colon =
\underset{\longleftarrow}{\lim}{\mathcal J}^k_{X }({\mathcal E})$.
Thus, ${\mathcal J}_{X }({\mathcal E})$ carries a natural topology.

The maps $\vac^{(k)}$ (respectively, $j^k$), $k=0,1,2,\ldots$ are
compatible with the projections $\pi_{l,k}$, i.e.
$\pi_{l,k}\circ\vac^{(l)} = \vac^{(k)}$ (respectively,
$\pi_{l,k}\circ j^l = j^k$). Let $\vac \colon =
\underset{\longleftarrow}{\lim} \vac^{(k)}$,
 $j^\infty\colon =\underset{\longleftarrow}{\lim}j^k$.

Let
\begin{multline*}
d_1 \colon  {\mathcal O}_{{X\times X} }\otimes_{\pr_2^{-1}{\mathcal O}_{X}}\pr_2^{-1}{\mathcal E} \to \\
\to \pr_1^{-1}\Omega^1_{X}\otimes_{\pr_1^{-1}{\mathcal
O}_{X}}{\mathcal O}_{{X\times X} }\otimes_{\pr_2^{-1}{\mathcal
O}_{X}}\pr_2^{-1}{\mathcal E}
\end{multline*}
denote the exterior derivative along the first factor. It satisfies
\begin{multline*}
d_1({\mathcal I}_{X}^{k+1}\otimes_{\pr_2^{-1}{\mathcal
O}_{X}}\pr_2^{-1}{\mathcal E})\subset
\\
\pr_1^{-1}\Omega^1_X\otimes_{\pr_1^{-1}{\mathcal O}_{X}}{\mathcal
I}_{X}^k\otimes_{\pr_2^{-1}{\mathcal O}_{X}}\pr_2^{-1}{\mathcal E}
\end{multline*}
for each $k$ and, therefore, induces the map
\[
d_1^{(k)} \colon  {\mathcal J}^k({\mathcal
E})\to\Omega^1_{X}\otimes_{{\mathcal O}_{X}}{\mathcal
J}^{k-1}({\mathcal E})
\]
The maps $d_1^{(k)}$ for different values of $k$ are compatible with
the maps $\pi_{l,k}$ giving rise to the \emph{canonical flat
connection}
\[
\nabla^{can}_{\mathcal E} \colon  {\mathcal J}_{X}({\mathcal
E})\to\Omega^1_{X}\otimes_{{\mathcal O}_{X}}{\mathcal
J}_{X}({\mathcal E}) \ .
\]

We will also use the following notations:
\begin{eqnarray*}
\mathcal{J}_{X} & \colon = &
\mathcal{J}_{X}(\mathcal{O}_{X }) \\
\overline{\mathcal{J}}_{X} & \colon = & \mathcal{J}_{X}/\vac(\mathcal{O}_X) \\
\mathcal{J}_{X,0} & \colon = & \ker(\mathcal{J}_{X}(\mathcal{O}_{X
}) \to \mathcal{O}_X)
\end{eqnarray*}

The canonical flat connection extends to the flat connection
\[
\nabla^{can}_\mathcal{E} \colon  \mathcal{J}_{X}(\mathcal{E}) \to
\Omega^1_X\otimes_{\mathcal{O}_X} \mathcal{J}_{X}(\mathcal{E}).
\]

Here and below by abuse of notation we write $(.)\otimes_{{\mathcal O}_{X}}{\mathcal J}_{X}({\mathcal E})$ for
$\underset{\longleftarrow}{\lim}(.)\otimes_{{\mathcal O}_{X}}{\mathcal J}^k({\mathcal E})$.

\subsubsection{De Rham complexes}\label{De Rham cxs}
Suppose that $\mathcal{F}$ is an $\mathcal{O}_X$-module and $\nabla
\colon  \mathcal{F} \to
\Omega^1_X\otimes_{\mathcal{O}_X}\mathcal{F}$ is a flat connection.
The flat connection $\nabla$ extends uniquely to a differential
$\nabla$ on $\Omega^\bullet_X\otimes_{\mathcal{O}_X}\mathcal{F}$
subject to the Leibniz rule with respect to the
$\Omega^\bullet_X$-module structure. We will make use of the
following notation :
\[
(\Omega^i_X\otimes_{\mathcal{O}_X}\mathcal{F})^{cl} \colon =
\ker(\Omega^i_X\otimes_{\mathcal{O}_X}\mathcal{F}
\xrightarrow{\nabla}
\Omega^{i+1}_X\otimes_{\mathcal{O}_X}\mathcal{F})
\]

Suppose that $(\mathcal{F}^\bullet, d)$ is a complex of
$\mathcal{O}_X$-modules with a flat connection $\nabla =
(\nabla^i)_{i\in\mathbb{Z}}$, i.e. for each $i\in\mathbb{Z}$,
$\nabla^i$ is a flat connection on $\mathcal{F}^i$ and $[d,\nabla] =
0$. Then,
$(\Omega^\bullet_X\otimes_{\mathcal{O}_X}\mathcal{F}^\bullet,
\nabla, \id\otimes d)$ is a double complex. We denote by
$\DR(\mathcal{F})$ the associated simple complex.


\subsection{Characteristic classes of cosimplicial $\mathcal{O}^\times$-gerbes}
\label{subsection: gerbes} In this section we consider an \'etale
simplicial manifold $X$, i.e. a simplicial manifold $X \colon
[p]\mapsto X_p$ such that
\begin{multline}\label{condition: etale}
\textrm{for each morphism $f\colon [p]\to[q]$ in $\Delta$,} \\
\textrm{the induced map $X(f)\colon X_q\to X_p$ is \emph{\'etale}}
\end{multline}
As a consequence, the collection of sheaves $\mathcal{O}_{X_p}$
(respectively, $\mathcal{J}_{X_p}$, etc.) form a \emph{special} (see
Definition \ref{def special sheaf}) cosimplicial sheaf on $X$ which
will be denoted $\mathcal{O}_X$ (respectively, $\mathcal{J}_X$,
etc.)

The goal of this section is to associate to a cosimplicial
$\mathcal{O}^\times$-gerbe $\mathcal S$ on $X$ a cohomology class
$[\mathcal{S}]\in
H^2(\real{X};\DR(\overline{\mathcal{J}}_{\real{X}}))$, where
$\real{X}$ is the subdivision of $X$ (see \ref{geometric
realization}).

\subsubsection{Gerbes on manifolds}\label{subsubsection: char class}
Suppose that $Y$ is a manifold. We begin by sketching a construction
which associated to an $\mathcal{O}_Y^\times$-gerbe $\mathcal S$ on
$Y$ a characteristic class in
$H^2(Y;\DR(\mathcal{J}_Y/\mathcal{O}_Y))$ which lifts the more
familiar de Rham characteristic class $[\mathcal{S}]_{dR}\in
H^3_{dR}(Y)$.

The map $\vac \colon  \mathcal{O}_Y \to \mathcal{J}_Y$ of sheaves of
rings induces the map of sheaves of abelian groups $\vac \colon
\mathcal{O}^\times_Y \to \mathcal{J}^\times_Y$. Let
$\overline{(\,.\,)} \colon \mathcal{J}_Y^{(\times)}\to
\mathcal{J}_Y^{(\times)}/\mathcal{O}_Y^{(\times)}$ denote the
projection.

Suppose that $\mathcal{S}$ is an $\mathcal{O}^\times_Y$-gerbe. The
composition
\[
\mathcal{O}^\times_Y \xrightarrow{\vac} \mathcal{J}^\times_Y
\xrightarrow{\nabla^{can}\log} (\Omega^1_Y\otimes
\mathcal{J}_Y)^{cl}
\]
gives rise to the $\Omega^1_Y\otimes \mathcal{J}_Y$-gerbe
$\nabla^{can}\log\vac\mathcal{S}$. Let
\[
\eth \colon  (\Omega^1_Y\otimes\mathcal{J}_Y)^{cl}[1] \to
\nabla^{can}\log\vac\mathcal{S}
\]
be a trivialization of the latter.

Since $\nabla^{can}\circ\nabla^{can}=0$, the $(\Omega^2_Y\otimes
\mathcal{J}_Y)^{cl}$-gerbe
$\nabla^{can}\nabla^{can}\log\vac\mathcal{S}$ is canonically
trivialized. Therefore, the trivialization $\nabla^{can}\eth$ of
$\nabla^{can}\log\vac\mathcal{S}$ induced by $\eth$ may (and will)
be considered as a $(\Omega^2_Y\otimes\mathcal{J}_Y)^{cl}$-torsor.
Let
\[
B\colon  (\Omega^2_Y\otimes\mathcal{J}_Y)^{cl} \to \nabla^{can}\eth
\]
be a trivialization of the $\nabla^{can}\eth$.

Since $\nabla^{can}\log\circ\vac = \vac\circ d\log$, it follows that
the $(\Omega^1_Y\otimes\overline{\mathcal{J}}_Y)^{cl}$-gerbe
$\overline{\nabla^{can}\log\vac\mathcal{S}}$ is canonically
trivialized. Therefore, its trivialization $\overline{\eth}$ may
(and will) be considered as a
$(\Omega^1_Y\otimes\overline{\mathcal{J}}_Y)^{cl}$-torsor. Moreover,
since $\nabla^{can}\nabla^{can}=0$ the
$(\Omega^2_Y\otimes\overline{\mathcal{J}}_{Y})^{cl}$-torsor
$\nabla^{can}\overline{\eth}\cong \overline{\nabla^{can}\eth}$ is
canonically trivialized, the trivialization induced by $B$ is a
section $\overline{B}$ of
$(\Omega^2_Y\otimes\overline{\mathcal{J}}_Y)^{cl}$, i.e.
$\overline{B}$ is a cocycle in
$\Gamma(X;\DR(\overline{\mathcal{J}}_Y))$.

One can show that the class of $\overline{B}$ in
$H^2(X;\DR(\overline{\mathcal{J}}_Y))$
\begin{enumerate}
\item depends only on $\mathcal{S}$ and not on the auxiliary choices of $\eth$ and $B$ made
\item coincides with the image of the class of $\mathcal{S}$ under the map $H^2(Y;\mathcal{O}^\times_Y)\to
    H^2(X;\DR(\overline{\mathcal{J}}_{Y}))$ induced by the composition
\[
\mathcal{O}_Y^\times \xrightarrow{\overline{(\, . \,)}}
\mathcal{O}_Y^\times/\mathbb{C}^\times \xrightarrow{\log}
\mathcal{O}_Y/\mathbb{C} \xrightarrow{j^\infty}
\DR(\overline{\mathcal{J}}_Y)
\]
\end{enumerate}

On the other hand, $H = \nabla^{can}B$ is a trivialization of the
canonically trivialized (by $\nabla^{can}\circ\nabla^{can}=0$)
$(\Omega^3_Y\otimes\mathcal{J}_Y)^{cl}$-torsor
$\nabla^{can}\nabla^{can}\eth$, hence a section of
$(\Omega^3_Y\otimes\mathcal{J}_Y)^{cl}$ which, clearly, satisfies
$\overline{H}=0$, i.e. is a closed $3$-form. Moreover, as is clear
from the construction, the class of $H$ in $H^3_{dR}(Y)$ is the
image of the class of $\overline{B}$ under the boundary map
$H^2(X;\DR(\overline{\mathcal{J}}_Y))\to H^3_{dR}(Y)$.

Below we present a generalization of the above construction to
\'etale simplicial manifolds.

\subsubsection{}\label{dlog-gerbe}
Suppose that $\mathcal{G}$ is a cosimplicial
$(\Omega^1_X\otimes\mathcal{J}_X)^{cl}$-gerbe on $X$. An example of
such is $\nabla^{can}\log\vac(\mathcal{S})$, where $\mathcal{S}$ is
an $\mathcal{O}_X^\times$-gerbe.

Consider the $(\Omega^2_X\otimes\mathcal{J}_X)^{cl}$-gerbe
$\nabla^{can}\mathcal{G}$. Since the composition
\[
(\Omega^1_{X_p}\otimes\mathcal{J}_{X_p})^{cl}\hookrightarrow
\Omega^1_{X_p}\otimes\mathcal{J}_{X_p}\xrightarrow{\nabla^{can}}
(\Omega^2_{X_p}\otimes\mathcal{J}_{X_p})^{cl}
\]
is equal to zero, $\nabla^{can}\mathcal{G}^p$ is canonically
trivialized for all $p=0,1,2,\ldots$. Therefore, for a morphism
$f\colon [p]\to[q]$, $\nabla^{can}\mathcal{G}_f$, being a morphism
of trivialized gerbes, may (and will) be regarded as a
$(\Omega^2_{X_q}\otimes\mathcal{J}_{X_q})^{cl}$-torsor.

Assume given, for each $p=0,1,2,\ldots$
\begin{itemize}
\item[(\texttt{i})] a choice of trivialization
\[
\eth^p \colon  (\Omega^1_{X_p}\otimes\mathcal{J}_{X_p})^{cl}
\to\mathcal{G}^p \ ;
\]
it induces the trivialization
\[
\nabla^{can}\eth^p \colon
(\Omega^2_{X_p}\otimes\mathcal{J}_{X_p})^{cl}[1]
\to\nabla^{can}\mathcal{G}^p \ .
\]
Since $\nabla^{can}\mathcal{G}^p$ is canonically trivialized,
$\nabla^{can}\eth^p$ may (and will) be regarded as a
$(\Omega^2_{X_p}\otimes\mathcal{J}_{X_p})^{cl}$-torsor.

\item[(\texttt{i\!i})] a choice of trivialization
\[
B^p \colon  (\Omega^2_{X_p}\otimes\mathcal{J}_{X_p})^{cl}
\to\nabla^{can}\eth^p \ .
\]
\end{itemize}

For a morphism $f\colon [p]\to[q]$ in $\Delta$ the trivialization
$\eth^p$ induces the trivialization
\[
X(f)^{-1}\eth^p \colon (\Omega^1_{X_q}\otimes\mathcal{J}_{X_q})^{cl}
= X(f)^{-1}(\Omega^1_{X_p}\otimes\mathcal{J}_{X_p})^{cl} \to
X(f)^{-1}\mathcal{G}^p \ ;
\]
thus, $\mathcal{G}_f$ is a morphism of trivialized gerbes, i.e. a
$(\Omega^1_{X_q}\otimes\mathcal{J}_{X_q})^{cl}$-torsor.

Assume given for each morphism $f\colon [p]\to[q]$
\begin{itemize}
\item[(\texttt{i\!i\!i})] a choice of trivialization $\beta_f \colon
    (\Omega^1_{X_q}\otimes\mathcal{J}_{X_q})^{cl}\to\mathcal{G}_f$.
\end{itemize}

For a pair of composable arrows
$[p]\stackrel{f}{\to}[q]\stackrel{g}{\to}[r]$ the discrepancy
\[
\mathcal{G}_{f,g} \colon = (\beta_g + X(g)^{-1}\beta_f) -
\beta_{g\circ f}
\]
is global section of
$(\Omega^1_{X_r}\otimes\mathcal{J}_{X_r})^{cl}$. Since the map
$\mathcal{J}_{X_r,0} \xrightarrow{\nabla^{can}}
(\Omega^1_{X_r}\otimes\mathcal{J}_{X_r})^{cl}$ is an isomorphism
there is a unique section $\beta_{f,g}$ of $\mathcal{J}_{X_r,0}$
such that
\begin{equation}\label{definition of beta in gerbes}
\nabla^{can}\beta_{f,g} = \mathcal{G}_{f,g} \ .
\end{equation}

\begin{lemma}\label{lemma:beta is a cocycle}
For any triple of composable arrows
$[p]\stackrel{f}{\to}[q]\stackrel{g}{\to}[r]\stackrel{h}{\to}[s]$
the relation $h_*\beta_{f,g}=\beta_{g,h}-\beta_{gf,h}+\beta_{f,hg}$
holds.
\end{lemma}
\begin{proof}
A direct calculation shows that
\[
h_*\mathcal{G}_{f,g}=\mathcal{G}_{g,h}-\mathcal{G}_{gf,h}+\mathcal{G}_{f,hg}
\ .
\]
Therefore (\eqref{definition of beta in gerbes}),
\[
\nabla^{can}(h_*\beta_{f,g} -
(\beta_{g,h}-\beta_{gf,h}+\beta_{f,hg})) = 0 \ .
\]
The map $\mathcal{J}_{X_r,0} \xrightarrow{\nabla^{can}}
(\Omega^1_{X_r}\otimes\mathcal{J}_{X_r})^{cl}$ is an isomorphism,
hence the claim follows.
\end{proof}

\subsubsection{}
We proceed in the notations introduced above. Let
\[
\pr_X \colon \Delta^q\times X_p\to X_p
\]
denote the projection. Let
\begin{equation}\label{pull-back of forms}
\pr_X^\dagger \colon
\pr_X^{-1}(\Omega^1_{X_p}\otimes_{\mathcal{O}_{X_p}}\mathcal{J}_{X_p})
\to
\Omega^1_{X_p\times\Delta^q}\otimes_{\pr_X^{-1}\mathcal{O}_{X_p}}\pr_X^{-1}\mathcal{J}_{X_p}
\end{equation}
denote the canonical map. For a
$(\Omega^1_{X_p}\otimes_{\mathcal{O}_{X_p}}\mathcal{J}_{X_p})$-gerbe
$\mathcal{G}$, let
\[
\pr_X^*\mathcal{G} = (\pr_X^\dagger)_*(\pr_X^{-1}\mathcal{G})
\]
Let
\[
\widetilde{\nabla}^{can} = \pr_X^*\nabla^{can} = d_\Delta\otimes\id
+ \id\otimes\nabla^{can} \ .
\]

Since $\sum_{i=0}^q t_i=1$ the composition
\begin{multline*}
\Omega^1_{X_p\times\Delta^q}\otimes_{\pr_X^{-1}\mathcal{O}_{X_p}}
\pr_X^{-1}\mathcal{J}_{X_p}\to \\
(\Omega^1_{X_p\times\Delta^q}\otimes_{\pr_X^{-1}\mathcal{O}_{X_p}}
\pr_X^{-1}\mathcal{J}_{X_p})^{\times(p+1)} \to \\
\Omega^1_{X_p\times\Delta^q}\otimes_{\pr_X^{-1}\mathcal{O}_{X_p}}\pr_X^{-1}\mathcal{J}_{X_p}
\end{multline*}
of the diagonal map with
$(\alpha_0,\ldots,\alpha_p)\mapsto\sum_{i=0}^q t_i\cdot\alpha_i$ is
the identity map, it follows that the composition
\begin{multline}\label{pull-back diag sum}
\pr_X^{-1}(\Omega^1_{X_p}\otimes_{\mathcal{O}_{X_p}}\mathcal{J}_{X_p}) \xrightarrow{\pr_X^\dagger} \\
\Omega^1_{X_p\times\Delta^q}\otimes_{\pr_X^{-1}\mathcal{O}_{X_p}}
\pr_X^{-1}\mathcal{J}_{X_p}\to
(\Omega^1_{X_p\times\Delta^q}\otimes_{\pr_X^{-1}\mathcal{O}_{X_p}}
\pr_X^{-1}\mathcal{J}_{X_p})^{\times(p+1)} \\ \to
\Omega^1_{X_p\times\Delta^q}\otimes_{\pr_X^{-1}\mathcal{O}_{X_p}}\pr_X^{-1}\mathcal{J}_{X_p}
\end{multline}
is equal to the map $\pr_X^\dagger$. Since, by definition, the
$(\Omega^1_{X_p\times\Delta^q}\otimes_{\pr_X^{-1}\mathcal{O}_{X_p}}\pr_X^{-1}\mathcal{J}_{X_p})$-gerbe
$\sum_{i=0}^p t_i\cdot\pr_X^*\mathcal{G}^p$ is obtained from
$\pr_X^{-1}\mathcal{G}^p$ via the ``change of structure group''
along the composition \eqref{pull-back diag sum}, it is canonically
equivalent to $\pr_X^*\mathcal{G}^p$.

Consider a simplex $\lambda\colon [n]\to\Delta$. Let
\[
\eth^\lambda \colon = \sum_{i=0}^n t_i\cdot\pr_X^*
\mathcal{G}_{\lambda(in)}(X(\lambda(in))^{-1}\eth^{\lambda(i)});
\]
thus $\eth^\lambda$ is a trivialization of
$\pr_X^*\mathcal{G}^{\lambda(n)}$. Since
$\lambda(in)_*\eth^{\lambda(i)}=\eth^{\lambda(n)} -
\mathcal{G}_{\lambda(in)}$
\[
\eth^\lambda = \pr_X^*\eth^{\lambda(n)} - \sum_{i=0}^n
t_i\cdot\pr_X^*\mathcal{G}_{\lambda(in)}
\]
Therefore,
\begin{multline*}
\widetilde{\nabla}^{can}\eth^\lambda = \sum_{i=0}^n
t_i\cdot\pr_X^*\lambda(in)_* \nabla^{can}\eth^{\lambda(i)} -
\sum_{i=0}^n dt_i\wedge\pr_X^*\mathcal{G}_{\lambda(in)} .
\end{multline*}
Let $B^\lambda$ denote the trivialization of
$\widetilde{\nabla}^{can}\eth^\lambda$ given by
\begin{multline}\label{B-lambda}
B^\lambda = \sum_{i=0}^n t_i\cdot\pr_X^*\lambda(in)_*B^{\lambda(i)}
- \sum_{i=0}^n dt_i\wedge\pr_X^*\beta_{\lambda(in)} -
\\
\widetilde{\nabla}^{can}\left(\sum_{0\leq i<j\leq
n}(t_idt_j-t_jdt_i)\wedge\pr_X^*\beta_{\lambda;ij}\right)
\end{multline}
where $\beta_{\lambda;ij} = \beta_{\lambda((ij)),\lambda((jn))}$.

\subsubsection{} Suppose that $\mu \colon  [m]\to\Delta$ is another simplex
and a morphism $\phi\colon [m]\to[n]$ such that
$\mu=\lambda\circ\phi$, i.e. $\phi$ is a morphism of simplices
$\mu\to\lambda$. Let $f \colon  \mu(m)\to\lambda(n)$ denote the map
$\lambda(\phi(m)\to n)$. The map $f$ induces the maps
\[
\begin{CD}
\Delta^m\times X_{\mu(m)} @<{\id\times X(f)}<< \Delta^m\times
X_{\lambda(n)} @>{f\times\id}>> \Delta^n\times X_{\lambda(n)}
\end{CD}
\]

\begin{prop}\label{prop: comb compat}
In the notations introduced above
\begin{enumerate}
\item $(\id\times X(f))^*\eth^\mu = (\phi\times\id)^*\eth^\lambda$
\item $(\id\times X(f))^*B^\mu = (\phi\times\id)^*B^\lambda$
\end{enumerate}
\end{prop}
\begin{proof}
For $\phi\colon [m]\to[n]$ the induced map $\phi^* \colon
\Omega^\bullet_{\Delta^n}\to\Omega^\bullet_{\Delta^m}$ is given by
$\phi^*(t_j) = \sum_{\phi(i)=j}t_i$.

\begin{multline*}
(\id\times X(f))^*\eth^\mu = \\ (\id\times X(f))^*\sum_{i=0}^m t_i\cdot\pr_{\mu(m)}^*\mu((im))_*\eth^{\mu(i)} = \\
\sum_{i=0}^m
t_i\cdot\pr_{\lambda(n)}^*f_*\mu((im))_*\eth^{\mu(i)} = \\
\sum_{j=0}^n\sum_{\phi(i)=j}t_i\cdot
\pr_{\lambda(n)}^*\lambda((jn))_*\eth^{\lambda(j)} =
(\phi\times\id)^*\eth^\lambda
\end{multline*}

Therefore, $(\id\times X(f))^*\widetilde{\nabla}^{can}\eth^\mu =
(\phi\times\id)^*\widetilde{\nabla}^{can}\eth^\lambda$.

By Lemma \ref{lemma: restriction of forms on simplex}, it suffices
to verify the second claim for $\phi\colon [1]\to[n]$. Let
$k=\lambda(\phi(0))$, $l=\lambda(\phi(1))$, $p=\lambda(n)$. With
these notations $f=(lp)$ and the left hand side reduces to
\[
t_0\cdot\pr_p^*(kp)_*B^k + t_1\cdot\pr_p^*(lp)_*B^l +
dt_0\wedge\pr_p^*(lp)_*\beta_{(kl)}
\]
while the right hand side reads
\begin{multline*}
t_0\cdot\pr_p^*(kp)_*B^k + t_1\cdot\pr_p^*(lp)_*B^l + \\
dt_0\wedge\pr_p^*\beta_{(kp)} + dt_1\wedge\pr_p^*\beta_{(lp)} + \\
(t_1dt_0-t_0dt_1)\wedge\pr_p^*\alpha_{(kl),(lp)}
\end{multline*}
Using $t_1dt_0-t_0dt_1 = dt_0 = -dt_1$ and the definition of
$\alpha_{(kl),(lp)}$ one sees that the two expressions are indeed
equal.
\end{proof}

\begin{lemma}\label{lemma: restriction of forms on simplex}
The map
\[
\Omega^{\leq 1}_n \to \prod_{\Hom_\Delta([1],[n])} \Omega_1
\]
induced by the maps $\Delta((ij))^* \colon  \Omega^{\leq 1}_n \to
\Omega_1$, $i\leq j$, is injective on the subspace of form with
coefficients of degree at most one.
\end{lemma}
\begin{proof}
Left to the reader.
\end{proof}

\subsubsection{}\label{1-forms-gerbe}
Suppose that $\mathcal{F}$ is a cosimplicial $\Omega^{1,cl}_X$-gerbe
on $X$.

The composition
\[
\mathcal{O}_X \xrightarrow{\vac} \mathcal{J}_X
\xrightarrow{\nabla^{can}} \Omega^1_X\otimes\mathcal{J}_X
\]
is a derivation, hence factors canonically through the map denoted
\[
\vac \colon  \Omega^1_X \to \Omega^1_X\otimes\mathcal{J}_X
\]
which maps closed forms to closed forms.

Put $\mathcal{G} = \vac(\mathcal{F})$ in \ref{dlog-gerbe}. Since the
composition
\[
\Omega^{1,cl}_{X_p} \to
(\Omega^1_{X_p}\otimes\mathcal{J}_{X_p})^{cl} \to
(\Omega^1_{X_p}\otimes\overline{\mathcal{J}}_{X_p})^{cl}
\]
is equal to the zero map, the
$(\Omega^1_{X_p}\otimes\overline{\mathcal{J}}_{X_p})^{cl}$-gerbe
$\overline{\mathcal{G}^p}$ is canonically trivialized for each $p$.
Therefore, $\overline{\eth^p}$ may (and will) be regarded as a
$(\Omega^1_{X_p}\otimes\overline{\mathcal{J}}_{X_p})^{cl}$-torsor.
The
$(\Omega^2_{X_p}\otimes\overline{\mathcal{J}}_{X_p})^{cl}$-torsor
$\nabla^{can}\overline{\eth^p}$ is canonically trivialized,
therefore, $\overline{B^p}$ may (and will) be regarded as a section
of $(\Omega^2_{X_p}\otimes\overline{\mathcal{J}}_{X_p})^{cl}$.

Since, for a morphism $f\colon [p]\to[q]$, $\mathcal{G}_f =
\vac(\mathcal{F}_f)$, it follows that $\overline{\mathcal{G}_f}$ is
canonically trivialized, hence $\overline{\beta_f}$ may (and will)
be regarded as a section of
$(\Omega^1_{X_p}\otimes\overline{\mathcal{J}}_{X_p})^{cl}$.

It follows that, for a simplex $\lambda\colon [n]\to\Delta$, the
formula \eqref{B-lambda} gives rise to a section
$\overline{B^\lambda}$ of $\Omega^2_{\Delta^n\times
X_{\lambda(n)}}\otimes
\pr_X^*\overline{\mathcal{J}}_{X_{\lambda(n)}}$ which clearly
satisfies $\widetilde{\nabla}^{can}\overline{B^\lambda} = 0$.

\begin{prop}\label{prop:class `1-forms-gerbe}
In the notations of \ref{1-forms-gerbe}
\begin{enumerate}
\item the assignment $\overline{B}\colon \lambda\mapsto\overline{B}^\lambda$ defines a cycle in the complex
    $\Tot(\Gamma(\real{X};\DR(\overline{\mathcal{J}}_{\real{X}}))$

\item The class of $\overline{B}$ in $H^2(\Tot(\Gamma(\real{X};\DR(\overline{\mathcal{J}}_{\real{X}})))$ coincides
    with the image of the class $[\mathcal{F}] \in H^2(X;\Omega^{1,cl}_X)$ of the gerbe $\mathcal{F}$ under the
    composition
\[
H^2(X;\Omega^{1,cl}_X) \to H^2(\real{X};\real{\Omega^{1,cl}_X}) \to
H^2(\Tot(\Gamma(\real{X};\DR(\overline{\mathcal{J}}_{\real{X}})))
\]
where the first map is the canonical isomorphism (see \ref{sheaves
on subdivision}) and the second map is induced by the map
$\Omega^{1,cl}_X \cong \mathcal{O}_X/\mathbb{C}
\xrightarrow{j^\infty} \overline{\mathcal{J}}_X$. In particular, the
class of $\overline{B}$ does not depend on the choices made.
\end{enumerate}
\end{prop}

\subsubsection{}\label{subsubsection: char class trivialized gerbe}
In the rest of the section we will assume that $\mathcal{S}$ is a
cosimplicial $\mathcal{O}_X^\times$-gerbe on $X$ such that the
gerbes ${\mathcal S}^p$ are trivial for all $p$, i.e. ${\mathcal
S}^p = \mathcal{O}_X^\times[1]$. Our present goal is to obtain
simplified expressions for $B$ and $\overline{B}$ in this case.

Since a morphism of trivial $\mathcal{O}^\times$-gerbes is an
$\mathcal{O}^\times$-torsor (equivalently, a locally free
$\mathcal{O}$-module of rank one), a cosimplicial gerbe
$\mathcal{S}$ as above is given by the following collection of data:
\begin{enumerate}
\item for each morphism $f\colon [p]\to[q]$ in $\Delta$ a line bundle ${\mathcal S}_f$ on $X_q$,

\item for each pair of morphisms $[p]\stackrel{f}{\longrightarrow}[q] \stackrel{g}{\longrightarrow}[r]$ an
    isomorphism ${\mathcal S}_{f,g}\colon  {\mathcal S}_g\otimes X(g)^*({\mathcal S}_f)\to{\mathcal S}_{g\circ f}$
    of line bundles on $X_r$
\end{enumerate}
These are subject to the associativity condition: for a triple of
composable arrows $[p]\stackrel{f}{\longrightarrow}[q]
\stackrel{g}{\longrightarrow}[r]\stackrel{h}{\longrightarrow}[s]$
\[
\mathcal{S}_{g\circ f, h}\otimes X(h)^{-1}\mathcal{S}_{f,g} =
\mathcal{S}_{f, h\circ g}\otimes\mathcal{S}_{g,h}
\]

In order to calculate the characteristic class of $\mathcal{S}$ we
will follow the method (and notations) of \ref{dlog-gerbe} and
\ref{1-forms-gerbe} with $\mathcal{F} = d\log(\mathcal{S})$ and the
following choices:
\begin{enumerate}
\item $\eth^p$ is the canonical isomorphism $\Omega^1_{X_p}\otimes\mathcal{J}_{X_p}[1] =
    \nabla^{can}\log\vac(\mathcal{O}^\times_{X_p})[1]$, i.e. is given by the trivial torsor
    $\Omega^1_{X_p}\otimes\mathcal{J}_{X_p} = \nabla^{can}\log\vac(\mathcal{O}^\times_{X_p})$;

\item $B^p$ is the canonical isomorphism $\Omega^2_{X_p}\otimes\mathcal{J}_{X_p} =
    \nabla^{can}(\Omega^1_{X_p}\otimes\mathcal{J}_{X_p})$.
\end{enumerate}
Then, $\mathcal{G}_f = \nabla^{can}\log\vac(\mathcal{S}_f) = \vac
d\log(\mathcal{S}_f)$ and $B_f$ is equal to the canonical
trivialization of $\nabla^{can}\mathcal{G}_f =
\nabla^{can}\nabla^{can}\log\vac({\mathcal S}_f)$ (stemming from
$\nabla^{can}\circ\nabla^{can}=0$).

For each $f \colon  [p]\to [q]$ we choose
\begin{enumerate}
\item[(3)] a $\mathcal{J}_{X_q}$-linear isomorphism $\sigma_f \colon \mathcal{S}_f\otimes\mathcal{J}_{X_q}\to
    \mathcal{J}_{X_q}(\mathcal{S}_f)$ which reduced to the identity modulo $\mathcal{J}_{X_q,0}$

\item[(4)] a trivialization of $d\log(\mathcal{S}_f)$, i.e. a connection $\nabla_f$ on the line bundle
    $\mathcal{S}_f$
\end{enumerate}
Let $\beta_f =
\sigma_f^{-1}\circ\nabla^{can}_{\mathcal{S}_f}\circ\sigma_f$; thus,
$\beta_f$ is a trivialization of $\mathcal{G}_f$. The choice of
$\nabla_f$ determines another trivialization of $\mathcal{G}_f$,
namely $\vac(\nabla_f) = \nabla_f\otimes\id +
\id\otimes\nabla^{can}_{\mathcal{O}}$.

Let $F_f \colon = \beta_f -
\vac(\nabla_f)\in\Gamma(X_q;\Omega^1_{X_q}\otimes\mathcal{J}_{X_q})$.
Flatness of $\nabla^{can}_{\mathcal{S}_f}$ implies that $F_f$
satisfies $\nabla^{can}_{\mathcal{O}}F_f + \vac(c(\nabla_f)) = 0$
which implies that $\nabla^{can}\overline F_f = 0$.

For a pair of composable arrows
$[p]\stackrel{f}{\to}[q]\stackrel{g}{\to}[r]$ we have
\begin{equation}\label{equation on beta in gerbes}
\nabla^{can}\overline\beta_{f,g} = \overline F_g + X(g)^*\overline
F_f - \overline F_{g\circ f} \ .
\end{equation}

With these notations, for $\lambda \colon  [n]\to\Delta$, we have:
\begin{multline*}
B^\lambda = -\sum_i dt_i\wedge(\vac(\nabla_{\lambda(in)}) +
F_{\lambda(in)}) - \\
\widetilde\nabla^{can}\left(\sum_{0\leq i<j\leq n} (t_idt_j -
t_jdt_i)\wedge\pr^*_X\beta_{\lambda;ij}\right)
\end{multline*}
and
\begin{multline}\label{B for gerbes}
\overline{B}^\lambda = -\sum_i dt_i\wedge\overline F_{\lambda(in)} - \\
\widetilde\nabla^{can}\left(\sum_{0\leq i<j\leq n} (t_idt_j -
t_jdt_i)\wedge\pr^*_X\overline\beta_{\lambda;ij}\right) \ .
\end{multline}

\begin{prop}
In the notations of \ref{subsubsection: char class trivialized
gerbe}
\begin{enumerate}
\item the assignment $\overline{B}\colon \lambda\mapsto\overline{B}^\lambda$ defines a cycle in the complex
    $\Tot(\Gamma(\real{X};\DR(\overline{\mathcal{J}}_{\real{X}}))$

\item The class of $\overline{B}$ in $H^2(\Tot(\Gamma(\real{X};\DR(\overline{\mathcal{J}}_{\real{X}})))$ coincides
    with the image of the class $[\mathcal{S}] \in H^2(X;\mathcal{O}^\times_X)$ of the gerbe $\mathcal{S}$ under
    the composition
\[
H^2(X;\mathcal{O}^\times_X) \to
H^2(\real{X};\real{\mathcal{O}^\times_X}) \to
H^2(\Tot(\Gamma(\real{X};\DR(\overline{\mathcal{J}}_{\real{X}})))
\]
where the first map is the canonical isomorphism (see \ref{sheaves
on subdivision}) and the second map is induced by the map
$\mathcal{O}^\times_X \to \mathcal{O}^\times_X/\mathbb{C}^\times
\xrightarrow{\log} \mathcal{O}_X/\mathbb{C} \xrightarrow{j^\infty}
\overline{\mathcal{J}}_X$. In particular, the class of
$\overline{B}$ does not depend on the choices made.
\end{enumerate}
\end{prop}

\section{Deformations and DGLA} \label{Deformations and DGLA}

\subsection{Deligne $2$-groupoid}

\subsubsection{Deligne $2$-groupoid}
In this subsection we review the construction of Deligne
$2$-groupoid of a nilpotent differential graded algebra (DGLA). We
follow \cite{G, G1} and the references therein.

Suppose that $\mathfrak{g}$ is a nilpotent DGLA such that
$\mathfrak{g}^i = 0$ for $i< -1$.

A Maurer-Cartan element of $\mathfrak{g}$  is an element $\gamma \in
\mathfrak{g}^1$ satisfying
\begin{equation} \label{eq:MC}
d\gamma + \frac{1}{2}[\gamma,\gamma]=0.
\end{equation}
We denote by $\MC^2(\mathfrak{g})_0$ the set of Maurer-Cartan
elements of $\mathfrak{g}$.

 The unipotent group $\exp
\mathfrak{g}^0$ acts on the set of Maurer-Cartan elements of
$\mathfrak{g}$ by gauge equivalences. This action is given by the
formula
\begin{equation*}
(\exp X) \cdot \gamma= \gamma- \sum_{i=0}^{\infty} \frac{(\ad
X)^i}{(i+1)!}(dX +[\gamma, X])
\end{equation*}
If $\exp X$  is a gauge equivalence between two Maurer-Cartan
elements $\gamma_1 $ and $\gamma_2= (\exp X)\cdot \gamma_1$ then
\begin{equation} \label{eq:MC equivalence}
d+\ad \gamma_2=  \Ad \exp X\,(d+\ad \gamma_1).
\end{equation}
We denote by $\MC^2(\mathfrak{g})_1(\gamma_1, \gamma_2)$ the set of
gauge equivalences between $\gamma_1$, $\gamma_2$. The composition
\begin{equation*}
\MC^2(\mathfrak{g})_1(\gamma_2,
\gamma_3)\times\MC^2(\mathfrak{g})_1(\gamma_1,
\gamma_2)\to\MC^2(\mathfrak{g})_1(\gamma_1, \gamma_3)
\end{equation*}
is given by  the product in the group $\exp \mathfrak{g}^0$.

If $\gamma \in \MC^2(\mathfrak{g})_0$  we can define a Lie bracket
$[\cdot, \cdot]_{\gamma}$ on $\mathfrak{g}^{-1}$ by
\begin{equation}\label{eq:mu-bracket}
[a,\,b]_{\gamma}=[a,\, d b+[\gamma, \,b]].
\end{equation}
With this bracket $\mathfrak{g}^{-1}$ becomes a nilpotent Lie
algebra. We denote by $\exp_{\gamma} \mathfrak{g}^{-1}$ the
corresponding unipotent group, and by $\exp_{\gamma}$ the
corresponding exponential map $\mathfrak{g}^{-1} \to \exp_{\gamma}
\mathfrak{g}^{-1}$. If $\gamma_1$, $\gamma_2$ are two Maurer-Cartan
elements, then the group $\exp_{\gamma} \mathfrak{g}^{-1}$ acts on
$\MC^2(\mathfrak{g})_1(\gamma_1, \gamma_2)$. Let $\exp_{\gamma} t
\in \exp_{\gamma} \mathfrak{g}^{-1}$ and let $\exp X \in
\MC^2(\mathfrak{g})_1(\gamma_1, \gamma_2)$. Then
\begin{equation*}
(\exp_{\gamma} t) \cdot (\exp X) = {\operatorname
{exp}}(dt+[\gamma,t])\, {\operatorname {exp}}X \in \exp
\mathfrak{g}^0
\end{equation*}
Such an element $\exp_{\gamma} t$ is called a $2$-morphism between
$\exp X$ and $(\exp t) \cdot (\exp X)$. We denote by
$\MC^2(\mathfrak{g})_2(\exp X,\exp Y)$ the set of $2$-morphisms
between $\exp X$ and $\exp Y$. This set is endowed with a vertical
composition given by the product in the group $\exp_{\gamma}
\mathfrak{g}^{-1}$.

 Let $\gamma_1$, $\gamma_2$, $\gamma_3 \in
\MC^2(\mathfrak{g})_0$. Let $\exp X_{12}$, $\exp Y_{12}\in
\MC^2(\mathfrak{g})_1(\gamma_1, \gamma_2)$ and $\exp X_{23}$, $\exp
Y_{23}\in \MC^2(\mathfrak{g})_1(\gamma_2, \gamma_3)$. Then one
defines the horizontal composition
\begin{multline*}
\otimes\colon  \MC^2(\mathfrak{g})_2(\exp X_{23}, \exp Y_{23})
\times
\MC^2(\mathfrak{g})_2(\exp X_{12}, \exp Y_{12}) \to\\
\MC^2(\mathfrak{g})_2(\exp X_{23}\exp X_{12}, \exp X_{23}\exp
Y_{12})
\end{multline*}
as follows. Let $\exp_{\gamma_2} t_{12} \in
\MC^2(\mathfrak{g})_2(\exp X_{12}, \exp Y_{12})$ and let
$\exp_{\gamma_3} t_{23} \in \MC^2(\mathfrak{g})_2(\exp X_{23}, \exp
Y_{23})$. Then
\begin{equation*}
\exp_{\gamma_{3}} t_{23} \otimes \exp_{\gamma_{2}} t_{12}
=\exp_{\gamma_{3}} t_{23}\exp_{\gamma_3}( e^{\ad X_{23}}(t_{12}) )
\end{equation*}

To summarize, the data described above forms a $2$-groupoid which we
denote by $\MC^2(\mathfrak{g})$ as follows:
\begin{enumerate}
\item the set of objects is $\MC^2(\mathfrak{g})_0$

\item the groupoid of morphisms $\MC^2(\mathfrak{g})(\gamma_1,\gamma_2)$, $\gamma_i\in\MC^2(\mathfrak{g})_0$
    consists of:
\begin{itemize}
\item objects i.e. $1$-morphisms in $\MC^2(\mathfrak{g})$ are given by
    $\MC^2(\mathfrak{g})_1(\gamma_1,\gamma_2) $  -- the gauge transformations between $\gamma_1$ and
    $\gamma_2$.
\item morphisms between $\exp X$, $\exp Y \in \MC^2(\mathfrak{g})_1(\gamma_1,\gamma_2)$ are given by
    $\MC^2(\mathfrak{g})_2(\exp X,\exp Y)$.
\end{itemize}
\end{enumerate}

 A morphism of nilpotent DGLA $ \phi \colon
\mathfrak{g} \to \mathfrak{h}$ induces a functor $\phi\colon
\MC^2(\mathfrak{g}) \to \MC^2(\mathfrak{g})$.

We have the following important result (\cite{GM}, \cite{G} and
references therein).
\begin{thm}\label{thm: quism invariance of mc}
Suppose that $\phi \colon  \mathfrak{g}\to\mathfrak{h}$ is a
quasi-isomorphism of DGLA and  let $\mathfrak{m}$ be a nilpotent
commutative ring. Then the induced map $\phi \colon
\MC^2(\mathfrak{g}\otimes \mathfrak{m})\to\MC^2(\mathfrak{h}\otimes
\mathfrak{m})$ is an equivalence of $2$-groupoids.
\end{thm}

Suppose now that  $\mathfrak{G}$: $[n]\to \mathfrak{G}^n$ is a
cosimplicial DGLA. We assume that each $\mathfrak{G}^n$ is a
nilpotent DGLA. We denote its component of degree $i$ by
$\mathfrak{G}^{n, i}$ and assume that $\mathfrak{G}^{n, i}=0$ for
$i<-1$. Then the  morphism of complexes
\begin{equation}\label{ker to Tot dgla}
\ker(\mathfrak{G}^0\rightrightarrows\mathfrak{G}^1) \to
\Tot(\mathfrak{G})
\end{equation}
(cf. \eqref{ker to Tot cxs}) is a morphism of DGLA.

\begin{prop}\label{prop: MC2 ker is MC2 tot}
Assume that $\mathfrak{G}$ satisfies the following condition:
\begin{equation}\label{acyclicity condition}
\text{for all $i\in\mathbb{Z}$, $H^p(\mathfrak{G}^{\bullet,i})=0$
for $p\neq 0$}.
\end{equation}
Then the morphism of $2$-groupoids
\[
\MC^2(\ker(\mathfrak{G}^0\rightrightarrows\mathfrak{G}^1)) \to
\MC^2(\Tot(\mathfrak{G}))
\]
induced by \eqref{ker to Tot dgla} is an equivalence.
\end{prop}
\begin{proof}
Follows from Lemma \ref{lemma: ker to Tot cxs} and Theorem \ref{thm:
quism invariance of mc}
\end{proof}

\subsection{Deformations and the Deligne $2$-groupoid}\label{ddgla}

\subsubsection{Pseudo-tensor categories}\label{subsubsection:pseudo-tensor cats}
In order to give a unified treatment of the deformation complex we
will employ the formalism of pseudo-tensor categories. We refer the
reader to \cite{BD} for details.

Let $X$ be a topological space. The category $\Sh_k(X)$ has a
canonical structure of a pseudo-tensor $\Sh_k(X)$-category. In
particular, for any finite set $I$, an $I$-family of $\{L_i\}_{i\in
I}$ of sheaves and a sheaf $L$ we have the \emph{sheaf}
\[
P^{\Sh_k(X)}_I(\{L_i\},L) = \shHom_k(\otimes_{i\in I}L_i,L)
\]
of $I$-operations.

In what follows we consider not necessarily full pseudo-tensor
$\Sh_k(X)$-subcategories $\Psi$ of $\Sh_k(X)$. Given such a category
$\Psi$, with the notations as above, we have the subsheaf
$P^\Psi_I(\{L_i\},L)$ of $P^{\Sh_k(X)}_I(\{L_i\},L)$.

We shall always assume that the pseudo-tensor category $\Psi$ under
consideration satisfies the following additional assumptions:
\begin{enumerate}
\item For any object $L$ of $\Psi$ and any finite dimensional $k$-vector space $V$ the sheaf $L\otimes_kV$ is in
    $\Psi$;

\item for any $k$-linear map of finite dimensional vector spaces $f \colon  V \to W$ the map $\id\otimes f$ belongs
    to $\Gamma(X;P^\Psi_\mathbf{1}((L\otimes_kV)^\mathbf{1}, L\otimes_kW))$.
\end{enumerate}


\subsubsection{Examples of pseudo-tensor categories}\label{subsubsection:examples of pseudo-tensor cats}
\begin{enumerate}
\item[\texttt{DIFF}] Suppose that $X$ is a manifold. Let $\mathtt{DIFF}$ denote the following pseudo-tensor
    category. The objects of $\mathtt{DIFF}$ are locally free modules of finite rank over $\mathcal{O}_X$. In the
    notations introduced above, $P^\mathtt{DIFF}_I(\{L_i\},L)$ is defined to be the sheaf of multidifferential
    operators $\bigotimes_{i\in I}L_i \to L$ (the tensor product is over $\mathbb{C}$).

\item[\texttt{JET}] For $X$ as in the previous example, let $\mathtt{JET}$ denote the pseudo-tensor category whose
    objects are locally free modules of finite rank over $\mathcal{J}_X$. In the notations introduced above,
    $P^\mathtt{JET}_I(\{L_i\},L)$ is defined to be the sheaf of continuous $\mathcal{O}_X$-(multi)linear maps
    $\bigotimes_{i\in I}L_i \to L$ (the tensor product is over $\mathcal{O}_X$).

\item[\texttt{DEF}] For $\Psi$ as in \ref{subsubsection:pseudo-tensor cats} and an Artin $k$-algebra $R$ let
    $\widetilde{\Psi(R)}$ denote the following pseudo-tensor category. An object of $\widetilde{\Psi(R)}$ is an
    $R$-module in $\Psi$ (i.e. an object $M$ of $\Psi$ together with a morphism of $k$-algebras $R \to \Gamma(X;
    P^\Psi_\vac(M^\vac,M))$) which \emph{locally on $X$} is $R$-linearly isomorphic in $\Psi$ to an object of the
    form $L\otimes_k R$, where $L$ is an object in $\Psi$. In the notations introduced above, the sheaf
    $P^{\widetilde{\Psi(R)}}_I(\{M_i\},M)$ of $I$-operations is defined as the subsheaf of $R$-multilinear maps in
    $P^\Psi_I(\{M_i\},M)$.

Let $\Psi(R)$ denote the full subcategory of $\widetilde{\Psi(R)}$
whose objects are isomorphic to objects of the form $L\otimes_k R$,
$L$ in $\Psi$.

Note that a morphism $R \to S$ of Artin $k$-algebras induces the
functor $(\, .\,)\otimes_R S \colon \widetilde{\Psi(R)} \to
\widetilde{\Psi(S)}$ which restricts to the functor $(\,
.\,)\otimes_R S \colon  \Psi(R) \to \Psi(S)$. It is clear that the
assignment $R \mapsto \widetilde{\Psi(R)}$ (respectively, $R \mapsto
\Psi(R)$) defines a functor on $\ArtAlg_k$ and the inclusion
$\Psi(R)(\, .\, ) \to \widetilde{\Psi(\, .\, )}$ is a morphism
thereof.
\end{enumerate}

\subsubsection{Hochschild cochains}
For $n = 1, 2, \ldots$ we denote by $\mathbf{n}$ the set
$\{1,2,\ldots,n\}$. For an object $\mathcal{A}$ of $\Psi$ we denote
by $\mathcal{A}^\mathbf{n}$ the $\mathbf{n}$-collection of objects
of $\Psi$ each of which is equal to $\mathcal{A}$ and set
\[
C^n(\mathcal{A}) =
P^\Psi_\mathbf{n}(\mathcal{A}^\mathbf{n},\mathcal{A}) \ ,
\]
and $C^0(\mathcal{A}) \colon = \mathcal{A}$. The sheaf
$C^n(\mathcal{A})$ is called the sheaf of Hochschild cochains on
$\mathcal{A}$ of degree $n$.

The graded sheaf of vector spaces $\mathfrak{g}(\mathcal{A}) \colon
= C^{\bullet}(\mathcal{A})[1]$ has a canonical structure of a graded
Lie algebra under the Gerstenhaber bracket denoted by $[\ ,\ ]$
below. Namely, $C^{\bullet}(\mathcal{A})[1]$ is canonically
isomorphic to the (graded) Lie algebra of derivations of the free
associative co-algebra generated by $\mathcal{A}[1]$.

For an operation
$m\in\Gamma(X;P^\Psi_\mathbf{2}(\mathcal{A}^\mathbf{2},\mathcal{A}))
= \Gamma(X;C^2(\mathcal{A})) = \Gamma(X;\mathfrak{g}(\mathcal{A}))$
the associativity of $m$ is equivalent to the condition $[m,m] = 0$.

Suppose that $\mathcal A$ is an associative algebra with the product
$m$ as above. Let $\delta = [m, .]$. Thus, $\delta$ is a derivation
of the graded Lie algebra $\mathfrak{g}(\mathcal{A})$ of degree one.
The associativity of $m$ implies that $\delta\circ\delta = 0$, i.e.
$\delta$ defines a differential on $\mathfrak{g}(\mathcal{A})$
called the Hochschild differential.

The algebra is called \emph{unital} if it is endowed with a global
section $1\in\Gamma(X;\mathcal{A})$ with the usual properties with
respect to the product $m$. For a unital algebra the subsheaf of
\emph{normalized} cochains (of degree $n$)
$\overline{C}^n(\mathcal{A})$ of $C^n(\mathcal{A})$ is defined as
the subsheaf of Hochschild cochains which vanish whenever evaluated
on $1$ as one the arguments for $n>0$; by definition,
$\overline{C}^0(\mathcal{A}) = C^0(\mathcal{A})$.

The graded subsheaf $\overline{C}^{\bullet}(\mathcal{A})[1]$ is
closed under the Gerstenhaber bracket and the action of the
Hochschild differential, and the inclusion
\[
\overline{C}^{\bullet}(\mathcal{A})[1]\hookrightarrow
C^\bullet(\mathcal{A})[1]
\]
is a quasi-isomorphism of DGLA.

Suppose in addition that $R$ is a commutative Artin $k$-algebra with
the (nilpotent) maximal ideal $\mathfrak{m}_R$. Then,
$\Gamma(X;\mathfrak{g}(A)\otimes_k\mathfrak{m}_R)$ is a nilpotent
DGLA concentrated in degree greater than or equal to $-1$.
Therefore, the Deligne $2$-groupoid
$\MC^2(\mathfrak{g}(A)\otimes_k\mathfrak{m}_R)$ is defined.
Moreover, it is clear that the assignment
\[
R\mapsto
\MC^2(\Gamma(X;\mathfrak{g}(\mathcal{A})\otimes_k\mathfrak{m}_R)
\]
extends to a functor on the category of commutative Artin algebras.

\subsubsection{$\mathtt{DIFF}$ and $\mathtt{JET}$}\label{diff and jet}
Unless otherwise stated, from now on a locally free module over
$\mathcal{O}_X$ (respectively, $\mathcal{J}_X$) of finite rank is
understood as an object of the pseudo-tensor category
$\mathtt{DIFF}$ (respectively, $\mathtt{JET}$) defined in
\ref{subsubsection:examples of pseudo-tensor cats}.

Suppose that $\mathcal{A}$ is an $\mathcal{O}_X$-Azumaya algebra,
i.e. a sheaf of (central) $\mathcal{O}_X$-algebras locally
isomorphic to $\Mat_n(\mathcal{O}_X)$. The canonical flat connection
$\nabla^{can}_{\mathcal A}$ on $\mathcal{J}_X(\mathcal{A})$ induces
the flat connection, still denoted $\nabla^{can}_\mathcal{A}$, on
$C^n(\mathcal{J}_X(\mathcal{A}))$.

The flat connection $\nabla^{can}_\mathcal{A}$ acts by derivation of
the Gerstenhaber bracket which commute with the Hochschild
differential $\delta$. Hence, we have the DGLA
$\Omega^\bullet_X\otimes_{\mathcal{O}_X}
C^\bullet(\mathcal{J}_{X}(\mathcal{A}))[1]$ with the differential
$\nabla^{can}_\mathcal{A} + \delta$.

\begin{lemma}
The de Rham complex $\DR(C^n(\mathcal{J}_X(\mathcal{A}))) \colon =
(\Omega^\bullet_X \otimes_{\mathcal{O}_X}
C^n(\mathcal{J}_X(\mathcal{A})), \nabla^{can}_{\mathcal A})$
satisfies
\begin{enumerate}
\item $H^i\DR(C^n(\mathcal{J}_X(\mathcal{A}))) = 0$ for $i\neq 0$
\item The map $j^\infty \colon  C^n(\mathcal{A})\to C^n(\mathcal{J}_X(\mathcal{A}))$ is an isomorphism onto
    $H^0\DR(C^n(\mathcal{J}_X(\mathcal{A})))$.
\end{enumerate}
\end{lemma}
\begin{cor}
The map
\[
j^\infty\colon C^\bullet(\mathcal{A})[1] \to
\Omega^\bullet_X\otimes_{{\mathcal O}_X}
C^\bullet(\mathcal{J}_X(\mathcal{A}))[1]
\]
is a quasi-isomorphism of DGLA.
\end{cor}

\subsubsection{Star products}\label{subsubsection: star products}
Suppose that $\mathcal{A}$ is an object of $\Psi$ with an
associative multiplication $m$ and the unit $1$ as above.

We denote by $\ArtAlg_k$ the category of (finitely generated,
commutative, unital) Artin $k$-algebras. Recall that an Artin
$k$-algebra is a local $k$-algebra $R$ with the maximal ideal
$\mathfrak{m}_R$ which is nilpotent, i.e. there exists a positive
integer $N$ such that $\mathfrak{m}_R^N = 0$; in particular, an
Artin $k$-algebra is a finite dimensional $k$-vector space. There is
a canonical isomorphism $R/\mathfrak{m}_R\cong k$. A morphism of
Artin algebras $\phi \colon  R \to S$ is a $k$-algebra homomorphism
which satisfies $\phi(\mathfrak{m}_R) \subseteq \mathfrak{m}_S$.

\begin{definition}\label{def:star product}
For $R\in\ArtAlg_k$ an \emph{$R$-star product} on $\mathcal{A}$ is
an $R$-bilinear operation
$m'\in\Gamma(X;P^\Psi_\mathbf{2}(\mathcal{A}^\mathbf{2},\mathcal{A}\otimes_kR))$
which is associative and whose image in
$\Gamma(X;P^\Psi_\mathbf{2}(\mathcal{A}^\mathbf{2},\mathcal{A}))$
(under composition with the canonical map $\mathcal{A}\otimes_kR \to
\mathcal{A}$) coincides with $m$.
\end{definition}

An $R$-star product $m'$ as in \ref{def:star product} determines an
$R$-bilinear operation $m'\in
P^\Psi_\mathbf{2}((\mathcal{A}\otimes_kR)^\mathbf{2},
\mathcal{A}\otimes_kR)$ which endows $\mathcal{A}\otimes_kR$ with a
structure of a unital associative $R$-algebra.

The $2$-category of $R$-star products on $\mathcal{A}$, denoted
$\Def(\mathcal{A})(R)$, is defined as follows:
\begin{itemize}
\item objects are the $R$-star products on $\mathcal{A}$,

\item a $1$-morphism $\phi \colon  m_1 \to m_2$ between the $R$-star products $m_i$ is an operation $\phi\in
    P^\Psi_\mathbf{1}(\mathcal{A}^\mathbf{1}, \mathcal{A}\otimes_kR)$ whose image in $\phi\in
    P^\Psi_\mathbf{1}(\mathcal{A}^\mathbf{1}, \mathcal{A})$ (under the composition with the canonical map
    $\mathcal{A}\otimes_kR \to \mathcal{A}$) is equal to the identity, and whose $R$-linear extension $\phi\in
    P^\Psi_\mathbf{1}(\mathcal{A}^\mathbf{1}, \mathcal{A}\otimes_kR)$ is a morphism of $R$-algebras.

\item a $2$-morphism $b \colon  \phi \to \psi$, where $\phi, \psi \colon  m_1 \to m_2$ are two $1$-morphisms, are
    elements $b \in 1\otimes 1 + \Gamma(X;\mathcal{A}\otimes_k \mathfrak{m}_R) \subset
    \Gamma(X;\mathcal{A}\otimes_k R)$ such that $m_2(\phi(a),b) = m_2(b,\psi(a))$ for all $a\in
    \mathcal{A}\otimes_k R$.
\end{itemize}
It follows easily from the above definition and the nilpotency of
$\mathfrak{m}_R$ that $\Def(\mathcal{A})(R)$ is a $2$-groupoid.

Note that $\Def(\mathcal{A})(R)$ is non-empty: it contains the
trivial deformation, i.e. the star product, still denoted $m$, which
is the $R$-bilinear extension of the product on $\mathcal{A}$.

It is clear that the assignment $R\mapsto \Def(\mathcal{A})(R)$
extends to a functor on $\ArtAlg_k$.

\subsubsection{Star products and the Deligne $2$-groupoid}
We continue in notations introduced above. In particular, we are
considering a sheaf of associative $k$-algebras $\mathcal A$ with
the product $m\in \Gamma(X; C^2(\mathcal{A}))$. The product $m$
determines the element, still denoted $m$ in
$\Gamma(X;\mathfrak{g}^1(\mathcal{A})\otimes_kR)$ for any
commutative Artin $k$-algebra $R$, hence, the Hochschild
differential $\delta \colon = [m,\ ]$ in
$\mathfrak{g}(\mathcal{A})\otimes_kR$.

Suppose that $m'$ is an $R$-star product on $\mathcal A$. Since
$\mu(m') \colon = m' - m = 0 \mod\mathfrak{m}_R$ we have
$\mu(m')\in\mathfrak{g}^1(\mathcal{A})\otimes_k\mathfrak{m}_R$.
Moreover, the associativity of $m'$ implies that $\mu(m')$ satisfies
the Maurer-Cartan equation, i.e. $\mu(m')\in\MC^2(\Gamma(X;
\mathfrak{g}(\mathcal{A})\otimes_k \mathfrak{m}_R))_0$.

It is easy to see that the assignment $m'\mapsto\mu(m')$ extends to
a functor
\begin{equation}\label{functor Def to MC}
\Def(\mathcal{A})(R)\to \MC^2(\Gamma(X;
\mathfrak{g}(\mathcal{A})\otimes_k \mathfrak{m}_R)) \ .
\end{equation}

The following proposition is well-known (cf. \cite{Ge, G, G1}).

\begin{prop}\label{Def equivalent to MC}
The functor \eqref{functor Def to MC} is an isomorphism of
$2$-groupoids.
\end{prop}

\subsection{Matrix Azumaya algebras}

\subsubsection{Azumaya algebras}
Suppose that $K$ is a sheaf of commutative $\mathbb{C}$-algebras.

An \emph{Azumaya $K$-algebra on $X$} is a sheaf of central
$K$-algebras which is a twisted form of (i.e. locally isomorphic to)
$\Mat_n(K)$ for suitable $n$. In our applications the algebra $K$
will be either $\mathcal{O}_X$ of $\mathcal{J}_X$.

There is a central extension of Lie algebras
\begin{equation}\label{ses:ad}
0\to K \to \mathcal{A} \xrightarrow{\delta} \Der_K(\mathcal{A}) \to
0
\end{equation}
where the map $\delta$ is given by $\delta(a)\colon b\mapsto[a,b]$.

For an $\mathcal{O}_X$-Azumaya algebra we denote by
$\mathcal{C}(\mathcal{A})$ the
$\Omega^1_X\otimes_{\mathcal{O}_X}\Der_{\mathcal{O}_X}(\mathcal{A})$-torsor
of (locally defined) connections $\nabla$ on $\mathcal{A}$ which
satisfy the Leibniz rule $\nabla(ab)=\nabla(a)b+a\nabla(b)$.

For the sake of brevity we will refer to Azumaya
$\mathcal{O}_X$-algebras simply as Azumaya algebras (on $X$).

\subsubsection{Splittings}\label{gerbe of splittings}
Suppose that $\mathcal{A}$ is an Azumaya algebra.

\begin{definition}
A \emph{splitting} of $\mathcal A$ is a pair $({\mathcal E},\phi)$
consisting of a vector bundle $\mathcal{E}$ and an isomorphism
$\phi\colon  \mathcal{A}\to\shEnd_{{\mathcal O}_X}({\mathcal E})$.

A morphism $f\colon (\mathcal{E}_1,\phi_1)\to(\mathcal{E}_2,\phi_2)$
of splittings is an isomorphism $f\colon
\mathcal{E}_1\to\mathcal{E}_2$ such that $\Ad(f)\circ\phi_2 =
\phi_1$.
\end{definition}

Let $S({\mathcal A})$ denote the stack such that, for $U\subseteq
X$, $S({\mathcal A})(U)$ is the category of splittings of
$\mathcal{A}\vert_U$.

The sheaf of automorphisms of any object is canonically isomorphic
$\mathcal{O}^\times_X$ (the subgroup of central units). As is easy
to see, $S({\mathcal A})$ is an $\mathcal{O}^\times_X$-gerbe.

Suppose that ${\mathcal A}$ and ${\mathcal B}$ are Azumaya algebras
on $X$ and $F$ is an $\mathcal{O}_X$-linear equivalence of
respective categories of modules.

\begin{lemma}
If $({\mathcal E},\phi)$ is a splitting of ${\mathcal A}$, then the
${\mathcal B}$-module $F({\mathcal E},\phi)$ is a splitting of
${\mathcal B}$.
\end{lemma}
\begin{cor}
The induced functor $S(F)\colon S({\mathcal A})\to S({\mathcal B})$
is an equivalence.
\end{cor}

In fact, it is clear that $S(.)$ extends to functor from the
$2$-category of Azumaya algebras on $X$ to the $2$-category of
$\mathcal{O}^\times_X$-gerbes.

\subsubsection{Matrix algebras}
Until further notice we work in a fixed pseudo-tensor subcategory
$\Psi$ of $\Sh_k(X)$ as in \ref{subsubsection:pseudo-tensor cats}.
In particular, all algebraic structures are understood to be given
by operations in $\Psi$.

Suppose that $K$ is a sheaf of commutative algebras on $X$.

\begin{definition}\label{def: matrix alg}
A \emph{matrix $K$-algebra} is a sheaf of associative $K$-algebras
$\mathcal{A}$ (on $X$) together with a decomposition $\mathcal{A} =
\sum_{i,j = 0}^p\mathcal{A}_{ij}$ into a \emph{direct} sum of
$K$-submodules which satisfies
\begin{enumerate}
\item for $0\leq i,j,k\leq p$ the product on $\mathcal{A}$ restricts to a map
    $\mathcal{A}_{ij}\otimes_K\mathcal{A}_{jk} \to \mathcal{A}_{ik}$ in
    $\Gamma(X;P^\Psi(\{\mathcal{A}_{ij},\mathcal{A}_{jk}\},\mathcal{A}_{ik}))$

\item for $0\leq i,j\leq p$ the left (respectively, right) action of $\mathcal{A}_{ii}$ (respectively,
    $\mathcal{A}_{jj}$) on $\mathcal{A}_{ij}$ given by the restriction of the product is unital.
\end{enumerate}
\end{definition}

Note that, for $0\leq i\leq p$, the composition $K \xrightarrow{1}
\mathcal{A} \to \mathcal{A}_{ii}$ (together with the restriction of
the product) endows the sheaf $\mathcal{A}_{ii}$ with a structure of
an associative algebra. The second condition in Definition \ref{def:
matrix alg} says that $\mathcal{A}_{ij}$ is a unital
$\mathcal{A}_{ii}$-module (respectively,
$\mathcal{A}_{jj}^{op}$-module).

For a matrix algebra as above we denote by
$\mathfrak{d}(\mathcal{A})$ the subalgebra of ``diagonal" matrices,
i.e. $\mathfrak{d}(\mathcal{A}) = \sum_{i=0}^p\mathcal{A}_{ii}$.

Suppose that $\mathcal{A} = \sum_{i,j=0}^p\mathcal{A}_{ij}$ and
$\mathcal{B} = \sum_{i,j=0}^p\mathcal{B}_{ij}$ are two matrix
$K$-algebras (of the same size).

\begin{definition}
A $1$-morphism $F \colon  \mathcal{A} \to \mathcal{B}$ of matrix
algebras is a morphism of sheaves of $K$-algebras in
$\Gamma(X;P^\Psi(\{\mathcal{A}\},\mathcal{B}))$ such that
$F(\mathcal{A}_{ij})\subseteq \mathcal{B}_{ij}$.
\end{definition}

Suppose that $F_1,\ F_2 \colon  \mathcal{A} \to \mathcal{B}$ are
$1$-morphisms of matrix algebras.

\begin{definition}
A $2$-morphism $b \colon  F_1 \to F_2$ is a section
$b\in\Gamma(X;\mathfrak{d}(B))$ such that $b\cdot F_1 = F_2\cdot b$.
\end{definition}

In what follows we will assume that a matrix algebra satisfies the
following condition:

for $0\leq i,j\leq p$ the sheaf $\mathcal{A}_{ij}$ is a locally free
module of rank one over $\mathcal{A}_{ii}$ and over
$\mathcal{A}_{jj}^{op}$.

\subsubsection{Combinatorial restriction for matrix algebras}
Suppose that $\mathcal{A}=\sum_{i,j=0}^q{\mathcal A}_{ij}$ is a
matrix $K$-algebra and $f\colon [p] \to [q]$ is a morphism in
$\Delta$.

\begin{definition}
The \emph{the combinatorial restriction $f^{\sharp}\mathcal{A}$ (of
$\mathcal{A}$ along $f$)} is the matrix algebra with the underlying
sheaf
\[
f^{\sharp}\mathcal{A} = \bigoplus_{i,j=0}^p(f^{\sharp}{\mathcal
A})_{ij},\ \ (f^{\sharp}{\mathcal A})_{ij}=\mathcal{A}_{f(i)f(j)}\ .
\]
The product on $f^{\sharp}\mathcal{A}$ is induced by the product on
$\mathcal{A}$.
\end{definition}

Suppose that $F \colon  \mathcal{A} \to \mathcal{B}$ is a
$1$-morphism of matrix algebras. The combinatorial restriction
$f^\sharp F \colon  f^{\sharp}\mathcal{A} \to f^{\sharp}\mathcal{B}$
is defined in an obvious manner. For a $2$-morphism $b \colon  F_1
\to F_2$ the combinatorial restriction $f^\sharp b\colon  f^\sharp
F_1 \to f^\sharp F_2$ is given by $(f^\sharp b)_{ii} =
b_{f(i)f(i)}$.

For $0\leq i\leq q$, $0\leq j\leq p$ let $({\mathcal A}_f)_{ij} =
{\mathcal A}_{if(j)}$. Let
\[
{\mathcal A}_f = \bigoplus_{i=0}^q\bigoplus_{j=0}^p ({\mathcal
A}_f)_{ij}
\]
The sheaf ${\mathcal A}_f$ is endowed with a structure of a
${\mathcal A}\otimes_K (f^\sharp{\mathcal A})^{op}$-module given by
\begin{equation}
(a b c)_{il}= \sum_{k=0}^q\sum_{j=0}^p   a_{ij}b_{j k}c_{kl}
\end{equation}
where $a=\sum_{i,j=0}^q a_{ij}\in{\mathcal A}$,
$b=\sum_{i=0}^q\sum_{j=0}^p b_{ij}\in{\mathcal A}_f$,
$c=\sum_{i,j=0}^p c_{ij}\in f^\sharp{\mathcal A}$ and $abc \in
\mathcal{A}_f$.

The $f^\sharp{\mathcal A}\otimes_K{\mathcal A}^{op}$-module
${\mathcal A}^{-1}_f$ is defined in a similar fashion, with
$({\mathcal A}^{-1}_f)_{ij}={\mathcal A}_{f(i)j}$, $0\leq i\leq p$,
$0\leq j\leq q$.

Let $\overrightarrow{\alpha}_f\colon  \mathcal{A}_f
\otimes_{f^{\sharp}\mathcal{A}} \mathcal{A}^{-1}_f \to \mathcal{A}$
be defined by
\begin{equation}
\alpha (b \otimes  c)_{ij} =\sum_{k=0}^p b_{ik} c_{kj}
\end{equation}
Where $b=\sum_{i=0}^q\sum_{j=0}^p b_{ij}\in\mathcal{A}_f$,
$c=\sum_{i=0}^p\sum_{j=0}^q c_{ij}\in{\mathcal A}^{-1}_f$. Similarly
one constructs an isomorphism $\overleftarrow{\alpha}_f \colon
\mathcal{A}_f^{-1} \otimes_{\mathcal{A}} \mathcal{A}_f \to
f^{\sharp}\mathcal{A}$ of $f^{\sharp}\mathcal{A}$ bimodules.

\begin{lemma} The bimodules ${\mathcal A}_f$ and ${\mathcal A}^{-1}_f$ together with the maps $\overrightarrow{\alpha}_f$ and $\overleftarrow{\alpha}_f$
implement a Morita equivalence between $f^{\sharp}\mathcal{A}$ and
$\mathcal{A}$.
\end{lemma}

\subsubsection{Matrix Azumaya algebras}
\begin{definition}\label{def: mat azum K-alg}
A \emph{matrix Azumaya $K$-algebra on $X$} is a matrix $K$-algebra
$\mathcal{A}=\sum_{i,j}\mathcal {A}_{ij}$ which satisfies the
additional condition $\mathcal{A}_{ii}= K$.
\end{definition}
A matrix Azumaya $K$-algebra is, in particular, an Azumaya
$K$-algebra. Let $\Der_K(\mathcal{A})^{loc}$ denote the sheaf of
$K$-linear derivations which preserve the decomposition. Note that
$\Der_K(\mathcal{A})^{loc}$ is a sheaf of \emph{abelian} Lie
algebras. The short exact sequence \eqref{ses:ad} restricts to the
short exact sequence
\[
0 \to K \to \mathfrak{d}(\mathcal{A}) \xrightarrow{\delta}
\Der_K({\mathcal A})^{loc} \to 0
\]

For an $\mathcal{O}_X$ matrix Azumaya algebra we denote by $\mathcal
{C}(\mathcal{A})^{loc}$ the subsheaf of $\mathcal{C}(\mathcal{A})$
whose sections are the connections which preserve the decomposition.
The sheaf ${\mathcal C}({\mathcal A})^{loc}$ is an
$\Omega^1_X\otimes_{\mathcal{O}_X}\Der_{{\mathcal O}_X}({\mathcal
A})^{loc}$-torsor.

If ${\mathcal A}$ is a matrix Azumaya algebra, then both
$\mathcal{A}\otimes_{\mathcal{O}_X} \mathcal{J}_X$ and
$\mathcal{J}_X(\mathcal{A})$ are matrix Azumaya
$\mathcal{J}_X$-algebras. Let
$\shIsom_0(\mathcal{A}\otimes_{\mathcal{O}_X}
\mathcal{J}_X,\mathcal{J}_X(\mathcal{A}))^{loc}$ denote the sheaf of
$\mathcal{J}_X$-matrix algebra isomorphisms
$\mathcal{A}\otimes_{\mathcal{O}_X}\mathcal{J}_X \to
\mathcal{J}_X(\mathcal{A})$ making the following diagram
commutative:
\begin{equation*}
\xymatrix{ & \mathcal{A}\otimes_{\mathcal{O}_X}\mathcal{J}_X \ar[rr]\ar[d]_{\id \otimes p} &&  \mathcal{J}_X(\mathcal{A})  \ar[d]^{p_{\mathcal{A}}}\\
           &\mathcal{A} \ar[rr]^{\id}&& \mathcal{A}
          }
\end{equation*}
where $p_{\mathcal{A}}$ is the canonical projection and $p \colon =
p_{\mathcal{O}_X}$.

Let $\shAut_0(\mathcal{A}\otimes_{\mathcal{O}_X}
\mathcal{J}_X)^{loc}$ denote the sheaf of $\mathcal{J}_X$-matrix
algebra automorphisms of $\mathcal{A}\otimes_{\mathcal{O}_X}
\mathcal{J}_X$ making the following diagram commutative:
\begin{equation*}
\xymatrix{ & \mathcal{A}\otimes_{\mathcal{O}_X}\mathcal{J}_X \ar[rr]\ar[d]_{\id \otimes p} &&  \mathcal{A}\otimes_{\mathcal{O}_X}\mathcal{J}_X \ar[d]^{\id\otimes p}\\
           &\mathcal{A} \ar[rr]^{\id}&& \mathcal{A}
          }
\end{equation*}
The sheaf $\shIsom_0(\mathcal{A}\otimes_{\mathcal{O}_X}
\mathcal{J}_X, \mathcal{J}_X(\mathcal{A}))^{loc}$ is a torsor under
$\shAut_0(\mathcal{A}\otimes_{\mathcal{O}_X} \mathcal{J}_X)^{loc}$
and the latter is soft.

Note that $\shAut_0(\mathcal{A}\otimes_{\mathcal{O}_X}
\mathcal{J}_X)^{loc}$ is a sheaf of pro-unipotent Abelian groups and
the map
\begin{equation}\label{exponential map}
\exp \colon  \Der_{\mathcal{O}_X}(\mathcal{A})^{loc}
\otimes_{\mathcal{O}_X} \mathcal{J}_{X,0}\to
\shAut_0(\mathcal{A}\otimes_{\mathcal{O}_X} \mathcal{J}_X)^{loc}
\end{equation}
is an isomoprhism of sheaves of groups.

\subsection{DGLA of local cochains}

\subsubsection{Local cochains on matrix algebras}\label{local cochains}
Suppose that ${\mathcal B} = \bigoplus_{i,j=0}^p{\mathcal B}_{ij}$
is a sheaf of matrix $k$-algebras. Under these circumstances one can
associate to ${\mathcal B}$ a DGLA of \emph{local cochains} defined
as follows.

Let $C^0(\mathcal{B})^{loc} = \mathfrak{d}(\mathcal{B})$. For $n\geq
1$ let $C^n({\mathcal B})^{loc}$ denote the subsheaf of
$C^n({\mathcal B})$ of multilinear maps $D$ such that for any
collection of $s_{i_kj_k}\in{\mathcal B}_{i_kj_k}$
\begin{enumerate}
\item $D(s_{i_1j_1}\otimes\cdots\otimes s_{i_nj_n}) = 0$ unless $j_k=i_{k+1}$ for all $k = 1,\dots ,n-1$

\item $D(s_{i_0i_1}\otimes s_{i_1i_2}\otimes\cdots\otimes s_{i_{n-1}i_n})\in{\mathcal B}_{i_0i_n}$
\end{enumerate}
For $I = (i_0,\ldots,i_n)\in [p]^{\times n+1}$ let
\begin{equation*}
C^I(\mathcal{B})^{loc} \colon = C^n({\mathcal B})^{loc} \bigcap
\shHom_k(\otimes_{j=0}^{n-1} \mathcal{B}_{i_j
i_{j+1}},\mathcal{B}_{i_0 i_n}) \ .
\end{equation*}
The restriction maps along the embeddings $\otimes_{j=0}^{n-1}
\mathcal{B}_{i_j i_{j+1}}\hookrightarrow\mathcal{B}^{\otimes n}$
induce an isomorphism $C^n(\mathcal{B})^{loc}\to\oplus_{I\in
[p]^{\times n+1}} C^I(\mathcal{B})^{loc}$.

The sheaf $C^\bullet({\mathcal B})^{loc}[1]$ is a subDGLA of
$C^\bullet({\mathcal B})[1]$ and the inclusion map is a
quasi-isomorphism.

\subsubsection{Combinatorial restriction of local cochains}
As in \ref{local cochains}, ${\mathcal B} =
\bigoplus_{i,j=0}^q{\mathcal B}_{ij}$ is a sheaf of matrix
$K$-algebras.

The DGLA $C^\bullet({\mathcal B})^{loc}[1]$ has additional variance
not exhibited by $C^\bullet({\mathcal B})[1]$. Namely, for $f\colon
[p] \to [q]$ there is a natural map of DGLA
\begin{equation}\label{combinatorial restiction of cochains}
f^\sharp \colon  C^\bullet({\mathcal B})^{loc}[1]\to
C^\bullet(f^\sharp{\mathcal B})^{loc}[1]
\end{equation}
defined as follows. Let $f^{ij}_\sharp\colon
(f^\sharp\mathcal{B})_{ij}\to\mathcal{B}_{f(i)f(j)}$ denote the
tautological isomorphism. For each collection $I =
(i_0,\ldots,i_n)\in [p]^{\times (n+1)}$ let
\begin{equation*}
f^I_\sharp \colon = \otimes_{j=0}^{n-1}f^{i_j i_{j+1}}_\sharp \colon
\otimes_{j=0}^{n-1} (f^\sharp\mathcal{B})_{i_j
i_{j+1}}\to\otimes_{i=0}^{n-1}\mathcal{B}_{f(i_j) f(i_{j+1})} \ .
\end{equation*}
Let $f^n_\sharp \colon = \oplus_{I\in [p]^{\times
(n+1)}}f^I_\sharp$. The map \eqref{combinatorial restiction of
cochains} is defined as restriction along $f^n_\sharp$.

\begin{lemma}
The map \eqref{combinatorial restiction of cochains} is a morphism
of DGLA
\begin{equation*}
f^\sharp \colon  C^\bullet(\mathcal{B})^{loc}[1]\to
C^\bullet(f^\sharp\mathcal{B})^{loc}[1] \ .
\end{equation*}
\end{lemma}

\subsubsection{Deformations of matrix algebras}
For a matrix algebra $\mathcal B$ on $X$ we denote by
$\Def(\mathcal{B})^{loc}(R)$ the subgroupoid of
$\Def(\mathcal{B})(R)$ with objects $R$-star products which respect
the decomposition given by $(\mathcal{B}\otimes_k R)_{ij} =
\mathcal{B}_{ij}\otimes_k R$ and $1$- and $2$-morphisms defined
accordingly. The composition
\[
\Def(\mathcal{B})^{loc}(R)\to \Def(\mathcal{B})(R) \to
\MC^2(\Gamma(X; C^\bullet(\mathcal{B})[1])\otimes_k \mathfrak{m}_R)
\]
takes values in $\MC^2(\Gamma(X;
C^\bullet(\mathcal{B})^{loc}[1])\otimes_k \mathfrak{m}_R)$ and
establishes an isomorphism of $2$-groupoids
$\Def(\mathcal{B})^{loc}(R)\cong \MC^2(\Gamma(X;
C^\bullet(\mathcal{B})^{loc}[1])\otimes_k \mathfrak{m}_R)$.

\section{Deformations of cosimplicial matrix Azumaya algebras}\label{Deformations of cosimplicial matrix Azumaya
algebras}

\subsection{Cosimplicial matrix algebras}\label{section:
cosimplicial matrix algebras}

Suppose that $X$ is a simplicial space. We assume given a
cosimplicial pseudo-tensor category $\Psi \colon  [p] \mapsto
\Psi^p$, $p = 0, 1, 2, \ldots$, where $\Psi^p$ is a pseudo-tensor
subcategory $\Sh_k(X_p)$ (see \ref{subsubsection:pseudo-tensor
cats}) so that for any morphism $f \colon  [p] \to [q]$ in $\Delta$
the corresponding functor $X(f)^{-1} \colon  \Sh_k(X_p) \to
\Sh_k(X_q)$ restricts to a functor $X(f)^{-1} \colon  \Psi^p \to
\Psi^q$. If $X$ is an \'etale simplicial manifold, then both
$\mathtt{DIFF}$ and $\mathtt{JET}$ are examples of such. In what
follows all algebraic structures are understood to be given by
operations in $\Psi$.

Suppose that $K$ is a special cosimplicial sheaf (see Definition
\ref{def special sheaf}) of commutative algebras on $X$.

\begin{definition}
A \emph{cosimplicial matrix $K$-algebra $\mathcal{A}$ on $X$} is
given by the following data:
\begin{enumerate}
\item  for each $p=0,1,2,\ldots$ a matrix $K^p$-algebra $\mathcal{A}^p = \sum_{i,j = 0}^p \mathcal{A}^p_{ij}$

\item for each morphism $f\colon [p]\to[q]$ in $\Delta$ an isomorphism of matrix $K^q$-algebras $f_* \colon
    X(f)^{-1}{\mathcal A}^p\to f^\sharp{\mathcal A}^q$.
\end{enumerate}
These are subject to the associativity condition: for any pair of
composable arrows $[p]\stackrel{f}{\to}[q]\stackrel{g}{\to}[r]$
\[
(g\circ f)_* = f^\sharp(g_*)\circ X(g)^{-1}(f_*)
\]
\end{definition}

\begin{remark}
As is clear from the above definition, a cosimplicial matrix algebra
is \emph{not} a cosimplicial sheaf of algebras in the usual sense.
\end{remark}

Suppose that $\mathcal{A}$ and $\mathcal{B}$ are two cosimplicial
matrix algebras on $X$. A \emph{$1$-morphism of cosimplicial matrix
algebras $F\colon  \mathcal{A}\to\mathcal{B}$} is given by the
collection of morphisms
\[
F^p\colon \mathcal{A}^p \to \mathcal{B}^p
\]
of matrix $K^p$-algebras subject to the compatibility condition: for
any morphism $f\colon [p]\to [q]$ in $\Delta$ the diagram
\[
\begin{CD}
 X(f)^{-1}{\mathcal A}^p
@>{f_*}>> f^\sharp{\mathcal A}^q \\
@V{X(f)^*F^p}VV @VV{f^{\sharp}F^q}V \\
X(f)^{-1}{\mathcal B}^p @>{f_*}>> f^\sharp{\mathcal B}^q
\end{CD}
\]
commutes. The composition of $1$-morphisms is given by the
composition of their respective components.

The identity $1$-morphism $\id_{\mathcal{A}} \colon  \mathcal{A}\to
\mathcal{A}$ is given by $\id_{\mathcal{A}}^p = \id$.

Suppose that $F_1, F_2 \colon  \mathcal{A}\to \mathcal{B}$ are two
$1$-morphisms. A \emph{$2$-morphism $b \colon F_1\to F_2$} is given
by a collection of $2$-morphisms  $b^p\colon F_1^p\to F_2^p  $
satisfying

\[
f_*(X(f)^*b^p)=f^{\sharp} b^q
\]
for any morphism $f\colon [p]\to [q]$ in $\Delta$. The compositions
of $2$-morphisms are again componentwise.

Let $\CosMatAlg^\Psi_K(X)$ denote the category of cosimplicial
matrix $K$-algebras with $1$- and $2$-morphisms defined as above. In
the case when $\Psi = \Sh_k(X)$, i.e. no restrictions are imposed,
we will simply write $\CosMatAlg_K(X)$.

\subsection{Deformations of cosimplicial matrix algebras}
\subsubsection{DGLA from cosimplicial matrix algebras}\label{DGLA from cosimplicial alg}
Suppose that $X$ is a simplicial space and $K$ is a special
cosimplicial sheaf of commutative $k$-algebras on $X$. To
$\mathcal{A} \in \CosMatAlg_K^\Psi(X)$ we associate a cosimplicial
sheaf of DGLA $\mathfrak{g}(\mathcal{A})$ on $\real{X}$ (see
\ref{geometric realization}).

Let $\mathfrak{g}^n(\mathcal{A})$ denote the sheaf on $\real{X}_n$
whose restriction to $\real{X}_\lambda$ is equal to
$X(\lambda(0n))^{-1} C^\bullet({\mathcal A}^{\lambda(0)})^{loc}[1]$.

For a morphism $f \colon  [m] \to [n]$ in $\Delta$ let
\[
f_* \colon  \real{X}(f)^{-1}\mathfrak{g}^m(\mathcal{A}) \to
\mathfrak{g}^n(\mathcal{A})
\]
denote the map whose restriction to $\real{X}_\lambda$ is equal to
the composition
\begin{multline*}
X(\Upsilon(f)^\lambda)^{-1}X(f^*(\lambda)(0m))^{-1} C^\bullet(\mathcal{A}^{f^*(\lambda)(0)})^{loc}[1] \cong \\
X(\Upsilon(f)^\lambda\circ
f^*(\lambda)(0m))^{-1}C^\bullet(\mathcal{A}^{f^*(\lambda)(0)})^{loc}[1]
\xrightarrow{f^\sharp} \\
X(\lambda(0n))^{-1}C^\bullet(f^\sharp\mathcal{A}^{\lambda(0)})^{loc}[1] \cong \\
X(\lambda(0n))^{-1}C^\bullet({\mathcal A}^{\lambda(0)})^{loc}[1]
\end{multline*}
in the notations of \ref{subsubsection: subdivision}.

We leave it to the reader to check that the assignment
$[n]\mapsto\mathfrak{g}^n(\mathcal{A})$, $f \mapsto f_*$ is a
cosimplicial sheaf of DGLA on $\real{X}$.

We will denote by $\mathfrak{d}(\mathcal{A})$ the cosimplicial
subDGLA $[n]\mapsto
X(\lambda(0n))^{-1}\mathfrak{d}(\mathcal{A}^{\lambda(0)})$.

For each $i = 0, 1, \ldots$ we have the cosimplicial vector space of
degree $i$ local cochains $\Gamma(\real{X};
\mathfrak{g}^{\bullet,i}(\mathcal{A}))$, $[n] \mapsto
\Gamma(\real{X}_n; \mathfrak{g}^{n,i}(\mathcal{A}))$. The following
theorem was proved in \cite{bgnt3} in the special case when $X$ is
the nerve of an open cover of a manifold. The proof of Theorem 5.2
in \cite{bgnt3} extends verbatim to give the following result.

\begin{thm}\label{theorem: acyclicity}
For each $i,j\in\mathbb{Z}$, $j\neq 0$, $H^j(\Gamma(\real{X};
\mathfrak{g}^{\bullet,i}(\mathcal{A}))) = 0$.
\end{thm}

Let
\begin{equation}\label{def complex cosimp matrix alg}
\mathfrak{G}(\mathcal{A}) =
\Tot(\Gamma(\real{X};\mathfrak{g}(\mathcal{A}))) .
\end{equation}
As will be shown in Theorem \ref{thm:mapgroup equiv}, the DGLA
$\mathfrak{G}(\mathcal{A})$ plays the role of a deformation complex
of $\mathcal A$.

\subsubsection{The deformation functor}
Suppose that $X$ is a simplicial space, $\Psi$ is as in
\ref{section: cosimplicial matrix algebras} subject to the
additional condition as in \ref{subsubsection: star products}. For
$\mathcal{A} \in \CosMatAlg^\Psi_k(X)$  we define the deformation
functor $\Def(\mathcal{A})$ on $\ArtAlg_k$ (see \ref{subsubsection:
star products} for Artin algebras).

\begin{definition}\label{definition:DefCosMatAlg}
An \emph{$R$-deformation $\mathbf{B}_R$ of $\mathcal A$} is a
cosimplicial matrix $R$-algebra structure on $\mathcal{A}_R \colon =
\mathcal{A}\otimes_k R$ with the following properties:
\begin{enumerate}
\item $\mathbf{B}_R\in\CosMatAlg^\Psi_R(X)$

\item for all $f \colon  [p]\to[q]$ the structure map $f^{\mathbf{B}}_* \colon  X(f)^{-1}\mathbf{B}^p\to
    f^\sharp\mathbf{B}^q$ is equal to the map $f^{\mathcal{A}}_*\otimes\id_R \colon
    X(f)^{-1}\mathcal{A}^p\otimes_k R \to f^\sharp\mathcal{A}^q\otimes_k R$.
\item the identification $\mathbf{B}_R\otimes_R k \cong \mathcal{A}$ is compatible with the respective cosimplicial
    matrix algebra structures.
\end{enumerate}
\end{definition}

A $1$-morphism of $R$-deformations of $\mathcal A$ is a $1$-morphism
of cosimplicial matrix $R$-algebras which reduces to the identity
$1$-morphism $\id_{\mathcal{A}}$ modulo the maximal ideal.

A $2$-morphism between $1$-morphisms of deformations is a
$2$-morphism which reduces to the identity endomorphism of
$\id_{\mathcal{A}}$ modulo the maximal.

We denote by $\Def(\mathcal{A})(R)$ the $2$-category with objects
  $R$-deformations of $\mathcal A$, $1$-morphisms and
$2$-morphisms as above. It is clear that the assignment
$R\mapsto\Def(\mathcal{A})(R)$ is natural in $R$.

\subsubsection{}
Suppose given an $R$-deformation $\mathbf B$ of $\mathcal A$  as
above. Then, for each $p=0,1,2,...$ we have the Maurer-Cartan (MC)
element $\gamma^p\in\Gamma(X_p;C^\bullet(\mathcal{A}^p)^{loc}[1])
\otimes_k \mathfrak{m}_R$ which corresponds to the $R$-deformation
$\mathbf{B}^p$ of the matrix algebra $\mathcal{A}^p$ on $X_p$. This
collection defines a MC element  $\gamma \in \Gamma (\real{X}_0,
\mathfrak{g}^0(\mathcal{A}))$. It is clear from the definition that
$\gamma \in \ker(\Gamma (\real{X}_0,
\mathfrak{g}^0(\mathcal{A}))\rightrightarrows \Gamma (\real{X}_1,
\mathfrak{g}^1(\mathcal{A})))\otimes_k \mathfrak{m}_R$. This
correspondence induces a bijection between the objects of
$\Def(X;\mathcal{A})(R)$ and those of $\MC^2(\ker(\Gamma
(\real{X}_0, \mathfrak{g}^0(\mathcal{A}))\rightrightarrows \Gamma
(\real{X}_1, \mathfrak{g}^1(\mathcal{A})))\otimes_k
\mathfrak{m}_R)$, which can clearly be extended to an isomorphism of
$2$-groupoids. Recall that we have a morphism of DGLA
\begin{equation}
\ker(\Gamma (\real{X}_0,
\mathfrak{g}^0(\mathcal{A}))\rightrightarrows \Gamma (\real{X}_1,
\mathfrak{g}^1(\mathcal{A}))) \to \mathfrak{G}(\mathcal{A})
\end{equation}
 where
$\mathfrak{G}(\mathcal{A}) =
\Tot(\Gamma(\real{X};\mathfrak{g}(\mathcal{A})))$.

 This morphism
induces a morphism of the corresponding Deligne $2$-groupoids.
Therefore we obtain a morphism of $2$-groupoids
\begin{equation}\label{mapgroup}
\Def(\mathcal{A})(R) \to \MC^2(\mathfrak{G}(\mathcal{A})\otimes_k
\mathfrak{m}_R)
\end{equation}
\begin{thm}\label{thm:mapgroup equiv}
The map \eqref{mapgroup} is an equivalence of $2$-groupoids.
\end{thm}
\begin{proof}
Theorem \ref{theorem: acyclicity} shows that the condition
\eqref{acyclicity condition} of the Proposition \ref{prop: MC2 ker
is MC2 tot} is satisfied. The statement of the Theorem then follows
from the conclusion of the Proposition \ref{prop: MC2 ker is MC2
tot}.
\end{proof}




\subsection{Deformations of cosimplicial matrix Azumaya algebras}
Suppose that $X$ is a \emph{Hausdorff} \'etale simplicial manifold.

\subsubsection{Cosimplicial matrix Azumaya algebras}
\begin{definition}
A \emph{cosimplicial matrix Azumaya $K$-algebra} $\mathcal{A}$ on
$X$ is a cosimplicial matrix $K$-algebra on $X$ (see \ref{section:
cosimplicial matrix algebras}) such that for every $p = 0, 1, 2,
\ldots$ the matrix algebra $\mathcal{A}^p$ is a matrix Azumaya
$K$-algebra on $X_p$ (see Definition \ref{def: mat azum K-alg}).
\end{definition}

Suppose that $\mathcal{A}$ is a cosimplicial matrix
$\mathcal{O}_X$-Azumaya algebra. Recall that, by convention we treat
such as objects of $\mathtt{DIFF}$ (see \ref{subsubsection:examples
of pseudo-tensor cats}).

Then, $\mathcal{J}_X(\mathcal{A})$ is a cosimplicial matrix
$\mathcal{J}_X$-Azumaya algebra, which we view as an object of
$\mathtt{JET}$, (see \ref{subsubsection:examples of pseudo-tensor
cats}) equipped with the canonical flat connection
$\nabla^{can}_{\mathcal{A}}$. Therefore, we have the cosimplicial
DGLA with flat connection $\mathfrak{g}(\mathcal{J}_X(\mathcal{A}))$
on $\real{X}$, hence the cosimplicial DGLA
$\DR(\mathfrak{g}(\mathcal{J}_X(\mathcal{A}))$. Let
\[
\mathfrak{G}_\DR(\mathcal{J}_X(\mathcal{A})) =
\Tot(\Gamma(\real{X};\DR(\mathfrak{g}(\mathcal{J}_X(\mathcal{A})))))
\]

The inclusion of the subsheaf of horizontal sections is a
quasi-isomorphism of DGLA $\mathfrak{g}(\mathcal{A})\to
\DR(\mathfrak{g}(\mathcal{J}_X(\mathcal{A})))$ which induces the
quasi-isomorphism of DGLA
\begin{equation}\label{incl}
\mathfrak{G}(\mathcal{A})\to
\mathfrak{G}_\DR(\mathcal{J}_X(\mathcal{A}))
\end{equation}

\subsubsection{Cosimplicial splitting}\label{cosimplicial splitting}
Suppose that $\mathcal{A}$ is a cosimplicial matrix
$\mathcal{O}_X$-Azumaya algebra.

According to \ref{gerbe of splittings}, for each $p=0,1,2,\ldots$ we
have the gerbe $S(\mathcal{A}^p)$ on $X_p$. Moreover, for each
morphism $f\colon [p]\to[q]$ in $\Delta$ we have the morphism
$f_*\colon X(f)^{-1} S(\mathcal{A}^p)\to S(\mathcal{A}^q)$ defined
as the composition
\[
X(f)^{-1} S(\mathcal{A}^p)\cong S(X(f)^{-1}\mathcal{A}^p)\cong
S(f^\sharp\mathcal{A}^q)
\stackrel{S(\mathcal{A}^q_f)}{\longrightarrow} S(\mathcal{A}^q)
\]

It is clear that the assignment $[p]\mapsto S(\mathcal{A}^p)$,
$f\mapsto f_*$ is a cosimplicial gerbe on $X$.

Let ${\mathcal E}^p\colon =\oplus_{j=0}^n \mathcal{A}^p_{j0}$. There
is a natural isomorphism $\mathcal{A}^p\cong\shEnd({\mathcal E}^p)$,
the action given by the isomorphism $\mathcal{A}_{ij}\otimes
\mathcal{A}_{j0}\to \mathcal{A}_{i0}$. In other words,
$\mathcal{E}^p$ is a splitting of the Azumaya algebra
$\mathcal{A}^p$, i.e. a morphism $\mathcal{E}^p \colon
\mathcal{O}^\times_{X_p}[1]\to S(\mathcal{A}^p)$.

We are going to extend the assignment $p\mapsto
\mathcal{O}^\times_{X_p}[1]$ to a cosimplicial gerbe
$\mathcal{S}_{\mathcal{A}}$ on $X$ so that $\mathcal{E}$ is a
morphism of cosimplicial gerbes $\mathcal{S}_{\mathcal{A}}\to
S(\mathcal{A})$. To this end, for $f \colon  [p]\to[q]$ let
$(\mathcal{S}_{\mathcal{A}})_{f} = \mathcal{A}^q_{0f(0)}$. For a
pair of composable arrows
$[p]\stackrel{f}{\to}[q]\stackrel{g}{\to}[r]$ let
$(\mathcal{S}_{\mathcal{A}})_{f,g}$ be defined by the isomorphism $
X(g)^{-1}( \mathcal{A}^q_{0f(0)})\otimes \mathcal{A}^r_{0g(0)}\cong
\mathcal{A}^r_{0g(f(0))}$.

Since $f_*\mathcal{A}^p_{i0}\cong\mathcal{A}^q_{f(i)f(0)}
\cong\mathcal{A}^q_{f(i)0}\otimes\mathcal{A}^q_{0f(0)}$ for every
$i$ we obtain a canonical chain of isomorphisms
\begin{equation}\label{2morph}
\mathcal{A}^q_f\otimes_{f^\sharp\mathcal{A}^q} f_*\mathcal{E}^p
\cong \bigoplus_{i=0}^q \mathcal{A}^q_{i f(0)} \cong
\bigoplus_{i=0}^q \mathcal{A}^q_{i 0}\otimes \mathcal{A}^q_{0 f(0)}
\cong {\mathcal E}^q \otimes \mathcal{A}^q_{0f(0)}.
\end{equation}

Let $\mathcal{E}_f$ be the $2$-morphism induced by the isomorphism
\eqref{2morph}. We than have the following:
\begin{lemma}\label{lemma: cosimp splitting}
  $(\mathcal{E}^p, \mathcal{E}_f)$ is a morphism of cosimplicial gerbes
$\mathcal{S}_{\mathcal{A}}\to S(\mathcal{A})$.
\end{lemma}

\subsubsection{Twisted DGLA of jets} \label{subsubsection: Twisted DGLA of jets}
\begin{definition}\label{def:twist of dgla}
For a DGLA $(\mathfrak{r},d)$ and a Maurer-Cartan element
$\gamma\in\Der(\mathfrak{r})$ we define the \emph{$\gamma$-twist of
$\mathfrak{r}$}, denoted $\mathfrak{r}_\gamma$, to be the DGLA whose
underlying graded Lie algebra coincides with that of $\mathfrak{r}$
and whose differential is equal to $d + \gamma$.
\end{definition}

In \ref{cosimplicial splitting} we associated with $\mathcal A$ a
cosimplicial gerbe $\mathcal{S}_{\mathcal{A}}$ on $X$. The
construction of \ref{subsubsection: char class trivialized gerbe}
associates to $\mathcal{S}_{\mathcal{A}}$ the characteristic class
$[\mathcal{S}_{\mathcal{A}}]\in
H^2(\real{X};\real{\DR(\overline{\mathcal{J}}_X)})$ represented by a
$2$-cocycle $\overline
B\in\Tot(\Gamma(\real{X};\real{\DR(\overline{\mathcal{J}}_X)}))$
dependent on appropriate choices.

Recall that under the standing assumption that $X$ is a Hausdorff
\'etale simplicial manifold, the cosimplicial sheaves
$\mathcal{O}_X$, $\mathcal{J}_X$, $\Omega^\bullet_X$ are special.

We have special cosimplicial sheaves of DGLA $\overline
C^\bullet(\mathcal{O}_X)[1]$ and $\DR(\overline
C^\bullet(\mathcal{J}_X)[1])$. The inclusion of the subsheaf of
horizontal sections is a quasi-isomorphism of DGLA $\overline
C^\bullet(\mathcal{O}_X)[1]\to\DR(\overline
C^\bullet(\mathcal{J}_X))[1]$. Let
\[
\mathfrak{G}_\DR(\mathcal{J}_X) =
\Tot(\Gamma(\real{X};\real{\DR(\overline
C^\bullet(\mathcal{J}_X)[1])}^\prime))
\]
(see \ref{realization for special sheaves} for
$\real{\bullet}^\prime$).

The canonical  isomorphism
$\real{\DR(\overline{\mathcal{J}}_X)}^\prime\cong
\real{\DR(\overline{\mathcal{J}}_X)} $ (see \ref{realization for
special sheaves}) induces the isomorphism of complexes
$\Tot(\Gamma(\real{X};\real{\DR(\overline{\mathcal{J}}_X)})) \cong
\Tot(\Gamma(\real{X};\real{\DR(\overline{\mathcal{J}}_X)}^\prime))$.
Thus, we may (and will) consider $\overline B$ as a $2$-cocycle in
the latter complex.

The adjoint action of the abelian DGLA $\mathcal{J}_X[1]$ on
$\overline{C}^\bullet(\mathcal{J}_X)[1]$ induced an action of
$\overline{\mathcal{J}}_X[1]$ on
$\overline{C}^\bullet(\mathcal{J}_X)[1]$. The latter action gives
rise to an action of the DGLA
$\Tot(\Gamma(\real{X};\real{\DR(\overline{\mathcal{J}}_X[1])}^\prime))$
on the DGLA $\mathfrak{G}_\DR(\mathcal{J}_X)$ by derivations. Since
the cocycle $\overline{B}$ is a Maurer-Cartan element in
$\Der(\mathfrak{G}_\DR(\mathcal{J}_X))$, the DGLA
$\mathfrak{G}_\DR(\mathcal{J}_X)_{\overline B}$ is defined.

\subsection{Construction}
Let $\mathcal A$ be a cosimplicial matrix $\mathcal{O}_X$-Azumaya
algebra on $X$. This section is devoted to the construction and
uniqueness properties of the isomorphism \eqref{cotr ind of choice}.

To simplify the notations we will denote $\mathcal{O}_X$
(respectively, $\mathcal{O}_{X_p}$, $\mathcal{J}_X$,
$\mathcal{J}_{X_p}$) by $\mathcal{O}$ (respectively $\mathcal{O}^p$,
$\mathcal{J}$, $\mathcal{J}^p$).

\subsubsection{Construction: step one}\label{subsubsection: step one}
We have the cosimplicial matrix $\mathcal{J}$-Azumaya algebra
$\mathcal{A}\otimes \mathcal{J}$, hence the cosimplicial graded Lie
algebras $\mathfrak{g}(\mathcal{A}\otimes \mathcal{J})$ and
$\Omega^\bullet_\real{X}\otimes\mathfrak{g}(\mathcal{A} \otimes
\mathcal{J})$. Let $\mathfrak{H}$ denote the graded Lie algebra
$\Tot(\Gamma(\real{X};\Omega^{\bullet}_\real{X} \otimes
\mathfrak{g}(\mathcal{A}\otimes \mathcal{J})))$.

We begin by constructing an isomorphism of  graded Lie algebras.
\begin{equation}\label{The ism}
\Sigma\colon   \mathfrak{H} \to
\mathfrak{G}_\DR(\mathcal{J}(\mathcal{A}))
\end{equation}

For each $p = 0,1,2,\ldots$ we choose
\[
\sigma^p\in\Isom_0(\mathcal{A}^p\otimes \mathcal{J}^p,
\mathcal{J}(\mathcal{A}^p))^{loc} \ .
\]
Consider a simplex $\lambda\colon [n] \to \Delta$.

For $0\leq i\leq n$    the composition
\begin{multline*}
X(\lambda(0n))^{-1}\mathcal{A}^{\lambda(0)}\to \\
X(\lambda(in))^{-1}X(\lambda(0i))^{-1}\mathcal{A}^{\lambda(0)}
\xrightarrow{\lambda(0i)_*} X(\lambda(in))^{-1}
\lambda(0i)^\sharp\mathcal{A}^{\lambda(i)}
\end{multline*}
defines an isomorphism
\begin{equation}\label{tau?}
X(\lambda(0n))^{-1}\mathcal{A}^{\lambda(0)}\to X(\lambda(in))^{-1}
\lambda(0i)^\sharp\mathcal{A}^{\lambda(i)}
\end{equation}
The composition of the map $X(\lambda(in))^* \lambda(0i)^\sharp$
with the isomorphism \eqref{tau?} defines the following maps, all of
which we denote by $\tau_{\lambda, i}^*$:
\begin{multline*}
\tau_{\lambda, i}^*\colon \Isom_0(\mathcal{A}^{\lambda(i)}\otimes
\mathcal{J}^{\lambda(i)},
\mathcal{J}(\mathcal{A}^{\lambda(i)}))^{loc}\to  \\
\Isom_0(X(\lambda(0n))^{-1}\mathcal{A}^{\lambda(0)} \otimes
\mathcal{J}^{\lambda(n)},
\mathcal{J}(X(\lambda(0n))^{-1}\mathcal{A}^{\lambda(0)}))^{loc}
\end{multline*}
\begin{multline*}
\tau_{\lambda, i}^* \colon \Omega^{\bullet}_{X_{\lambda(i)}} \otimes
C^{\bullet}( \mathcal{A}^{\lambda(i)})^{loc}
\otimes \mathcal{J}^{\lambda(i)} \to \\
 \Omega^{\bullet}_{X_{\lambda(n)}}
\otimes
C^{\bullet}((X(\lambda(0n))^{-1}\mathcal{A}^{\lambda(0)})^{loc}
\otimes \mathcal{J}^{\lambda(n)}
\end{multline*}
etc. Define $\sigma^\lambda_i =\tau_{\lambda, i}^*
\sigma^{\lambda(i)}$.

Recall that $\pr_X\colon  X_p\times\Delta^n\to X_p$ denotes the
projection. Let $\sigma^\lambda$ be an element  in $\Isom_0(
\pr_X^{-1}(X(\lambda(0n))^{-1}\mathcal{A}^{\lambda(0)} \otimes
\mathcal{J}^{\lambda(0)}), \ \pr_X^{-1}
\mathcal{J}(X(\lambda(0n))^{-1}\mathcal{A}^{\lambda(0)}))^{loc} $
  defined
as follows. For each morphism $f\colon [p]\to[q]$ in $\Delta$ there
is a unique
\[
\vartheta(f)\in \Gamma(X_q;\Der_{\mathcal{O}^q_X}
(X(f)^{-1}\mathcal{A}^p)^{loc} \otimes \mathcal{J}^q_0)
\]
such that the composition
\begin{equation*}
X(f)^{-1}\mathcal{A}^p \otimes \mathcal{J}^q
\xrightarrow{\exp(\vartheta(f))} X(f)^{-1}\mathcal{A}^p \otimes
\mathcal{J}^q \xrightarrow{X(f)^{-1}\sigma^p}
X(f)^{-1}\mathcal{J}(\mathcal{A}^p)
\end{equation*}
is equal to the composition
\begin{multline*}
X(f)^{-1}\mathcal{A}^p \otimes \mathcal{J}^q \xrightarrow{f_*}
f^\sharp\mathcal{A}^q \otimes \mathcal{J}^q \xrightarrow{f^\sharp(\sigma^q)} \\
f^\sharp\mathcal{J}(\mathcal{A}^q) \xrightarrow{(f_*)^{-1}}
X(f)^{-1}\mathcal{J}(\mathcal{A}^p).
\end{multline*}
Let
\begin{equation}\label{equationsigma}
\sigma^{\lambda}=(\pr_X^*\sigma^{\lambda}_n) \circ \exp
(-\sum_{i=0}^n t_i\cdot\lambda(0i)^{\sharp}\vartheta(\lambda(in))).
\end{equation}

In this formula and below we use the isomorphism \eqref{tau?} to
view $\lambda(0i)^{\sharp}\vartheta(\lambda(in))$ as an element of
$\Gamma(X_{\lambda(n)};\Der_{\mathcal{O}^{\lambda(n)}}(X(\lambda(0n))^{-1}\mathcal{A}^{\lambda(0)})^{loc}
\otimes \mathcal{J}^{\lambda(n)})$.

The isomorphism of algebras $\sigma^{\lambda}$ induces the
isomorphism of graded Lie algebras
\begin{multline*}
 \Omega^\bullet_{X_{\lambda(n)}\times \Delta^n}\otimes
\pr_X^{-1}(\mathfrak{g}_{\lambda}^n(\mathcal{A} \otimes
\mathcal{J})) \xrightarrow{(\sigma^\lambda)^*} \\
\Omega^\bullet_{X_{\lambda(n)}\times \Delta^n}\otimes
\pr_X^{-1}(\mathfrak{g}_{\lambda}^n(\mathcal{J}(\mathcal{A})))
\end{multline*}

\begin{lemma} The map $\sigma^{\lambda}$ induces an isomorphism
\begin{equation*}
 \Omega_n\otimes_{\mathbb{C}}\Omega^\bullet_{X_{\lambda(n)}}\otimes
\mathfrak{g}_{\lambda}^n(\mathcal{A} \otimes \mathcal{J})
\xrightarrow{(\sigma^\lambda)^*}
\Omega_{n}\otimes_{\mathbb{C}}\Omega^\bullet_{X_{\lambda(n)}}\otimes
\mathfrak{g}_{\lambda}^n(\mathcal{J}(\mathcal{A}))
\end{equation*}
\end{lemma}
\begin{proof}
It is sufficient to check that $\exp (-\sum_{i=0}^n
t_i\cdot\lambda(0i)^{\sharp}\vartheta(\lambda(in)))$ maps
$\mathfrak{g}_{\lambda}^n(\mathcal{A} \otimes \mathcal{J})$ into
$\Omega_n\otimes_\mathbb{C} \mathfrak{g}_{\lambda}^n(\mathcal{A}
\otimes \mathcal{J})$. Since the Lie algebra $\Der_{\mathcal{O}^q}
(X(f)^{-1}\mathcal{A}^p)^{loc}$ is commutative, this is a
consequence of the following general statement: if $\mathcal{A}$ is
an Azumaya algebra on $X$, $D\in C^\bullet(\mathcal{A}\otimes
\mathcal{J})[1]$ and $\vartheta \in \Gamma(X;\Der_{\mathcal{O}_X}
(\mathcal{A})\otimes \mathcal{J}_{X, 0})$, then
$(\exp(t\ad\vartheta))^*D$ is polynomial in $t$. But this is so
because $\vartheta$ is inner and $D$ is section of a sheaf of jets
of \emph{multidifferential} operators.
\end{proof}
It is clear that the collection of maps $(\sigma^\lambda)^*$ is a
morphism of cosimplicial graded Lie algebras; the desired
isomorphism $\Sigma$ \eqref{The ism} is, by definition the induced
isomorphism of graded Lie algebras. We now describe the differential
on $\mathfrak{H}$ induced by the differential on
$\mathfrak{G}_{\DR}(\mathcal{J}(\mathcal{A}))$ via the isomorphism
\eqref{The ism}.

Recall that the differential in
$\mathfrak{G}_{\DR}(\mathcal{J}(\mathcal{A}))$  is given by $\delta+
\widetilde{\nabla}^{can}$. It is easy to see that $\delta$ induces
the Hochschild differential $\delta$ on $\mathfrak{H}$. The
canonical connection $\widetilde{\nabla}^{can}$ induces the
differential whose component corresponding to the  simplex $\lambda$
is the connection $(\sigma^{\lambda})^{-1}\circ X(\lambda(0n))^*
\widetilde{\nabla}^{can} \circ \sigma^{\lambda}$.

To get a more explicit description of this connection choose for
each $p = 0,1,2,\ldots$ a connection
$\nabla^p\in\mathcal{C}(\mathcal{A}^p)^{loc}(X_p)$; it gives rise to
the connection $\nabla^p\otimes\id + \id\otimes\nabla^{can}$ on
$\mathcal{A}^p \otimes \mathcal{J}^p$. Let
\[
\Phi^p = (\sigma^p)^{-1}\circ\nabla^{can}\circ\sigma^p -
(\nabla^p\otimes\id + \id\otimes\nabla^{can} ) \ ,
\]
$\Phi^p\in \Gamma(X_p;\Omega^1_{X_p}\otimes
\Der_{\mathcal{O}^p}(\mathcal{A}^p)^{loc} \otimes \mathcal{J}^p)
\subset \Gamma(X_p;\Omega^1_{X_p}\otimes C^1(\mathcal{A}^p\otimes
\mathcal{J}^p)^{loc})$.

Let
\[
\nabla^\lambda = \sum_{i=0}^n
t_i\tau_{\lambda,i}^*\nabla^{\lambda(i)}\otimes\id +
\id\otimes\widetilde{\nabla}^{can}
\]
be the induced derivation of $\Omega_n
\otimes_{\mathbb{C}}\Gamma(X_{\lambda(n)};
X(\lambda(0n))^{-1}\mathcal{A}^{\lambda(0)})$. It is easy to see
that the collection $\nabla^{\lambda}$ induces a derivation on
$\mathfrak{H}$ which we denote by $\nabla$.

Let
\[
\Phi^\lambda = (\sigma^\lambda)^{-1}\circ X(\lambda(0n))^*
\widetilde{\nabla}^{can} \circ\sigma^\lambda - \nabla^\lambda \ ,
\]
$\Phi^\lambda\in\Omega_n\otimes\Gamma( X_{\lambda(n)};
\Omega^{\bullet}_{X_{\lambda(n)}} \otimes
\Der_{\mathcal{O}^{\lambda(n)}}(X(\lambda(0n))^{-1}\mathcal{A}^{\lambda(0)})^{loc}
\otimes \mathcal{J}^{\lambda(n)})$. Using the formula
\eqref{equationsigma} we obtain:
\begin{equation}\label{phil}
\Phi^\lambda = \sum_{i=0}^n t_i\cdot
\tau_{\lambda,i}^*\Phi^{\lambda(i)} - \sum_{i=0}^n dt_i
\wedge\lambda(0i)^{\sharp}\vartheta(\lambda(in)) \ .
\end{equation}
It is easy to see that the collection $\Phi^{\lambda}$ defines an
element total degree one in $\mathfrak{H}$. The differential induced
on $\mathfrak{H}$ via the isomorphism \eqref{The ism} can therefore
be written as
\begin{equation}\label{differ}
\delta + \nabla + \ad\Phi
\end{equation}

\subsubsection{}
Now we construct an automorphism of the graded Lie algebra
$\mathfrak{H}$ which conjugates the differential given by the
formula \eqref{differ} into a simpler one. This is achieved by
constructing $F\in
\Tot(\Gamma(\real{X};\Omega^\bullet_\real{X}\otimes
(\mathfrak{d}(\mathcal{A}) \otimes \mathcal{J})))
\subset\mathfrak{H}$ with components $F^{\lambda}$ such that
$\Phi^{\lambda}= - \delta F^{\lambda}$.  This construction requires
the following choices:
\begin{enumerate}
\item for each $p$ we chose
\[
F^p\in\Gamma(X_p;\Omega^1_{X_p} \otimes (\mathfrak{d}(\mathcal{A}^p)
\otimes \mathcal{J}^p))
\]
such that $\Phi^p= - \delta F^p$, and

\item for each morphism $f\colon [p]\to[q]$ in $\Delta$ we chose
\[
D(f)\in \Gamma(X_q; X(f)^{-1}\mathfrak{d}(\mathcal{A}^p) \otimes
\mathcal{J}^q_0)
\]
such that $\vartheta(f)= \delta D(f)$.
\end{enumerate}

The unit map
\begin{equation}\label{unit map}
\mathcal{J}^r \to X(g\circ f)^{-1}\mathcal{A}^p \otimes
\mathcal{J}^r
\end{equation}
gives the canonical identification of the former with the center of
the latter.

\begin{lemma}\label{lemma: central section}
For a pair of composable arrows $[p]\xrightarrow{f}[q]
\xrightarrow{g}[r]$ the section $X(g)^*D(f)+ f^{\sharp}D(g)-D(g
\circ f) \in \Gamma(X_r;X(g\circ f)^{-1}\mathcal{A}^p\otimes
\mathcal{J}^r_0)$ is the image of a unique section
\[
\beta(f, g) \in \Gamma(X_r; \mathcal{J}^r_0)
\]
under the map \eqref{unit map}.(Here we regard $f^\sharp D(g)$ as an
element of $\Gamma(X_r;X(g\circ f)^{-1}\mathcal{A}^p\otimes
\mathcal{J}^r)$ using \eqref{tau?}.)
\end{lemma}
\begin{proof}
Follows from the identity $X(g)^*\vartheta(f)+
f^{\sharp}\vartheta(g)-\vartheta(g\circ f) = 0$.
\end{proof}

Let
\begin{multline*} F^\lambda =
\sum_{i=0}^n t_i\tau_{\lambda,i}^* F^{\lambda(i)}
- \sum_{i=0}^n dt_i \wedge \lambda(0i)^{\sharp}D(\lambda(in)) + \\
\sum_{0\le i \le j \le n} (t_i dt_j -t_j dt_i) \wedge
\beta(\lambda(ij), \lambda(jn)) \ ,
\end{multline*}
where we regard $\lambda(0i)^{\sharp}D(\lambda(in))$ as an element
of
$\Gamma(X_{\lambda(n)};X(\lambda(0n))^{-1}\mathfrak{d}(\mathcal{A}^{\lambda(0)})
\otimes \mathcal{J}^{\lambda(n)})$, so that
\[
F^\lambda\in \Omega_n\otimes \Gamma(X_{\lambda(n)};
\Omega^\bullet_{X_{\lambda(n)}} \otimes
(X(\lambda(0n))^{-1}\mathfrak{d}(\mathcal{A}^{\lambda(0)}) \otimes
\mathcal{J}^{\lambda(n)})) \ .
\]

\begin{lemma}\label{lemma:Phi is delta F}
{~}
\begin{enumerate}
\item The collection $\{ F^{\lambda}\}$ defines an element $F\in \mathfrak{H}$

\item $\Phi = - \delta F$
\end{enumerate}
\end{lemma}
\begin{proof}
Direct calculation using Lemma \ref{lemma: restriction of forms on
simplex}.
\end{proof}

As before, $\overline{F}$ denotes the image of $F$ in
$\Tot(\real{X}; \Omega^\bullet_\real{X}\otimes
(\mathfrak{d}(\mathcal{A}) \otimes
\overline{\mathcal{J}}_\real{X}))$.

The unit map $\mathcal{O} \to \mathcal{A}$ (inclusion of the center)
induces the embedding
\begin{multline}\label{iincl}
\Tot(\Gamma(\real{X};\Omega^\bullet_\real{X}\otimes
\overline{\mathcal{J}}_\real{X}))
\hookrightarrow \\
\Tot(\Gamma(\real{X};\Omega^\bullet_\real{X}\otimes
(\mathfrak{d}(\mathcal{A}) \otimes \overline{\mathcal{J}}_\real{X}))
\end{multline}

\begin{lemma}\label{defw} The element
\[
-\nabla\overline{F} \in \\
\Tot(\Gamma(\real{X};\Omega^\bullet_\real{X} \otimes
(\mathfrak{d}(\mathcal{A}) \otimes
\overline{\mathcal{J}}_\real{X})))
\]
is  the image of a unique closed form
\[
\omega \in \Tot(\Gamma(\real{X}; \Omega^\bullet_\real{X}\otimes
\overline{\mathcal{J}}_\real{X}))
\]
under the inclusion \eqref{incl}. The cohomology class of $\omega$
is independent of the choices made in its construction.
\end{lemma}
\begin{proof}
Using $\Phi = -\delta F$ (Lemma \ref{lemma:Phi is delta F}) and the
fact that \eqref{differ} is a differential,  one obtains by direct
calculation the identity $\delta(\nabla F) + \nabla^2=0$. Note that
$\nabla^2 = \delta\theta$ for some $\theta \in \Tot
(\Gamma(\real{X}; \Omega^\bullet_\real{X} \otimes
\mathfrak{d}(\mathcal{A})))$. Hence $\nabla F + \theta$ is a central
element of $\Tot(\Gamma(\real{X}; \Omega^\bullet_\real{X} \otimes
(\mathfrak{d}(\mathcal{A}) \otimes \mathcal{J} )))$, and the first
statement follows.
\end{proof}

Following longstanding traditions we denote by $\iota_G$ the adjoint
action of $G\in \Tot(\Gamma(\real{X}; \Omega^\bullet_\real{X}
\otimes (\mathfrak{d}(\mathcal{A}) \otimes \mathcal{J} ))) \subset
\mathfrak{H}$.

\begin{prop}
\begin{equation}\label{conj}
\exp (\iota_{F})\circ (\delta + \nabla + \Phi) \circ\exp
(-\iota_{F}) = \delta + \nabla - \iota_{\nabla F}.
\end{equation}
\end{prop}
\begin{proof}
Analogous to the proof of Lemma 16 of \cite{bgnt2}. Details are left
to the reader.
\end{proof}

\subsubsection{}\label{subsubsection: step last}
It follows from \eqref{conj} that the map
\begin{equation}
\exp(- \iota_{F}) \colon  \mathfrak{H} \to \mathfrak{H}
\end{equation}
is an isomorphism of DGLA where the  differential in the left hand
side is given by $\delta + \nabla - \iota_{\nabla F}$
 and the differential in the right hand side is given by $\delta + \nabla +
 \Phi$, as in\eqref{differ}.

Consider the map
\begin{equation}\label{cotrace map}
\cotr\colon    \overline{C}^\bullet(\mathcal{J}^p)[1]\to
  C^\bullet(\mathcal{A}^p\otimes \mathcal{J}^p)[1]
\end{equation}
defined as follows:
\begin{equation}\label{define cotrace}
\cotr(D)(a_1\otimes j_1, \dots, a_n \otimes j_n) = a_0\ldots a_n
D(j_1, \ldots, j_n).
\end{equation}

The map \eqref{cotrace map} is a quasiisomorphism of DGLAs (cf.
\cite{Loday}, section 1.5.6; see also \cite{bgnt2} Proposition 14).

The maps \eqref{cotrace map} induce the map of graded Lie algebras
\begin{equation}\label{tot cotr}
\cotr \colon  \mathfrak{G}_\DR(\mathcal{J}) \to \mathfrak{H}
\end{equation}

\begin{lemma}\label{lemma: cotr dgla}
The map \eqref{tot cotr} is a quasiisomorphism of DGLA, where the
source and the target are equipped with the differentials
$\delta+\widetilde{\nabla}^{can} + \iota_\omega$ and $\delta +
\nabla - \iota_{\nabla F}$ respectively, i.e. \eqref{tot cotr} is a
morphism of DGLA
\[
\cotr \colon  \mathfrak{G}_\DR(\mathcal{J})_\omega \to \mathfrak{H}
\]
\end{lemma}

\subsubsection{}\label{subsubsection: case of mat}
Recall that in the section \ref{cosimplicial splitting} we
introduced the bundles $\mathcal{E}^p = \bigoplus_{j=0}^p
\mathcal{A}^p_{j0}$ over $X_p$. For $f\colon [p] \to [q]$ there is a
canonical isomorphism $X(f)^{-1}\mathcal{E}^p \cong
f^{\sharp}\mathcal{E}^q \otimes \mathcal{A}^q_{0 f(0)}$ which we use
to identify the former with the latter.

We make the following choices of additional structure:
\begin{enumerate}
\item for each $p = 0,1,2,\ldots$ an isomorphism $\sigma_\mathcal{E}^p \colon  \mathcal{E}^p \otimes \mathcal{J}^p
    \to \mathcal{J}(\mathcal{E}^p)$ such that $\sigma_{\mathcal{E}}^p(\mathcal{A}^p_{j0}\otimes \mathcal{J}^p)
    \subset \mathcal{J}(\mathcal{A}^p_{j0})$

\item for each $p = 0,1,2,\ldots$ a connection $\nabla^p_{\mathcal{E}}$ on $\mathcal{E}^p$

\item for every $f\colon [p] \to [q]$ an isomorphism $\sigma_f \colon  \mathcal{A}^q_{0f(0)} \otimes \mathcal{J}^q
    \to \mathcal{J}(\mathcal{A}^q_{0f(0)})$

\item for every $f\colon [p] \to [q]$ a connection $ \nabla_f$ on $\mathcal{A}^q_{0f(0)}$
\end{enumerate}
Let $\sigma^p\in\Isom_0(\mathcal{A}^p \otimes \mathcal{J}^p,
\mathcal{J}(\mathcal{A}^p))^{loc}$ denote the isomorphism induced by
$\sigma_\mathcal{E}^p$. Let $\nabla^p \in
\mathcal{C}(\mathcal{A}^p)^{loc}(X_p)$ denote the connection induced
by $\nabla^p_{\mathcal{E}}$. Let
\begin{eqnarray*}
F^p & \colon = &
(\sigma^p_{\mathcal{E}})^{-1}\circ\nabla^{can}_{\mathcal{E}^p} \circ
\sigma^p_{\mathcal{E}} - (\nabla^p_{\mathcal{E}}\otimes\id +
\id\otimes\nabla^{can}) \\
F_f & \colon = &
(\sigma_f)^{-1}\circ\nabla^{can}_{\mathcal{A}^q_{0f(0)}}\circ\sigma_f
- (\nabla_f\otimes\id + \id\otimes\nabla^{can})
\end{eqnarray*}
$F^p \in\Gamma(X_p;\Omega^1_{X_p}\otimes (\mathcal{A}^p\otimes
\mathcal{J}^p))$. We define $D(f)\in
\Gamma(X_q;X(f)^{-1}\mathfrak{d}(\mathcal{A}^p)\otimes
\mathcal{J}^q)$ by the equation
\begin{equation*}
X(f)^*\sigma_{\mathcal{E}}^p \circ \exp(D(f)) =f^{\sharp}
\sigma^q_\mathcal{E} \otimes \sigma_f
\end{equation*}

In \ref{cosimplicial splitting} we constructed the cosimplicial
gerbe $\mathcal{S}_\mathcal{A}$ on $X$ such that
$\mathcal{S}_\mathcal{A}^p$ is trivialized. Starting with the
choices of $\sigma_f$, $\nabla_f$ as above we calculate the
representative $\overline{B}$ of the characteristic class of
$\mathcal{S}_\mathcal{A}$ using \ref{subsubsection: char class
trivialized gerbe}. By Lemma \ref{lemma: cosimp splitting} the
collection of bundles $\mathcal{E}^p$ establishes an equivalence
between $\mathcal{S}_\mathcal{A}$ and the cosimplicial gerbe
$S(\mathcal{A})$ (of splittings of $\mathcal{A}$). Hence,
$\overline{B}$ is a representative of the characteristic class of
$\mathcal{S}_\mathcal{A}$.

In the notations of \ref{subsection: gerbes}, for $f \colon  [p]\to
[q]$, we have
\[
\mathcal{G}_f = \shIsom_0(\mathcal{A}^q_{0f(0)} \otimes
\mathcal{J}^q, \mathcal{J}(\mathcal{A}^q_{0f(0)}))
\]
and, under this identification, $\beta_f =
\sigma_f^{-1}\circ\nabla^{can}\circ\sigma_f$.

Then, $\beta(f, g)$ is uniquely determined by the equation
\begin{equation*}
(\sigma_g \otimes X(g)^*\sigma_f) = \sigma_{g\circ f}\circ \exp
\beta(f, g)
\end{equation*}
which implies
\begin{equation}\label{equation on beta}
\nabla^{can}\beta(f,g) =
\sigma_g^{-1}\circ\nabla^{can}_{\mathcal{S}_{\mathcal{A}g}}\circ
\sigma_g +
\sigma_f^{-1}\circ\nabla^{can}_{\mathcal{S}_{\mathcal{A}f}}\circ
\sigma_f - \sigma_{g\circ
f}^{-1}\circ\nabla^{can}_{\mathcal{S}_{\mathcal{A}g\circ f}}\circ
\sigma_{g\circ f}
 \ .
\end{equation}
Equations \eqref{equation on beta} and \eqref{equation on beta in
gerbes} show that $\beta(f,g)$ coincides with $\beta_{f,g}$ defined
by \eqref{definition of beta in gerbes}.

\begin{lemma}
The form $\overline{B}$ of \ref{subsubsection: char class
trivialized gerbe} coincides with the form $\omega$ of Lemma
\ref{defw}
\end{lemma}
\begin{proof}
For $f \colon  [p] \to [q]$ in $\Delta$ the following identities
hold:
\begin{align}
(\nabla^p\otimes\id + \id\otimes\nabla^{can}) \overline{F}^p=0\\
\nabla^{can} \overline{F}_f=0\\
\nabla^{can}
\overline{D}(f)+X(f)^*\overline{F}^p=f^{\sharp}\overline{F}^q
+\overline{F}_f
\end{align}
Using these identities we  compute  :
\begin{multline*}
\nabla^{\lambda} \overline{F}^{\lambda}=\sum_{i=0}^n t_i \cdot
\tau_{\lambda,i}^*(\nabla^{\lambda(i)}\otimes \id +\id \otimes
\nabla^{can})\overline{F}^{\lambda(i)}+
\\
 \sum_{i=0}^n dt_i\wedge \tau_{\lambda,i}^*\overline{F}^{\lambda(i)} +\sum_{i=0}^n
dt_i \wedge \lambda(0i)^{\sharp}\nabla^{can} (\overline{D}(\lambda(in))) + \\
\widetilde{\nabla}^{can}(\sum_{0\le i \le j \le n} (t_i dt_j -t_j
dt_i)\wedge\bar{\beta}(\lambda(ij),
\lambda(jn))) = \\
\sum_{i=0}^n dt_i\wedge\lambda(0i)^{\sharp}(\nabla^{can}
(\overline{D}(\lambda(in))+X(\lambda(in))^*\overline{F}^{\lambda(i)}) + \\
\widetilde{\nabla}^{can}(\sum_{0\le i \le j \le n} (t_i dt_j -t_j
dt_i)\wedge\bar{\beta}(\lambda(ij),
\lambda(jn)))= \\
\sum_{i=0}^n
dt_i\wedge\lambda(0i)^{\sharp}(\lambda(in)^{\sharp}\overline{F}^n
+\overline{F}_{\lambda(in)}) + \\
\widetilde{\nabla}^{can}(\sum_{0\le i \le j \le n} (t_i dt_j -t_j
dt_i)\wedge\bar{\beta}(\lambda(ij),
\lambda(jn))) = \\
(\sum_{i=0}^n dt_i)\wedge \lambda(0n)^{\sharp}\overline{F}^n
+\sum_{i=0}^n dt_i \wedge \overline{F}_{\lambda(in)} +
\\
\widetilde{\nabla}^{can}(\sum_{0\le i \le j \le n} (t_i dt_j -t_j
dt_i)\wedge\bar{\beta}(\lambda(ij)))=\\
\sum_{i=0}^n dt_i\wedge\overline{F}_{\lambda(in)}
+\widetilde{\nabla}^{can}(\sum_{0\le i \le j \le n} (t_i dt_j -t_j
dt_i)\wedge\bar{\beta}(\lambda(ij)))
\end{multline*}
The result is identical to the formula \eqref{B for gerbes}.
\end{proof}

\subsubsection{}\label{subsubsection: the map}
In what follows we will denote a choice of auxiliary data as in
\ref{subsubsection: case of mat} by $\varpi$.

Given a choice of auxiliary data $\varpi$ we denote by
$\overline{B}_\varpi$ the corresponding characteristic form, by
$\Sigma_\varpi$ the corresponding map \eqref{The ism}, etc., and by
\[
\Upsilon_\varpi \colon
\mathfrak{G}_\DR(\mathcal{J})_{\overline{B}_\varpi} \to
\mathfrak{G}_\DR(\mathcal{J}(\mathcal{A}))
\]
the quasiisomorphism of DGLA defined as the composition
\[
\mathfrak{G}_\DR(\mathcal{J})_{\overline{B}_\varpi}
\xrightarrow{\cotr} \mathfrak{H} \xrightarrow{\exp\iota_{F_\varpi}}
\mathfrak{H} \xrightarrow{\Sigma_\varpi}
\mathfrak{G}_\DR(\mathcal{J}(\mathcal{A}))
\]
To finish the construction, we will ``integrate" the ``function"
$\varpi\mapsto\Upsilon_\varpi$ over the ``space" of choices of
auxiliary data in order to produce
\begin{itemize}
\item the DGLA $\mathfrak{G}_\DR(\mathcal{J})_{[S(\mathcal{A})]}$

\item for each choice of auxiliary data $\varpi$ a quasiisomorphism
\[
\pr_\varpi \colon  \mathfrak{G}_\DR(\mathcal{J})_{[S(\mathcal{A})]}
\to \mathfrak{G}_\DR(\mathcal{J})_{\overline{B}_\varpi}
\]

\item the morphism in the derived category of DGLA
\[
\overline{\Upsilon} \colon
\mathfrak{G}_\DR(\mathcal{J})_{[S(\mathcal{A})]} \to
\mathfrak{G}_\DR(\mathcal{J}(\mathcal{A}))
\]
\end{itemize}
such that the morphism in the derived category given by the
composition
\[
\mathfrak{G}_\DR(\mathcal{J})_{\overline{B}_\varpi}
\xrightarrow{(\pr_\varpi)^{-1}}
\mathfrak{G}_\DR(\mathcal{J})_{[S(\mathcal{A})]}
\xrightarrow{\overline{\Upsilon}}
\mathfrak{G}_\DR(\mathcal{J}(\mathcal{A}))
\]
coincides with $\Upsilon_\varpi$.

\subsubsection{Integration}
To this end, for a cosimplicial matrix Azumaya algebra $\mathcal{A}$
on $X$ (respectively, a cosimplicial $\mathcal{O}^\times_X$-gerbe
$\mathcal{S}$ as in \ref{subsubsection: cosimp gerbe}) let
$\AuxData(\mathcal{A})$ (respectively, $\AuxData(\mathcal{S})$)
denote the category with objects choices of auxiliary data (1) --
(4) as in \ref{subsubsection: case of mat} (respectively,
(\texttt{i})--(\texttt{i\!i\!i}) as in \ref{dlog-gerbe} with
$\mathcal{G} = \nabla^{can}\log\vac(\mathcal{S})$) and one-element
morphism sets. Thus, $\AuxData(\mathcal{A})$ (respectively,
$\AuxData(\mathcal{S})$) is a groupoid such that every object is
both initial and final.

For a cosimplicial matrix Azumaya algebra $\mathcal{A}$ we have the
functor
\begin{equation}\label{proj aux data}
\pi \colon  \AuxData(\mathcal{A}) \to \AuxData(S(\mathcal{A}))
\end{equation}
which associates to $\varpi \in \AuxData(\mathcal{A})$ (in the
notations of \ref{subsubsection: the map}) the auxiliary data as in
\ref{subsubsection: char class trivialized gerbe}, items (1)--(3).
Here we use the equivalence of Lemma \ref{lemma: cosimp splitting}.

For a category $\mathcal{C}$ we denote by $\Sing(\mathcal{C})$
denote the category, whose objects are ``singular simplices" $\mu
\colon  [m] \to \mathcal{C}$. For $\mu \colon [m] \to \mathcal{C}$,
$\nu \colon  [n] \to \mathcal{C}$, a morphism $f \colon \mu \to \nu$
is an \emph{injective} (on objects) morphism $f \colon [m] \to [n]$
such that $\mu = \nu\circ f$. The functor \eqref{proj aux data}
induces the functor
\[
\pi \colon \Sing(\AuxData(\mathcal{A})) \to
\Sing(\AuxData(S(\mathcal{A})))
\]

For $\mu \colon [m] \to \AuxData(\mathcal{A})$ (respectively, $\mu
\colon  [m] \to \AuxData(\mathcal{S})$) , $0\leq i \leq m$, the
choice of auxiliary data $\mu(i)$ consists of
$\sigma^p_\mathcal{E}(\mu(i))$, $\nabla^p_\mathcal{E}(\mu(i))$,
$\sigma_f(\mu(i))$, $\nabla_f(\mu(i))$ (respectively,
$\eth^p(\mu(i))$, $B^p(\mu(i))$, $\beta_f(\mu(i))$) for all objects
$[p]$ and morphisms $f$ in $\Delta$.

In either case, let $X(\mu) \colon = \Delta^m\times X$. Then,
$X(\mu)$ is an \'etale simplicial manifold. Let $\mathcal{A}(\mu) =
\pr_X^*\mathcal{A}$ (respectively, $\mathcal{S}(\mu) =
\pr_X^*\mathcal{S}$). Then, $\mathcal{A}(\mu)$ is a cosimplicial
matrix $\mathcal{O}_{X(\mu)}$-Azumaya algebra on $X(\mu)$
(respectively, a cosimplicial $\mathcal{O}_{X(\mu)}^\times$-gerbe).

For $\mu \colon [m] \to \AuxData(\mathcal{A})$ let
$\sigma^p_\mathcal{E}(\mu)$ (respectively, $\nabla^p_\mathcal{E}$,
$\sigma_f(\mu)$, $\nabla_f(\mu)$) denote the convex combination of
$\pr_X^*\sigma^p_\mathcal{E}(\mu(i))$ (respectively,
$\pr_X^*\nabla^p_\mathcal{E}(\mu(i))$, $\pr_X^*\sigma_f(\mu(i))$,
$\pr_X^*\nabla_f(\mu(i))$), $i = 0,\ldots,m$. The collection
consisting of $\sigma^p_\mathcal{E}$, $\nabla^p_\mathcal{E}$,
$\sigma_f(\mu)$, $\nabla_f(\mu)$ for all objects $[p]$ and morphisms
$f$ in $\Delta$ constitutes a choice of auxiliary data, denoted
$\widetilde{\mu}$.

Similarly, for $\mu \colon  [m] \to \AuxData(\mathcal{S})$ one
defines $\widetilde{\mu}$ as the collection of auxiliary data
$\eth^p(\mu)$ (respectively, $B^p(\mu)$, $\beta_f(\mu)$) consisting
of convex combinations of $\eth^p(\mu(i))$ (respectively,
$B^p(\mu(i))$, $\beta_f(\mu(i))$), $i = 0,\ldots,m$. The
construction of \ref{subsection: gerbes} apply with $X \colon =
X(\mu)$, $\nabla^{can} \colon = \pr^*_X\nabla^{can}$, $\mathcal{S}
\colon = \mathcal{S}(\mu)$ yielding the cocycle
$\overline{B}_{\widetilde{\mu}} \in \Gamma(\real{X(\mu)};
\DR(\overline{\mathcal{J}}_{\real{X(\mu)}}))$. A morphism $f \colon
\mu \to \nu$ in $\Sing(\AuxData(\mathcal{S}))$ induces a
quasiisomorphism of complexes
\[
f^* \colon  \Gamma(\real{X(\nu)};
\DR(\overline{\mathcal{J}}_{\real{X(\nu)}})) \to
\Gamma(\real{X(\mu)}; \DR(\overline{\mathcal{J}}_{\real{X(\mu)}})).
\]
Moreover, we have $f^*(\overline{B}_{\widetilde{\nu}}) =
\overline{B}_{\widetilde{\mu}}$.

Hence, as explained in \ref{subsubsection: Twisted DGLA of jets},
for $\mu \colon  [m] \to \AuxData(\mathcal{S})$, we have the DGLA
$\mathfrak{G}_\DR(\mathcal{J}_{X(\mu)})_{\overline{B}_{\widetilde{\mu}}}$.
Moreover, a morphism $f \colon \mu \to \nu$ in
$\Sing(\AuxData(\mathcal{S}))$ induces a quasiisomorphism of DGLA
\[
f^* \colon
\mathfrak{G}_\DR(\mathcal{J}_{X(\nu)})_{\overline{B}_{\widetilde{\nu}}}
\to
\mathfrak{G}_\DR(\mathcal{J}_{X(\mu)})_{\overline{B}_{\widetilde{\mu}}}
.
\]
Let
\begin{equation}\label{defgs}
\mathfrak{G}_\DR(\mathcal{J}_X)_{[\mathcal{S}]} \colon = \liminv_\mu
\mathfrak{G}_\DR(\mathcal{J}_{X(\mu)})_{\overline{B}_{\widetilde{\mu}}}
\end{equation}
where the limit is taken over the category
$\Sing(\AuxData(\mathcal{S}))$.

For $\mu \colon  [m] \to \AuxData(\mathcal{A})$ the constructions of
\ref{subsubsection: step one} -- \ref{subsubsection: step last}
apply with $X \colon = X(\mu)$, $\mathcal{J} \colon =
\mathcal{J}_{X(\mu)}$, $\nabla^{can} \colon = \pr^*_X\nabla^{can}$,
$\mathcal{A} \colon = \mathcal{A}(\mu)$ and the choice of auxiliary
data $\widetilde{\mu}$ yielding the quasiisomorphism of DGLA
\[
\Upsilon_{\widetilde{\mu}} \colon
\mathfrak{G}_\DR(\mathcal{J}_{X(\mu)})_{\overline{B}_{\widetilde{\pi(\mu)}}}
\to \mathfrak{G}_\DR(\mathcal{J}_{X(\mu)}(\mathcal{A}(\mu)))
\]


For a morphism $f \colon \mu \to \nu$ in
$\Sing(\AuxData(\mathcal{A}))$, the diagram
\[
\begin{CD}
\mathfrak{G}_\DR(\mathcal{J}_{X(\nu)})_{\overline{B}_{\widetilde{\pi(\nu)}}}
@>{\Upsilon_{\widetilde{\nu}}}>>
\mathfrak{G}_\DR(\mathcal{J}_{X(\nu)}(\mathcal{A}(\nu))) \\
@V{f^*}VV @VV{f^*}V \\
\mathfrak{G}_\DR(\mathcal{J}_{X(\mu)})_{\overline{B}_{\widetilde{\pi(\mu)}}}
@>{\Upsilon_{\widetilde{\mu}}}>>
\mathfrak{G}_\DR(\mathcal{J}_{X(\mu)}(\mathcal{A}(\mu)))
\end{CD}
\]
is commutative. Thus, we have two functors,
$\Sing(\AuxData(\mathcal{A}))^{op} \to \mathrm{DGLA}$, namely,
\[
\mu \mapsto
\mathfrak{G}_\DR(\mathcal{J}_{X(\mu)})_{\overline{B}_{\widetilde{\pi(\mu)}}}
\]
and
\[
\mu \mapsto \mathfrak{G}_\DR(\mathcal{J}_{X(\mu)}(\mathcal{A}(\mu)))
,
\]
and a morphism of such, namely, $\mu \mapsto
\Upsilon_{\widetilde{\mu}}$. Since the first functor factors through
$\Sing(\AuxData(S(\mathcal{A})))$ there is a canonical morphism of
DGLA
\begin{equation}\label{restriction invlim}
\liminv_\mu
\mathfrak{G}_\DR(\mathcal{J}_{X(\mu)})_{\overline{B}_{\widetilde{\pi(\mu)}}}
\to \mathfrak{G}_\DR(\mathcal{J}_X)_{[S(\mathcal{A})]} .
\end{equation}
On the other hand, $\Upsilon$ induces the morphism
\begin{equation}\label{lim ups}
\liminv_\mu \Upsilon_{\widetilde{\mu}}\colon \liminv_\mu
\mathfrak{G}_\DR(\mathcal{J}_{X(\mu)})_{\overline{B}_{\widetilde{\pi(\mu)}}}
\to \liminv_\mu
\mathfrak{G}_\DR(\mathcal{J}_{X(\mu)}(\mathcal{A}(\mu))) .
\end{equation}

\begin{lemma}
The morphisms \eqref{restriction invlim} and \eqref{lim ups} are
quasiisomorphisms.
\end{lemma}

For each $\mu\in\Sing(\AuxData(\mathcal{A}))$ we have the map
\[
\pr_X^*(\mu) \colon  \mathfrak{G}_\DR(\mathcal{J}_{X}(\mathcal{A}))
\to \mathfrak{G}_\DR(\mathcal{J}_{X(\mu)}(\mathcal{A}(\mu))) .
\]
Moreover, for any morphism $f \colon  \mu \to \nu$ in
$\Sing(\AuxData(\mathcal{A}))$, the diagram
\[
\begin{CD}
\mathfrak{G}_\DR(\mathcal{J}_X(\mathcal{A}))
@>{\pr_X^*(\nu)}>> \mathfrak{G}_\DR(\mathcal{J}_{X(\nu)}(\mathcal{A}(\nu))) \\
@V{\id}VV @VV{f^*}V \\
\mathfrak{G}_\DR(\mathcal{J}_X(\mathcal{A})) @>{\pr_X^*(\mu)}>>
\mathfrak{G}_\DR(\mathcal{J}_{X(\mu)}(\mathcal{A}(\mu)))
\end{CD}
\]
is commutative. Therefore, we have the map
\begin{equation}\label{pull-back lim}
\pr_X^* \colon  \mathfrak{G}_\DR(\mathcal{J}_X(\mathcal{A})) \to
\liminv_\mu \mathfrak{G}_\DR(\mathcal{J}_{X(\mu)}(\mathcal{A}(\mu)))
\end{equation}
Note that, for $\nu \colon  [0]\to \AuxData(\mathcal{A})$, we have
$\mathfrak{G}_\DR(\mathcal{J}_{X(\nu)}(\mathcal{A}(\nu))) =
\mathfrak{G}_\DR(\mathcal{J}_X(\mathcal{A}))$, and the composition
\[
\mathfrak{G}_\DR(\mathcal{J}_X(\mathcal{A})) \xrightarrow{\pr_X^*}
\liminv_\mu \mathfrak{G}_\DR(\mathcal{J}_{X(\mu)}(\mathcal{A}(\mu)))
\to \mathfrak{G}_\DR(\mathcal{J}_{X(\nu)}(\mathcal{A}(\nu)))
\]
is equal to the identity.

\begin{lemma}
For each $\nu \colon  [0]\to \AuxData$ the morphism in the derived
category induced by the canonical map
\[
\liminv_\mu \mathfrak{G}_\DR(\mathcal{J}_{X(\mu)}(\mathcal{A}(\mu)))
\to \mathfrak{G}_\DR(\mathcal{J}_{X(\nu)}(\mathcal{A}(\nu)))
\]
is inverse to $\pr_X^*$.
\end{lemma}

\begin{cor}
The morphism
\[
\liminv_\mu \mathfrak{G}_\DR(\mathcal{J}_{X(\mu)}(\mathcal{A}(\mu)))
\to \mathfrak{G}_\DR(\mathcal{J}_{X(\nu)}(\mathcal{A}(\nu))) =
\mathfrak{G}_\DR(\mathcal{J}_X(\mathcal{A}))
\]
in derived category does not depend on $\nu \colon  [0]\to
\AuxData$.
\end{cor}

Let
\begin{equation}\label{cotr ind of choice}
\overline{\Upsilon} \colon
\mathfrak{G}_\DR(\mathcal{J})_{[S(\mathcal{A})]} \to
\mathfrak{G}_\DR(\mathcal{J}_X(\mathcal{A}))
\end{equation}
denote the morphism in the derived category represented by
\begin{multline*}
\mathfrak{G}_\DR(\mathcal{J})_{[S(\mathcal{A})]}
\xleftarrow{\eqref{restriction invlim}}
\liminv_\mu \mathfrak{G}_\DR(\mathcal{J}_{X(\mu)})_{\overline{B}_{\widetilde{\pi(\mu)}}} \xrightarrow{\eqref{lim ups}} \\
\liminv_\mu \mathfrak{G}_\DR(\mathcal{J}_{X(\mu)}(\mathcal{A}(\mu)))
\to \mathfrak{G}_\DR(\mathcal{J}_{X(\nu)}(\mathcal{A}(\nu))) =
\mathfrak{G}_\DR(\mathcal{J}_X(\mathcal{A}))
\end{multline*}
for any $\nu \colon  [0]\to \AuxData$.

\section{Applications to \'etale groupoids}\label{Applications to etale groupoids}

\subsection{Algebroid stacks}
In this section we review the notions of algebroid stack and twisted
form. We also define the notion of descent datum and relate it with
algebroid stacks.

\subsubsection{Algebroids}\label{subsubsection: algebroids}
For a category $\mathcal{C}$ we denote by $i\mathcal{C}$ the
subcategory of isomorphisms in $\mathcal{C}$; equivalently,
$i\mathcal{C}$ is the maximal subgroupoid in $\mathcal{C}$.

Suppose that $R$ is a commutative $k$-algebra.

\begin{definition}
An \emph{$R$-algebroid} is a \emph{nonempty} $R$-linear category
$\mathcal{C}$ such that the groupoid $i\mathcal{C}$ is connected
\end{definition}

Let $\Algd_R$ denote the $2$-category of $R$-algebroids (full
$2$-subcategory of the $2$-category of $R$-linear categories).

Suppose that $A$ is an $R$-algebra. The $R$-linear category with one
object and morphisms $A$ is an $R$-algebroid denoted $A^+$.

Suppose that $\mathcal{C}$ is an $R$-algebroid and $L$ is an object
of $\mathcal{C}$. Let $A = \End_\mathcal{C}(L)$. The functor
$A^+\to\mathcal{C}$ which sends the unique object of $A^+$ to $L$ is
an equivalence.

Let $\Alg^2_R$ denote the $2$-category of with
\begin{itemize}
\item objects $R$-algebras
\item $1$-morphisms homomorphism of $R$-algebras
\item $2$-morphisms $\phi\to\psi$, where $\phi,\psi \colon A\to B$ are two $1$-morphisms are elements $b\in B$ such
    that $b\cdot\phi(a) = \psi(a)\cdot b$ for all $a\in A$.
\end{itemize}
It is clear that the $1$- and the $2$- morphisms in $\Alg^2_R$ as
defined above induce $1$- and $2$-morphisms of the corresponding
algebroids under the assignment $A \mapsto A^+$. The structure of a
$2$-category on $\Alg^2_R$ (i.e. composition of $1$- and $2$-
morphisms) is determined by the requirement that the assignment
$A\mapsto A^+$ extends to an embedding $(.)^+ \colon \Alg^2_R\to
\Algd_R$.

Suppose that $R\to S$ is a morphism of commutative $k$-algebras. The
assignment $A\to A\otimes_R S$ extends to a functor $(.)\otimes_R S
\colon \Alg^2_R\to \Alg^2_S$.

\subsubsection{Algebroid stacks}
\begin{definition}
A stack in $R$-linear categories $\mathcal{C}$ on a space $Y$ is an
\emph{$R$-algebroid stack} if it is locally nonempty and locally
connected by isomorphisms, i.e. the stack $\widetilde{i\mathcal{C}}$
is a gerbe.
\end{definition}

\begin{example}
Suppose that $\mathcal{A}$ is a sheaf of $R$-algebras on $Y$. The
assignment $X\supseteq U\mapsto\mathcal{A}(U)^+$ extends in an
obvious way to a prestack in $R$-algebroids denoted $\mathcal{A}^+$.
The associated stack $\widetilde{\mathcal{A}^+}$ is canonically
equivalent to the stack of locally free $\mathcal{A}^{op}$-modules
of rank one. The canonical morphism
$\mathcal{A}^+\to\widetilde{\mathcal{A}^+}$ sends the unique object
of $\mathcal{A}^+$ to the free module of rank one.
\end{example}

$1$-morphisms and $2$-morphisms of $R$-algebroid stacks are those of
stacks in $R$-linear categories. We denote the $2$-category of
$R$-algebroid stacks by $\AlgStack_R(Y)$.

Suppose that $G$ is an \'etale category.

\begin{definition}
An $R$-algebroid stack on $G$ is a stack $\underline{\mathcal{C}} =
(\mathcal{C},\mathcal{C}_{01},\mathcal{C}_{012})$ on $G$ such that
$\mathcal{C}\in\AlgStack_R(N_0G)$, $\mathcal{C}_{01}$ is a
$1$-morphism in $\AlgStack_R(N_1G)$ and $\mathcal{C}_{012}$ is a
$2$-morphism in $\AlgStack_R(N_2G)$.
\end{definition}

\begin{definition}
A $1$-morphism $\underline{\phi} = (\phi_0, \phi_{01}) \colon
\underline{\mathcal{C}} \to \underline{\mathcal{D}}$ of
$R$-algebroid stacks on $G$ is a morphism of stacks on $G$ such that
$\phi_0$ is a $1$-morphism in $\AlgStack_R(N_0G)$ and $\phi_{01}$ is
a $2$-morphism in $\AlgStack(N_1G)$.
\end{definition}

\begin{definition}
A $2$-morphism $b \colon \underline{\phi} \to \underline{\psi}$
between $1$-morphisms $\underline{\phi}, \underline{\psi}
\colon\underline{\mathcal{C}} \to \underline{\mathcal{D}}$ is a
$2$-morphism $b \colon \phi_0 \to \psi_0$ in $\AlgStack_R(N_0G)$.
\end{definition}

We denote the $2$-category of $R$-algebroid stacks on $G$ by
$\AlgStack_R(G)$.

\subsection{Base change for algebroid stacks}
For an $R$-linear category $\mathcal{C}$ and homomorphism of
commutative $k$-algebras $R\to S$ we denote by $\mathcal{C}\otimes_R
S$ the category with the same objects as $\mathcal{C}$ and morphisms
defined by $\Hom_{\mathcal{C}\otimes_R S}(A,B) =
\Hom_\mathcal{C}(A,B)\otimes_R S$.

For an $R$-algebra $A$ the categories $(A\otimes_R S)^+$ and
$A^+\otimes_R S$ are canonically isomorphic.

For a prestack $\mathcal{C}$ in $R$-linear categories we denote by
$\mathcal{C}\otimes_R S$ the prestack associated to the fibered
category $U\mapsto\mathcal{C}(U)\otimes_R S$.

For $U\subseteq X$, $A,B\in\mathcal{C}(U)$, there is an isomorphism
of sheaves $\shHom_{\mathcal{C}\otimes_R S}(A,B) =
\shHom_\mathcal{C}(A,B)\otimes_R S$.

The proof of the following lemma can be found in \cite{bgnt3} (Lemma
4.13 of loc. cit.)
\begin{lemma}
Suppose that $\mathcal{A}$ is a sheaf of $R$-algebras on a space $Y$
and $\mathcal{C}\in \AlgStack_R(Y)$ is an $R$-algebroid stack.
\begin{enumerate}
\item $(\widetilde{\mathcal{A}^+}\otimes_R S)\widetilde{~~}$ is an algebroid stack
equivalent to $\widetilde{(\mathcal{A}\otimes_R S)^+}$ .

\item $\widetilde{\mathcal{C}\otimes_R S}$ is an algebroid stack.
\end{enumerate}
\end{lemma}

The assignment $\mathcal{C} \mapsto \widetilde{\mathcal{C}\otimes_R
S}$ extends to a functor denoted
\[
(.)\widetilde{\otimes}_R S \colon \AlgStack_R(Y) \to \AlgStack_S(Y)
\]
and, hence, for an \'etale category $G$, to a functor
\[
(.)\widetilde{\otimes}_R S \colon \AlgStack_R(G) \to \AlgStack_S(G)
\]

There is a canonical $R$-linear morphism $\mathcal{C} \to
\mathcal{C}\widetilde{\otimes}_R S$ (where the target is considered
as a stack in $R$-linear categories by restriction of scalars) which
is characterized by an evident universal property.

\subsection{The category of trivializations}
Let $\Triv_R(G)$ denote the $2$-category with
\begin{itemize}
\item objects the pairs $(\underline{\mathcal{C}}, L)$, where
$\underline{\mathcal{C}}$ is an $R$-algebroid stack on $G$ such that
$\mathcal{C}(N_0G)\neq\emptyset$, and $L\in\mathcal{C}(N_0G)$

\item $1$-morphism $(\underline{\mathcal{C}}, L) \to (\underline{\mathcal{D}}, M)$ the pairs
    $(\underline{\phi},\phi_\tau)$, where $\underline{\phi}\colon\underline{\mathcal{C}} \to
    \underline{\mathcal{D}}$ is a morphism in $\AlgStack_R(G)$ and $\phi_\tau \colon \phi_0(L) \to M$ is an
    \emph{isomorphism} in $\mathcal{D}(N_0G)$.

\item $2$-morphisms $(\underline{\phi},\phi_\tau) \to
(\underline{\psi},\psi_\tau)$ are the $2$-morphisms
$\underline{\phi} \to \underline{\psi}$.
\end{itemize}

The composition of $1$-morphisms is defined by
$(\underline{\phi},\phi_\tau)\circ (\underline{\psi},\psi_\tau) =
(\underline{\phi}\circ\underline{\psi}, \phi_\tau\circ
\phi_0(\psi_\tau))$.

The assignment $(\underline{\mathcal{C}}, L) \mapsto
\underline{\mathcal{C}}$, $(\underline{\phi},\phi_\tau) \mapsto
\underline{\phi}$, $b \mapsto b$ extends to a functor
\begin{equation}\label{forget triv}
\Triv_R(G) \to \AlgStack_R(G)
\end{equation}

For a homomorphism of algebras $R\to S$ and
$(\underline{\mathcal{C}}, L)\in \Triv_R(G)$ we denote by
$(\underline{\mathcal{C}}, L)\widetilde{\otimes}_R S$ the pair which
consists of $\underline{\mathcal{C}}\widetilde{\otimes}_R S \in
\AlgStack_S(G)$ and the image of $L$, denoted $L\otimes_R S$, in
$\underline{\mathcal{C}}\widetilde{\otimes}_R S$.

It is clear that the forgetful functors \eqref{forget triv} commute
with the base change functors.

\subsubsection{Algebroid stacks from cosimplicial matrix algebras}\label{subsubsection: triv from cma}
Suppose that $G$ is an \'etale category and $\mathcal{A}$ is a
cosimplicial matrix $R$-algebra on $NG$.

Let
\begin{equation}\label{triv:algd}
\mathcal{C} = \widetilde{(\mathcal{A}^{0 op})^+} \ .
\end{equation}
In other words, $\mathcal{C}$ is the stack of locally free
$\mathcal{A}$-modules of rank one. There is a canonical isomorphism
$\mathcal{C}^\s{1}_i = \widetilde{(\mathcal{A}^{1 op}_{ii})^+}$.

Let
\begin{equation}\label{triv:functor}
\mathcal{C}_{01} = \mathcal{A}^1_{01} \otimes_{\mathcal{A}^1_{11}}
(.) \colon \mathcal{C}^\s{1}_1 \to \mathcal{C}^\s{1}_0 .
\end{equation}
The multiplication pairing
$\mathcal{A}^2_{01}\otimes\mathcal{A}^2_{12} \to \mathcal{A}^2_{02}$
and the isomorphisms $\mathcal{A}^2_{ij} =
(\mathcal{A}^1_{01})^\s{2}_{ij}$ determine the morphism
\begin{equation}\label{triv:conv}
\mathcal{C}_{012} \colon
\mathcal{C}^\s{2}_{01}\circ\mathcal{C}^\s{2}_{12} =
\mathcal{A}^2_{01} \otimes_{\mathcal{A}^2_{11}} \mathcal{A}^2_{12}
\otimes_{\mathcal{A}^2_{22}} (.) \to \mathcal{A}^2_{02}
\otimes_{\mathcal{A}^2_{22}} (.) = \mathcal{C}^\s{2}_{02}
\end{equation}
Let $\underline{\mathcal{C}} =
(\mathcal{C},\mathcal{C}_{01},\mathcal{C}_{012})$.

\begin{lemma}
The triple $\underline{\mathcal{C}}$ defined by \eqref{triv:algd},
\eqref{triv:functor}, \eqref{triv:conv} is an algebroid stack on
$G$.
\end{lemma}

We denote by $\stack(\mathcal{A})$ the algebroid stack
$\underline{\mathcal{C}}$ associated to the cosimplicial matrix
algebra $\mathcal{A}$ by the above construction.

Suppose that $\mathcal{A}$ and $\mathcal{B}$ are cosimplicial matrix
$R$-algebras on $NG$ and $F \colon \mathcal{A} \to \mathcal{B}$ is a
$1$-morphism of such. Let $\underline{\mathcal{C}} =
\stack(\mathcal{A})$, $\underline{\mathcal{D}} =
\stack(\mathcal{B})$. The map $F^0 \colon \mathcal{A}^0 \to
\mathcal{B}^0$ of sheaves of $R$-algebras on $N_0G$ induces the
morphism $\phi_0 \colon \mathcal{C} \to \mathcal{D}$ of
$R$-algebroid stacks on $N_0G$. Namely, we have
\begin{equation}\label{triv:1-mor-0}
\phi_0 = \mathcal{B}^0 \otimes_{\mathcal{A}^0} (.)
\end{equation}

The map $F^1 \colon \mathcal{A}^1 \to \mathcal{B}^1$ restricts to
the map $F^1_{01} \colon \mathcal{A}^1_{01} \to \mathcal{B}^1_{01}$
which induces the map
\[
\mathcal{B}^1_{00} \otimes_{\mathcal{A}^1_{00}} \mathcal{A}^1_{01}
\to \mathcal{B}^1_{01} = \mathcal{B}^1_{01}
\otimes_{\mathcal{B}^1_{11}} \mathcal{B}^1_{11}
\]
of $\mathcal{B}^1_{00}\otimes(\mathcal{A}^1_{11})^{op}$-modules,
hence the $2$-morphism
\begin{equation}\label{triv:1-mor-01}
\phi_{01} \colon \phi^\s{1}_0 \circ \mathcal{C}_{01} \to
\mathcal{D}_{01} \circ \phi^\s{1}_1
\end{equation}
Let $\underline{\phi} = (\phi_0,\phi_{01})$.

\begin{lemma}
The pair $\underline{\phi}$ defined by \eqref{triv:1-mor-0},
\eqref{triv:1-mor-01} is a $1$-morphism in $\AlgStack_R(G)$.
\end{lemma}

For a $1$-morphism of cosimplicial matrix $R$-algebras $F \colon
\mathcal{A} \to \mathcal{B}$ on $NG$ we denote by $\stack(F) \colon
\stack(\mathcal{A}) \to \stack(\mathcal{B})$ the $1$-morphism in
$\AlgStack_R(G)$ given by the above construction.

Suppose that $F_1, F_2 \colon \mathcal{A} \to \mathcal{B}$ are
$1$-morphisms of cosimplicial matrix $R$-algebras on $NG$ and $b
\colon F_1 \to F_2$ is a $2$-morphism of such. The $2$-morphism $b^0
\colon F_1^0 \to F_2^0$ induces the $2$-morphism of functors
$\widetilde{(\mathcal{A}^{0 op})^+} \to \widetilde{(\mathcal{B}^{0
op})^+}$.

\begin{lemma}
The $2$-morphism $b^0$ is a $2$-morphism in $\AlgStack_R(G)$.
\end{lemma}

For a $2$-morphism $b \colon F_1 \to F_2$ we denote by $\stack(b)
\colon \stack(F_1) \to \stack(F_2)$ the corresponding $2$-morphism
in $\AlgStack_R(G)$ given by the above construction.

We denote the canonical trivialization of the algebroid stack
$\stack(\mathcal{A})$ by $\one_\mathcal{A}$. Let
$\triv(\mathcal{A})$ denote the object of $\Triv_R(G)$ given by the
pair $(\stack(\mathcal{A}), \one_\mathcal{A})$.

If $F \colon \mathcal{A} \to \mathcal{B}$ is a $1$-morphism in
$\CosMatAlg_R(NG)$, then $\stack(F)(\one_\mathcal{A}) =
\one_\mathcal{B}$. Let $\stack(F)_\tau$ denote this identification;
let $\triv(F) = (\stack(F),\stack(F)_\tau)$.

For a $2$-morphism $b$ in $\CosMatAlg_R(G)$ let $\triv(b) =
\stack(b)$.

\begin{prop}\label{prop: cosmatalg to algstack}
{~} \begin{enumerate}
\item The assignments $\mathcal{A} \mapsto \stack(\mathcal{A})$, $F
\mapsto \stack(F)$, $b \mapsto \stack(b)$ define a functor
\[
\stack \colon \CosMatAlg_R(NG) \to \AlgStack_R(G) \ .
\]

\item The assignments $\mathcal{A} \mapsto
\triv(\mathcal{A})$, $F \mapsto \triv(F)$, $b \mapsto \triv(b)$
define a functor
\[
\triv \colon \CosMatAlg_R(NG) \to \Triv_R(G) \ .
\]
which lifts the functor $\stack$.
\end{enumerate}
\end{prop}

\subsubsection{Base change for $\stack$ and $\triv$}
A morphism $f \colon R \to S$ of commutative $k$-algebras induces
the $R$-linear morphism $\id\otimes f \colon \mathcal{A} \to
\mathcal{A}\otimes_R S$ in $\CosMatAlg_R(NG)$, hence, the morphism
$\stack(\id\otimes f) \colon \stack(\mathcal{A}) \to
\stack(\mathcal{A}\otimes_R S)$ in $\AlgStack_R(G)$. By the
universal property of the base change morphism, the latter factors
canonically through a unique morphism
$\stack(\mathcal{A})\widetilde{\otimes}_R S \to
\stack(\mathcal{A}\otimes_R S)$ such that $\one_\mathcal{A}\otimes_R
S \mapsto \one_{\mathcal{A}\otimes_R S}$ and the induced map
$\mathcal{A}\otimes_R S =
\shEnd_{\stack(\mathcal{A}\widetilde{\otimes}_R
S}(\one_\mathcal{A}\otimes_R S) \to
\shEnd_{\stack(\mathcal{A}\otimes_R S)}(\one_{\mathcal{A}\otimes_R
S}) = \mathcal{A}\otimes_R S$ is the identity map. In particular, it
is an equivalence. The inverse equivalence
$\stack(\mathcal{A}\otimes_R S) \to
\stack(\mathcal{A})\widetilde{\otimes}_R S$ is induced by the
canonical morphism $(\mathcal{A}\otimes_R S)^+ =
\shEnd_{\stack(\mathcal{A}\otimes_R S)}(\one_{\mathcal{A}\otimes_R
S})^+ \to \stack(\mathcal{A})\widetilde{\otimes}_R S$. Thus, we have
the canonical mutually inverse equivalences $\triv(\mathcal{A})
\widetilde{\otimes}_R S \leftrightarrows \triv(\mathcal{A} \otimes_R
S)$.

A morphism $F \colon \mathcal{A} \to \mathcal{B}$ in
$\CosMatAlg_R(G)$ gives rise to the diagram
\[
\begin{CD}
\triv(\mathcal{A}) \widetilde{\otimes}_R S @>{\triv{F} \widetilde{\otimes}_R S}>> \triv(\mathcal{A}) \widetilde{\otimes}_R S \\
@VVV @VVV \\
\triv(\mathcal{A} \otimes_R S) @>{\triv(F \otimes_R S)}>>
\triv(\mathcal{A} \otimes_R S)
\end{CD}
\]
which is commutative up to a unique $2$-morphism.

\subsubsection{Cosimplicial matrix algebras from algebroid stacks}\label{Cosimplicial matrix algebras from algebroid stacks}
We fix $(\underline{\mathcal{C}},L)\in\Triv_R(G)$ and set, for the
time being, $\mathcal{A}^0 = \shEnd_\mathcal{C}(L)^{op}$. The
assignment $L' \to \shHom_\mathcal{C}(L, L')$ defines an equivalence
$\mathcal{C} \to \widetilde{(\mathcal{A}^{0 op})^+}$. The inverse
equivalence is determines by the assignment $\one_\mathcal{A}
\mapsto L$.

For $n\geq 1$, $0\leq i,j\leq n$ let
\begin{equation}\label{mat from triv entry}
\mathcal{A}^n_{ij} = \left\{ \begin{array}{lcl}
\shHom_{\mathcal{C}^\s{n}_i}(L^\s{n}_i,
\mathcal{C}^\s{n}_{ij}(L^\s{n}_j))
& \textmd{if} & i\leq j \\
\shHom_{\mathcal{C}^\s{n}_j}(\mathcal{C}^\s{n}_{ji}(L^\s{n}_i),
L^\s{n}_j) & \textmd{if} & j\leq i
\end{array} \right.
\end{equation}
where $\mathcal{C}^\s{n}_{ii} = \id_{\mathcal{C}^\s{n}_i}$. Then,
the sheaf $\mathcal{A}^n_{ij}$ has a canonical structure of a
$(\mathcal{A}^0)^\s{n}_{i}\otimes (\mathcal{A}^0)^{\s{n}
op}_{j}$-module. Let
\[
\mathcal{A}^n = \bigoplus_{i,j = 0}^n \mathcal{A}^n_{ij}
\]

The definition of the multiplication of matrix entries comprises
several cases. Suppose that $i\leq j\leq k$. We have the map
\begin{multline}\label{mat mult prep ijk}
\mathcal{A}^n_{jk} = \shHom_{\mathcal{C}^\s{n}_j}(L^\s{n}_j, \mathcal{C}^\s{n}_{jk}(L^\s{n}_k)) \xrightarrow{\mathcal{C}^\s{n}_{ij}} \\
\shHom_{\mathcal{C}^\s{n}_i}(\mathcal{C}^\s{n}_{ij}(L^\s{n}_j),
(\mathcal{C}^\s{n}_{ij}\circ \mathcal{C}^\s{n}_{jk})(L^\s{n}_k))
\xrightarrow{\mathcal{C}^\s{n}_{ijk}} \\
\shHom_{\mathcal{C}^\s{n}_i}(\mathcal{C}^\s{n}_{ij}(L^\s{n}_j),
\mathcal{C}^\s{n}_{ik}(L^\s{n}_k))
\end{multline}
The multiplication of matrix entries is given by
\begin{multline}\label{mat mult ijk}
\mathcal{A}^n_{ij}\otimes \mathcal{A}^n_{jk}
\xrightarrow{\id \otimes \eqref{mat mult prep ijk}} \\
\shHom_{\mathcal{C}^\s{n}_i}(L^\s{n}_i,
\mathcal{C}^\s{n}_{ij}(L^\s{n}_j)) \otimes
\shHom_{\mathcal{C}^\s{n}_i}(\mathcal{C}^\s{n}_{ij}(L^\s{n}_j),
\mathcal{C}^\s{n}_{ik}(L^\s{n}_k))
\\
\xrightarrow{\circ} \shHom_{\mathcal{C}^\s{n}_i}(L^\s{n}_i,
\mathcal{C}^\s{n}_{ik}(L^\s{n}_k)) = \mathcal{A}^n_{ik}
\end{multline}

Suppose that $j\leq i\leq k$. We have the map
\begin{multline}\label{mat mult prep jik}
\mathcal{A}^n_{ik} = \shHom_{\mathcal{C}^\s{n}_i}(L^\s{n}_i,
\mathcal{C}^\s{n}_{ik}(L^\s{n}_k))
\xrightarrow{\mathcal{C}^\s{n}_{ji}} \\
\shHom_{\mathcal{C}^\s{n}_j}(\mathcal{C}^\s{n}_{ji}(L^\s{n}_i),
(\mathcal{C}^\s{n}_{ji}\circ\mathcal{C}^\s{n}_{ik})(L^\s{n}_k))
\xrightarrow{\mathcal{C}^\s{n}_{jik}} \\
\shHom_{\mathcal{C}^\s{n}_j}(\mathcal{C}^\s{n}_{ji}(L^\s{n}_i),
\mathcal{C}^\s{n}_{jk}(L^\s{n}_k))
\end{multline}
The multiplication of matrix entries is given by
\begin{multline}
\mathcal{A}^n_{ij}\otimes \mathcal{A}^n_{jk}  = \\
\shHom_{\mathcal{C}^\s{n}_j}(\mathcal{C}^\s{n}_{ji}(L^\s{n}_i),
L^\s{n}_j) \otimes \shHom_{\mathcal{C}^\s{n}_j}(L^\s{n}_j,
\mathcal{C}^\s{n}_{jk}(L^\s{n}_k))
\xrightarrow{\circ} \\
\shHom_{\mathcal{C}^\s{n}_j}(\mathcal{C}^\s{n}_{ji}(L^\s{n}_i),
\mathcal{C}^\s{n}_{ik}(L^\s{n}_i)) \xrightarrow{\eqref{mat mult prep
jik}^{-1}} \mathcal{A}^n_{ik}
\end{multline}
We leave the remaining cases and the verification of the
associativity of multiplication on $\mathcal{A}^n$ to the reader.

\begin{lemma}
The collection of matrix algebras $\mathcal{A}^n$, $n = 0, 1, 2,
\ldots$, has a canonical structure of a cosimplicial matrix
$R$-algebra.
\end{lemma}

We denote by $\cma(\underline{\mathcal{C}},L)$ the cosimplicial
matrix algebra associated to $(\underline{\mathcal{C}},L)$ by the
above construction.

Suppose that $(\underline{\phi},\phi_\tau) \colon
(\underline{\mathcal{C}},L) \to (\underline{\mathcal{D}},M)$ is a
$1$-morphism in $\Triv_R(G)$. Let $\mathcal{A} =
\cma(\underline{\mathcal{C}},L)$, $\mathcal{B} =
\cma(\underline{\mathcal{D}},M)$.

Let $F^0 \colon \mathcal{A}^0 \to \mathcal{B}^0$ denote the
composition
\[
\mathcal{A}^0 = \shEnd_{\mathcal{C}}(L)^{op} \xrightarrow{\phi_0}
\shEnd_{\mathcal{D}}(\phi_0(L))^{op} \xrightarrow{\phi_\tau}
\shEnd_{\mathcal{D}}(M)^{op} = \mathcal{B}^0
\]
For $n\geq 1$, $0\leq i\leq j\leq n$ let $F^n_{ij} \colon
\mathcal{A}^n_{ij} \to \mathcal{B}^n_{ij}$ (see \eqref{mat from triv
entry}) denote the composition
\begin{multline*}
\shHom_{\mathcal{C}^\s{n}_i}(L^\s{n}_i,
\mathcal{C}^\s{n}_{ij}(L^\s{n}_j)) \xrightarrow{\phi^\s{n}_i} \\
\shHom_{\mathcal{D}^\s{n}_i}(\phi^\s{n}_i(L^\s{n}_i),
(\phi^\s{n}_i\circ \mathcal{C}^\s{n}_{ij})(L^\s{n}_j))
\xrightarrow{\phi^\s{n}_{ij}} \\
\shHom_{\mathcal{D}^\s{n}_i}(\phi^\s{n}_i(L^\s{n}_i),
(\mathcal{D}^\s{n}_{ij}\circ \phi^\s{n}_j)(L^\s{n}_j))
\xrightarrow{\phi^\s{n}_\tau} \\
\shHom_{\mathcal{D}^\s{n}_i}(M^\s{n}_i,
\mathcal{D}^\s{n}_{ij}(M^\s{n}_j)) .
\end{multline*}
The construction of $F^n_{ij}$ in the case $j\leq i$ is similar and
is left to the reader. Let $F^n = \oplus F^n \colon \mathcal{A}^n
\to \mathcal{B}^n$.

\begin{lemma}
The collection of maps $F^n \colon \mathcal{A}^n \to \mathcal{B}^n$
is a $1$-morphism of cosimplicial matrix $R$-algebras.
\end{lemma}

We denote by $\cma(\underline{\phi},\phi_\tau) \colon
\cma(\underline{\mathcal{C}},L) \to \cma(\underline{\mathcal{D}},M)$
the $1$-morphism of cosimplicial matrix algebras associated to the
$1$-morphism $(\underline{\phi},\phi_\tau)$ in $\Triv(G)$ by the
above construction.

Suppose that $(\underline{\phi},\phi_\tau),
(\underline{\psi},\psi_\tau) \colon (\underline{\mathcal{C}},L) \to
(\underline{\mathcal{D}},M)$ are $1$-morphisms in $\Triv(G)$ and $b
\colon (\underline{\phi},\phi_\tau) \to
(\underline{\psi},\psi_\tau)$ is a $2$-morphism.

Let $\cma(b)^0 \colon \cma(\underline{\phi},\phi_\tau)^0 \to
\cma(\underline{\psi},\psi_\tau)^0$ denote the composition
\[
M \xrightarrow{(\phi_\tau)^{-1}} \phi(L) \xrightarrow{b(B)} \psi(L)
\xrightarrow{\psi_\tau} M
\]
(i.e. $\cma(b)^0\in \Gamma(N_0G;\shEnd_\mathcal{D}(M))$.)

\begin{lemma}
The section $\cma(b)^0\in \Gamma(N_0G;\shEnd_\mathcal{D}(M))$ is a
$2$-morphism $\cma(\underline{\phi},\phi_\tau)^0 \to
\cma(\underline{\psi},\psi_\tau)^0$.
\end{lemma}
\begin{proof}
For $f\in \shEnd_\mathcal{C}(L)$ we have
\begin{eqnarray*}
\cma(b)^0 \circ \cma(\underline{\phi},\phi_\tau)^0(f) & = &
\psi_\tau \circ b(L) \circ (\phi_\tau)^{-1} \circ \phi_\tau \circ
\phi(f) \circ (\phi_\tau)^{-1} \\
 & = & \psi_\tau \circ b(L) \circ \phi(f) \circ (\phi_\tau)^{-1} \\
 & = & \psi_\tau \circ \psi(f) \circ b(L) \circ (\phi_\tau)^{-1} \\
 & = & \psi_\tau \circ \psi(f) \circ (\psi_\tau)^{-1} \circ \psi_\tau
 \circ b(L) \circ (\phi_\tau)^{-1} \\
 & = & \cma(\underline{\psi},\psi_\tau)^0(f) \circ \cma(b)^0
\end{eqnarray*}
\end{proof}

For $n\geq 1$ let $\cma(b)^n = (\cma(b)^n_{ij}) \in
\Gamma(N_nG;\cma(\underline{\mathcal{D}},M)^n)$ denote the
``diagonal'' matrix with $\cma(b)^n_{ii} = (\cma(b)^0)^\s{n}_i$.

\begin{lemma}
The collection $\cma(b)$ of sections $\cma(b)^n$, $n = 0,1,2,\ldots$
is a $2$-morphism $\cma(\underline{\phi},\phi_\tau) \to
\cma(\underline{\psi},\psi_\tau)$.
\end{lemma}

\begin{prop}
{~}\begin{enumerate}
\item The assignments $(\underline{\mathcal{C}},L) \mapsto
\cma(\underline{\mathcal{C}},L)$, $(\underline{\phi},\phi_\tau)
\mapsto \cma(\underline{\phi},\phi_\tau)$, $b \mapsto \cma(b)$
define a functor
\[
\cma \colon \Triv_R(G) \to \CosMatAlg_R(NG)
\]

\item The functors $\cma$ commute with the base change functors.
\end{enumerate}
\end{prop}

For $\mathcal{A}\in \CosMatAlg(NG)$ we have
\[
\cma(\triv(\mathcal{A}))^0 = \shEnd_{\widetilde{(\mathcal{A}^{0
op})^+}}(\one_\mathcal{A})^{op} = \mathcal{A}^0
\]
For $0\leq i \leq j\leq n$ there is a canonical isomorphism
\[
\shHom_{\widetilde{(\mathcal{A}^{n op}_{ii})^+}}(\one^\s{n}_i,
\triv(\mathcal{A})^\s{n}_{ij}(\one^\s{n}_j)) \cong
\mathcal{A}^n_{ij} \ .
\]
For $0\leq j\leq i\leq n$ we have
\[
\shHom_{\widetilde{(\mathcal{A}^{n
op}_{jj})^+}}(\triv(\mathcal{A})^\s{n}_{ji}(\one^\s{n}_i),
\one^\s{n}_j) \cong \shHom_{\mathcal{A}^n_{jj}}(\mathcal{A}^n_{ji},
\mathcal{A}^n_{jj}) \cong \mathcal{A}^n_{ij} \ ,
\]
where the last isomorphism comes from the (multiplication) pairing
$\mathcal{A}^n_{ji} \otimes \mathcal{A}^n_{ij} \to
\mathcal{A}^n_{jj}$. These isomorphisms give rise to the canonical
isomorphism of CMA
\[
\cma(\triv(\mathcal{A})) \cong \mathcal{A}
\]

On the other hand, given $(\underline{\mathcal{C}},L)\in\Triv(G)$,
we have the canonical equivalence
$\widetilde{(\shEnd_\mathcal{C}(L)^{op})^+} \to \mathcal{C}$
determined by $\one_{\widetilde{(\shEnd_\mathcal{C}(L)^{op})^+}}
\mapsto L$.

To summarize, we have the following proposition.
\begin{prop}\label{inverse}
The functors $\triv$ and $\cma$ are mutually quasi-inverse
equivalences of categories.
\end{prop}

\subsubsection{$\Psi$-algebroid stacks}
Suppose that $X$ is a space and $\Psi$ is a pseudo-tensor
subcategory of $\Sh_k(X)$ as in \ref{subsubsection:pseudo-tensor
cats}.

An algebroid stack $\mathcal{C}\in\AlgStack_k(X)$ is called a
\emph{$\Psi$-algebroid stack} if for any open subset $U\subseteq X$
and any three objects $L_1,L_2,L_3 \in \mathcal{C}(U)$ the
composition map
\[
\shHom_\mathcal{C}(L_1,L_2)\otimes_k\shHom_\mathcal{C}(L_2,L_3) \to
\shHom_\mathcal{C}(L_1,L_2)
\]
is in $\Psi$.

Suppose that $\mathcal{C}$ and $\mathcal{D}$ are $\Psi$-algebroid
stacks. A $1$-morphism $\phi \colon \mathcal{C} \to \mathcal{D}$ is
called a \emph{$\Psi$-$1$-morphism} if for any open subset
$U\subseteq X$ and any two objects $L_1,L_2 \in \mathcal{C}(U)$ the
map
\[
\phi \colon \shHom_\mathcal{C}(L_1,L_2) \to
\shHom_\mathcal{D}(\phi(L_1),\phi(L_2))
\]
is in $\Psi$.

Suppose that $\phi,\psi \colon  \mathcal{C} \to \mathcal{D}$ are
$\Psi$-$1$-morphisms. A $2$-morphism $b \colon \phi \to \psi$ is
called a \emph{$\Psi$-$2$-morphism} if for any open subset
$U\subseteq X$ and any two objects $L_1,L_2\in\mathcal{C}(U)$ the
map
\[
b \colon \shHom_\mathcal{D}(\phi(L_1),\phi(L_2)) \to
\shHom_\mathcal{D}(\psi(L_1),\psi(L_2))
\]
is in $\Psi$.

We denote by $\AlgStack^\Psi_k(X)$ the subcategory of
$\AlgStack_k(X)$ whose objects, $1$-morphisms and $2$-morphisms are,
respectively, the $\Psi$-algebroid stacks, the $\Psi$-$1$- and the
$\Psi$-$2$-morphisms.

Suppose that $G$ is an \'etale category and let $\Psi\colon p
\mapsto \Psi^p$ be a cosimplicial pseudo-tensor subcategory of
$\Sh_k(NG)$ as in \ref{section: cosimplicial matrix algebras}. An
algebroid stack $\underline{\mathcal{C}} =
(\mathcal{C},\mathcal{C}_{01},\mathcal{C}_{012})\in\AlgStack_k(G)$
is called a \emph{$\Psi$-algebroid stack} if $\mathcal{C}$ is a
$\Psi$-algebroid stack on $N_0G$, $\mathcal{C}_{01}$ is a
$\Psi$-1-morphism and $\mathcal{C}_{012}$ is a $\Psi$-2-morphism.
Similarly, one defines $\Psi$-1-morphism and $\Psi$-2-morphism of
$\Psi$-algebroid stacks on $G$; the details are left to the reader.
We denote the resulting subcategory of $\AlgStack_k(G)$ by
$\AlgStack_k^\Psi(G)$.

The subcategory $\Triv^\Psi_k(G)$ of $\Triv_k(G)$ has objects pairs
$(\underline{\mathcal{C}}, L)$ with
$\underline{\mathcal{C}}\in\AlgStack_k^\Psi(G)$, and $1$-morphisms
and $2$-morphisms restricted accordingly.

\subsubsection{Deformations of $\Psi$-algebroid stacks}
Recall (\ref{subsubsection:examples of pseudo-tensor cats}, Example
\texttt{DEF}) that, for $R\in\ArtAlg_k$ we have the pseudo-tensor
category $\Psi(R)$.

Suppose that $G$ is an \'etale category and
$\underline{\mathcal{C}}$ is a $\Psi$-algebroid stack on $G$.

We denote by $\widetilde{\Def}(\underline{\mathcal{C}})(R)$ the
following category:
\begin{itemize}
\item an object of $\widetilde{\Def}(\underline{\mathcal{C}})(R)$ is a pair $(\underline{\mathcal{D}},
    \underline{\pi})$ where $\underline{\mathcal{D}}$ is a $\widetilde{\Psi(R)}$-algebroid stack on $G$ and
    $\underline{\pi} \colon \underline{\mathcal{D}}\widetilde{\otimes}_R k \to \underline{\mathcal{C}}$ is an
    equivalence of $\Psi$-algebroid stacks,

\item a $1$-morphism $(\underline{\mathcal{D}}, \underline{\pi}) \to (\underline{\mathcal{D}}', \underline{\pi}')$
    of $\widetilde{\Psi(R)}$-deformations of $\underline{\mathcal{C}}$ is a pair $(\underline{\phi}, \beta)$, where
    $\underline{\phi} \colon \underline{\mathcal{D}} \to \underline{\mathcal{D}}'$ is a $1$-morphism of
    $\widetilde{\Psi(R)}$-algebroid stacks and $\beta \colon \underline{\pi} \to
    \underline{\pi}'\circ(\underline{\phi}\widetilde{\otimes}_R k)$ is $2$-isomorphism,

\item a $2$-morphism $b \colon (\underline{\phi}_1, \beta_1) \to (\underline{\phi}_2, \beta_2)$ is a $2$-morphism
    $b \colon \underline{\phi}_1 \to \underline{\phi}_2$ such that $\beta_2 = (\id_{\pi'}\otimes b) \circ \beta_1$
\end{itemize}

Suppose that $(\underline{\phi}, \beta) \colon
(\underline{\mathcal{D}}, \underline{\pi}) \to
(\underline{\mathcal{D}}', \underline{\pi}')$ and
$(\underline{\phi}', \beta') \colon (\underline{\mathcal{D}}',
\underline{\pi}') \to (\underline{\mathcal{D}}'',
\underline{\pi}'')$ are $1$-morphisms of $R$-deformations of
$\underline{\mathcal{C}}$. We have the composition
\begin{equation}\label{compmordef pi}
\pi \xrightarrow{\beta}
\pi'\circ(\underline{\phi}\widetilde{\otimes}_R k)
\xrightarrow{\beta'\otimes\id_{\underline{\phi}\widetilde{\otimes}_R
k}} \pi''\circ(\underline{\phi}'\widetilde{\otimes}_R k)\circ
(\underline{\phi}\widetilde{\otimes}_R k)
\end{equation}
Then, the pair $(\underline{\phi}'\circ\underline{\phi},
\eqref{compmordef pi})$ is a $1$-morphism of $R$-deformations
$(\underline{\mathcal{D}}, \underline{\pi}) \to
(\underline{\mathcal{D}}'', \underline{\pi}'')$ defined to be the
composition $(\underline{\phi}', \beta') \circ (\underline{\phi},
\beta)$.

Vertical and horizontal compositions of $2$-morphisms of deformation
are those of $2$-morphisms of underlying algebroid stacks.

\begin{lemma}
The $2$-category $\widetilde{\Def}(\underline{\mathcal{C}})(R)$ is a
$2$-groupoid.
\end{lemma}

We denote by $\Def(\underline{\mathcal{C}})(R)$ the full subcategory
of $\Psi(R)$-algebroid stacks.

For $(\underline{\mathcal{C}}, L) \in \Triv^\Psi(G)$ we define the
$2$-category $\Def(\underline{\mathcal{C}}, L)(R)$ as follows:
\begin{itemize}
\item the objects are quadruples $(\underline{\mathcal{D}}, \underline{\pi}, \pi_\tau, M)$ such that
\begin{list}{}{}
\item $(\underline{\mathcal{D}}, \underline{\pi})\in \Def(\underline{\mathcal{C}})(R)$,
\item $(\underline{\mathcal{D}}, M)\in \Triv^{\Psi(R)}(G)$,
\item $(\underline{\pi}, \pi_\tau) \colon (\underline{\mathcal{D}}, M)\otimes_R k \to (\underline{\mathcal{C}}, L)$
    is a $1$-morphism in $\Triv^\Psi(G)$;
\end{list}
\item a $1$-morphism $(\underline{\mathcal{D}}, \underline{\pi}, \pi_\tau, M) \to (\underline{\mathcal{D}}', \underline{\pi}', \pi'_\tau, M')$ is a triple $(\underline{\phi}, \phi_\tau, \beta)$ where
\begin{list}{}{}
\item $(\underline{\phi},\beta) \colon (\underline{\mathcal{D}}, \underline{\pi}) \to (\underline{\mathcal{D}}',
    \underline{\pi}')$  is a $1$-morphisms in $\Def(\underline{\mathcal{C}})(R)$,
\item $(\underline{\phi}, \phi_\tau) \colon (\underline{\mathcal{D}}, M) \to (\underline{\mathcal{D}}', M')$ is a
    $1$-morphism in $\Triv^{\Psi(R)}(G)$,
\end{list}
\item a $2$-morphism $(\underline{\phi}, \phi_\tau, \beta) \to (\underline{\phi}', \phi'_\tau, \beta')$ is a $2$-morphisms $\underline{\phi} \to \underline{\phi}'$.
\end{itemize}

\subsubsection{$\triv$ for deformations}\label{triv for deformations}
Let $\mathcal{A}\in\CosMatAlg^\Psi_k(NG)$, $R\in\ArtAlg_k$. The
functors $\stack$ and $\triv$ were defined in \ref{subsubsection:
triv from cma} (see Proposition \ref{prop: cosmatalg to algstack}).
\begin{lemma}
{~}
\begin{enumerate}
\item The functor $\stack$ restricts to the functor $\stack \colon \Def(\mathcal{A})(R) \to
    \AlgStack^{\Psi(R)}(G)$.

\item The functor $\triv$ restricts to the functor $\triv \colon \Def(\mathcal{A})(R) \to \Triv^{\Psi(R)}(G)$.
\end{enumerate}
\end{lemma}

For $\mathbf{B}\in\Def(\mathcal{A})(R)$ the identification
$\mathbf{B}\otimes_R k = \mathcal{A}$ induces the equivalence
$\underline{\pi_\mathbf{B}} \colon
\triv(\mathbf{B})\widetilde{\otimes}_R k \cong
\triv(\mathbf{B}\otimes_R k) = \triv(\mathcal{A})$. Let
\[
\widetilde{\triv}(\mathbf{B}) \colon= (\stack(\mathbf{B}),
\underline{\pi_\mathbf{B}}, \id, \one_\mathbf{B}) \ .
\]

A morphism $F \colon \mathbf{B} \to \mathbf{B}'$ in
$\Def(\mathcal{A})(R)$ induces the morphism $\triv(F) \colon
\triv(\mathbf{B}) \to \triv(\mathbf{B}')$. The diagram
\[
\begin{CD}
\triv(\mathbf{B})\widetilde{\otimes}_R k @>{\triv(F)\widetilde{\otimes}_R k}>> \triv(\mathbf{B}')\widetilde{\otimes}_R k \\
@V{\underline{\pi_\mathbf{B}}}VV @VV{\underline{\pi_\mathbf{B}'}}V \\
\triv(\mathcal{A}) @= \triv(\mathcal{A})
\end{CD}
\]
commutes up to a unique $2$-morphism, hence, $\triv(F)$ gives rise
in a canonical way to a morphism in $\Def(\triv(\mathcal{A}))(R)$.

Let $F_i \colon \mathbf{B} \to \mathbf{B}'$, $i=1,2$ be two
$1$-morphisms in $\Def(\mathcal{A})(R)$. A $2$-morphism $b \colon
F_1 \to F_2$ in $\Def(\mathcal{A})(R)$ induces the $2$-morphism
$\triv(b) \colon \triv(F_1) \to \triv(F_2)$ in $\Triv_R(G)$ which is
easily seen to be a $2$-morphism in $\Def(\triv(\mathcal{A}))(R)$.

\begin{lemma}
The assignment $\mathbf{B} \mapsto \widetilde{\triv}(\mathbf{B})$
extends to a functor
\begin{equation}\label{triv for def}
\widetilde{\triv} \colon \Def(\mathcal{A})(R) \to
\Def(\triv(\mathcal{A}))(R)
\end{equation}
naturally in $R\in\ArtAlg_k$.
\end{lemma}


\begin{thm}\label{triv for def equiv}
The functor \eqref{triv for def} is an equivalence.
\end{thm}
\begin{proof}
First, we show that, for $\mathbf{B}, \mathbf{B}' \in
\Def(\mathcal{A})(R)$ the functor
\begin{equation}\label{triv for def on hom}
\Hom_{\Def(\mathcal{A})(R)}(\mathbf{B},\mathbf{B}') \to
\Hom_{\Def(\triv(\mathcal{A}))(R)}(\widetilde{\triv}(\mathbf{B}),\widetilde{\triv}(\mathbf{B}'))
\end{equation}
induced by the functor $\widetilde{\triv}$ is an equivalence.

For $1$-morphism $(\underline{\phi}, \phi_\tau,\beta) \in
\Hom_{\Def(\triv(\mathcal{A}))(R)}(\widetilde{\triv}(\mathbf{B}),\widetilde{\triv}(\mathbf{B}'))$
consider the composition
\begin{equation}\label{cma triv comp}
\mathbf{B} \cong \cma(\triv(\mathbf{B}))
\xrightarrow{\cma(\underline{\phi})} \cma(\triv(\mathbf{B}')) \cong
\mathbf{B}'
\end{equation}
It is easy to see that $\beta$ induces an isomorphism $\eqref{cma
triv comp}\otimes_R k \cong \id_\mathcal{A}$. Thus, the functor
\eqref{triv for def on hom} is essentially surjective.

Let $F_i \colon \mathbf{B} \to \mathbf{B}'$, $i=1,2$ be two
$1$-morphisms in $\Def(\mathcal{A})(R)$. It is easy to see that the
isomorphism
\begin{multline*}
\Hom(\triv(F_1), \triv(F_2)) \to \\
\Hom(\cma(\triv(F_1)),\cma(\triv(F_2))) \cong \Hom(F_1,F_2)
\end{multline*}
induced by the functor $\cma$ restricts to an isomorphism of
respective space of morphisms in
$\Hom_{\Def(\triv(\mathcal{A}))(R)}(\widetilde{\triv}(\mathbf{B}),\widetilde{\triv}(\mathbf{B}'))$
and $\Hom_{\Def(\mathcal{A})(R)}(\mathbf{B},\mathbf{B}')$.

It remains to show that the functor \eqref{triv for def} is
essentially surjective. Consider $(\underline{\mathcal{D}},
\underline{\pi}, \pi_\tau, M) \in \Def(\triv(\mathcal{A}))(R)$. Let
$\mathcal{B} = \cma(\underline{\mathcal{D}}, M)$. The morphism
$(\underline{\pi}, \pi_\tau)$ induces the isomorphism
$\cma(\underline{\pi}, \pi_\tau) \colon \mathcal{B}\otimes_R k \cong
\mathcal{A}$.

We choose isomorphisms in $\Psi(R)$ $\mathcal{B}^0 \cong
\mathcal{A}^0\otimes_k R$, $\mathcal{B}^1_{01} \cong
\mathcal{A}^1_{01}\otimes_k R$ and $\mathcal{B}^1_{10} \cong
\mathcal{A}^1_{10}\otimes_k R$ which induce respective restrictions
of the isomorphism $\cma(\underline{\pi}, \pi_\tau)$. The above
choices give rise in a canonical way to isomorphisms
$\mathcal{B}^n_{ij} \cong \mathcal{A}^n_{ij}\otimes_k R$ for $n = 0,
1, 2, \ldots$, $0 \leq i,j \leq n$. Let $\mathbf{B}$ denote the
cosimplicial matrix algebra structure on $\mathcal{A}\otimes_k R$
induce by that on $\mathcal{B}$ via the above isomorphisms. It is
easy to see that $\mathbf{B}\in\Def(\mathcal{A})(R)$. The
isomorphism of cosimplicial matrix algebras
$\mathcal{B}\cong\mathbf{B}$ induces the equivalence
$(\underline{\mathcal{D}}, \underline{\pi}, \pi_\tau, M) \cong
\widetilde{\triv}(\mathbf{B})$.
\end{proof}

\begin{thm}\label{mt1}
 Let $(\underline{\mathcal{C}}, L) \in \Triv^\Psi(G)$. Then we have a canonical equivalence of $2$-groupoids
\[
\Def(\underline{\mathcal{C}}, L) \cong
\MC^2(\mathfrak{G}(\cma(\mathcal{C}, L))\otimes_k
 \mathfrak{m}_R)
\]
\end{thm}
\begin{proof}
This is a direct consequence of the Proposition \ref{inverse},
Theorem \ref{triv for def equiv} and Theorem \ref{thm:mapgroup
equiv}.
\end{proof}

\subsubsection{Deformation theory of twisted forms of $\mathcal{O}$}
Let $G$ be an \'etale groupoid. We now apply the results of the
preceding sections with the $\Psi = \mathtt{DIFF}$ (see
\ref{subsubsection:examples of pseudo-tensor cats}) and omit it from
notations.

Suppose that $\underline{\mathcal{S}} =
(\mathcal{S},\mathcal{S}_{01},\mathcal{S}_{012})$ is a twisted form
of $\mathcal{O}_G$, i.e. and an algebroid stack
$\underline{\mathcal{S}}$ on $G$ such that $\mathcal{S}$ is locally
equivalent to $\widetilde{\mathcal{O}_{N_0G}^+}$.

Let $\mathbb{B}$ be a basis of the topology on $N_0G$. Let
$\mathcal{E} \colon= \mathcal{E}_\mathbb{B}(G)$ denote the
corresponding \'etale category of embeddings, $\lambda \colon
\mathcal{E} \to G$ the canonical map (see \ref{cat of emb}).

The functors $\lambda^{-1}$ and $\lambda_!$ (see \ref{subsubsection:
stacks on et cat}, \ref{subsubsection: inv im stack equiv}) restrict
to mutually quasi-inverse equivalences of $2$-categories
\[
\lambda_R^{-1}\colon \AlgStack_R(G) \rightleftarrows
\AlgStack_R(\mathcal{E}) \colon  \lambda^R_!
\]
natural in $R$. The explicit construction of \ref{subsubsection:
adjoint top space}, \ref{subsubsection: general case} shows that the
equivalence of respective categories of sheaves on $G$ and
$\mathcal{E}$ induces an equivalence of respective
$\widetilde{\mathtt{DIFF}(R)}$ categories (which, however, do not
preserve the respective $\mathtt{DIFF}(R)$ subcategories).

In particular, for $(\underline{\mathcal{D}},\underline{\pi}) \in
\widetilde{\Def}(\underline{\mathcal{S}})(R)$,
$\lambda^{-1}\underline{\mathcal{D}}$ is a
$\widetilde{\mathtt{DIFF}(R)}$-algebroid stack on $\mathcal{E}$,
there is a natural equivalence
$\lambda^{-1}\underline{\mathcal{D}}\widetilde{\otimes}_R k \cong
\lambda^{-1}(\underline{\mathcal{D}}\widetilde{\otimes}_R
\mathbb{C})$, hence $\underline{\pi}$ induces the equivalence
$\lambda^*(\underline{\pi}) \colon
\lambda^{-1}\underline{\mathcal{D}}\widetilde{\otimes}_R \mathbb{C}
\to \lambda^{-1}\underline{\mathcal{S}}$. The assignment
\[
(\underline{\mathcal{D}},\underline{\pi}) \mapsto
\lambda_R^{-1}(\underline{\mathcal{D}},\underline{\pi}) \colon=
(\lambda^{-1}\underline{\mathcal{D}}, \lambda^*(\underline{\pi}))
\]
extends to a functor
\begin{equation}\label{G to E defs}
\lambda^{-1} \colon \widetilde{\Def}(\underline{\mathcal{S}})
\xrightarrow{\cong}
\widetilde{\Def}(\lambda^{-1}\underline{\mathcal{S}}) .
\end{equation}

\begin{lemma}\label{equi}
The functor \eqref{G to E defs} is an equivalence of $2$-groupoids.
\end{lemma}
\begin{proof}
Follows from the properties of $\lambda_!$.
\end{proof}

\begin{thm}\label{thm: passage to E}
Suppose that $\mathbb{B}$ is a basis of $N_0G$ which consists of
Hausdorff contractible open sets. Let $R\in\ArtAlg_\mathbb{C}$.
\begin{enumerate}
\item $\lambda^{-1}\mathcal{S}(N_0\mathcal{E}_\mathbb{B})$ is nonempty and connected by isomorphisms.

\item Let $L\in \lambda^{-1}\mathcal{S}(N_0\mathcal{E}_\mathbb{B})$ be a
trivialization. The functor
\begin{equation}\label{fforget triv}
\Xi\colon \Def(\lambda^{-1}\underline{\mathcal{S}}, L)(R) \to
\Def(\lambda^{-1}\underline{\mathcal{S}})(R)
\end{equation}
defined by $\Xi(\underline{\mathcal{D}},\underline{\pi},\pi_\tau,M)
= (\underline{\mathcal{D}},\underline{\pi})$,
$\Xi(\underline{\phi},\phi_\tau,\beta) = (\underline{\phi},\beta)$,
$\Xi(b) = b$ is an equivalence.

\item The functor $\Def(\lambda^{-1}\underline{\mathcal{S}})(R) \to \widetilde{\Def}(\lambda^{-1}\underline{\mathcal{S}})(R)$ is an equivalence.
\end{enumerate}
\end{thm}
\begin{proof}
Since
$H^l(N_0\mathcal{E}_\mathbb{B};\mathcal{O}^\times_{N_0\mathcal{E}_\mathbb{B}})$
is trivial for $l\neq 0$ the first statement follows.

Consider $(\underline{\mathcal{D}},\underline{\pi},\pi_\tau,M),
(\underline{\mathcal{D}}',\underline{\pi}',\pi'_\tau,M') \in
\Def(\lambda^{-1}\underline{\mathcal{S}},L)(R)$, and $1$-morphisms
$(\underline{\phi},\phi_\tau,\beta), (\underline{\psi}, \psi_\tau,
\gamma) \colon (\underline{\mathcal{D}},\underline{\pi},\pi_\tau,M)
\to (\underline{\mathcal{D}}',\underline{\pi}',\pi'_\tau,M')$.

It is clear from the definition of $\Xi$ that the induced map
\[
\Xi \colon \Hom((\underline{\phi},\phi_\tau,\beta),
(\underline{\psi}, \psi_\tau, \gamma)) \to
\Hom((\underline{\phi},\beta),(\underline{\psi},\gamma))
\]
is an isomorphism, i.e. the functor
\begin{equation}\label{forget triv on hom}
\Xi \colon \Hom_{\Def(\lambda^{-1}\mathcal{S},
L)(R)}((\underline{\phi},\phi_\tau,\beta), (\underline{\psi},
\psi_\tau, \gamma)) \to
\Hom_{\Def(\lambda^{-1}\mathcal{S})(R)}((\underline{\phi},\beta),
(\underline{\psi}, \gamma))
\end{equation}
is fully faithful.

Consider a $1$-morphism
\[
(\underline{\phi},\beta) \colon
\Xi(\underline{\mathcal{D}},\underline{\pi},\pi_\tau,M)
=(\underline{\mathcal{D}},\underline{\pi}) \to
(\underline{\mathcal{D}}',\underline{\pi}') =
\Xi(\underline{\mathcal{D}}',\underline{\pi}',\pi'_\tau,M')
\]
in $\Def(\lambda^{-1}\mathcal{S})$. We have the isomorphism
\begin{multline*}
\underline{\pi}'\colon
\Hom_{\underline{\mathcal{D}}'\widetilde{\otimes}_R
\mathbb{C}}((\underline{\phi}\otimes_R \mathbb{C})(M\otimes_R
\mathbb{C}), M'\otimes_R \mathbb{C}) \to \\
\Hom_{\lambda^{-1}\underline{\mathcal{S}}}(\underline{\pi}'(\underline{\phi}\otimes_R
\mathbb{C})(M\otimes_R \mathbb{C}),\underline{\pi}'(M'\otimes_R
\mathbb{C}))
\end{multline*}
Let $\varpi\colon (\underline{\phi}\otimes_R \mathbb{C})(M\otimes_R
\mathbb{C}) \to M'\otimes_R \mathbb{C}$ denote the map such that
$\underline{\pi}'(\varpi) = (\pi'_\tau)^{-1}\circ \pi_\tau\circ
\beta^{-1}$. Note that $\varpi$ is an isomorphism. Since the map
\[
\Hom_{\underline{\mathcal{D}}'}(\phi(M),M') \to
\Hom_{\underline{\mathcal{D}}'\widetilde{\otimes}_R
\mathbb{C}}((\underline{\phi}\otimes_R \mathbb{C})(M\otimes_R
\mathbb{C}), M'\otimes_R \mathbb{C})
\]
is surjective there exists $\widetilde{\varpi} \colon \phi(M) \to
M'$ which lifts $\varpi$. Since the latter is an isomorphism and
$\mathfrak{m}_R$ is nilpotent, the map $\widetilde{\varpi}$ is an
isomorphism. The triple
$(\underline{\phi},\widetilde{\varpi},\beta)$ is a $1$-morphism
$(\underline{\mathcal{D}},\underline{\pi},\pi_\tau,M) \to
(\underline{\mathcal{D}}',\underline{\pi}',\pi'_\tau,M')$ such that
$\Xi(\underline{\phi},\widetilde{\varpi},\beta) =
(\underline{\phi},\beta)$.This shows that the functor \eqref{forget
triv on hom} is essentially surjective, hence an equivalence.

It remains to show that \eqref{fforget triv} is essentially
surjective. It suffices to show that, for any deformation
$(\mathcal{D},\pi) \in \Def((N_0\lambda)^{-1}\mathcal{S}, L)(R)$,
there exists an object $M\in \mathcal{D}(N_0\mathcal{E}_\mathbb{B})$
and an isomorphism $M\otimes_R \mathbb{C} \cong L$. This is implied
by the following fact: if $X$ is a \emph{Hausdorff} manifold, any
deformation of $\widetilde{\mathcal{O}_X^+}$ is a star-product. In
other words, for any open covering $\mathcal{U}$ of $X$, denoting
the corresponding \'etale groupoid by $\mathcal{U}$ and by
$\epsilon\colon  \mathcal{U} \to X$ the canonical map, the functor
\begin{equation}\label{lift of def}
\MC^2(\Gamma(X;\mathfrak{g}(\mathcal{O}_X))\otimes_\mathbb{C}
\mathfrak{m}_R) \to \Def(\widetilde{\mathcal{O}_\mathcal{U}^+},
\one)(R)
\end{equation}
is an equivalence. Let $\mathcal{A}_X \colon=
\cma(\widetilde{\mathcal{O}^+})$, $\mathcal{A}_\mathcal{U} \colon=
\cma(\widetilde{\mathcal{O}_\mathcal{U}^+}, \one)$. We have the
commutative diagram
\[
\begin{CD}
\Gamma(X; \mathfrak{g}(\mathcal{O}_X)) @>>> \mathfrak{G}(\mathcal{O}_X) @>{\epsilon^*}>> \mathfrak{G}(\mathcal{O}_\mathcal{U}) \\
& & @V{\cotr}VV @VV{\cotr}V \\
& & \mathfrak{G}(\mathcal{A}_X) @>{\epsilon^*}>>
\mathfrak{G}(\mathcal{A}_\mathcal{U}) .
\end{CD}
\]
After the identifications
$\Def(\widetilde{\mathcal{O}_\mathcal{U}^+}, \one)(R) \cong
\Def(\mathcal{A})(R) \cong
\MC^2(\mathfrak{G}(\mathcal{A})\otimes_\mathbb{C} \mathfrak{m}_R)$
the functor \eqref{lift of def} is induced by the composition
$\Gamma(X; \mathfrak{g}(\mathcal{O}_X)) \to
\mathfrak{G}(\mathcal{A}_\mathcal{U})$ (of morphisms in the above
diagram), hence it is sufficient to show that the latter is a
quasi-isomorphism. The cotrace (vertical) maps are
quasi-isomorphisms by \cite{Loday} ; the top horizontal composition
is the canonical map $\Gamma(X;\mathcal{F}) \to
\Gamma(\real{N\mathcal{U}}; \real{\epsilon^*\mathcal{F}})$ (with
$\mathcal{F} = \mathfrak{g}(\mathcal{O}_X)$) which is a
quasi-isomorphism for any bounded below complex of sheaves
$\mathcal{F}$ which satisfies $H^i(X;\mathcal{F}^j) = 0$ for all
$i\neq 0$ and all $j$. This finishes the proof of the second claim.

Suppose given $(\underline{\mathcal{D}}, \underline{\pi}) \in
\widetilde{\Def}(\lambda^{-1}\mathcal{S})(R)$, i.e., in particular,
$\mathcal{D}$ is a $\widetilde{\mathtt{DIFF}}(R)$-stack. In order to
establish the last claim we need to show that, in fact,
$\mathcal{D}$ is a $\mathtt{DIFF}(R)$-stack. Suppose that $L_1$ and
$L_2$ are two (locally defined) objects in $\mathcal{D}$; let
$\mathcal{F} \colon= \shHom_\mathcal{D}(L_1,L_2)$, $\mathcal{F}_0
\colon= \mathcal{F}\otimes_R \mathbb{C}$. By assumption,
$\mathcal{F}$ is \emph{locally} isomorphic to
$\mathcal{F}_0\otimes_\mathbb{C} R$ in $\mathtt{DIFF}$ by a map
which induces the identification $\mathcal{F}_0 =
(\mathcal{F}_0\otimes_\mathbb{C} R)\otimes_R \mathbb{C}$. We need to
establish the existence of a global such an isomorphism $\mathcal{F}
\cong \mathcal{F}_0\otimes_\mathbb{C} R$. Let
$\shIsom_0(\mathcal{F}_0\otimes_\mathbb{C} R, \mathcal{F})$ denote
the sheaf of locally defined $R$-linear morphisms
$\mathcal{F}_0\otimes_\mathbb{C} R \to \mathcal{F}$ in
$\widetilde{\mathtt{DIFF}}$ which reduce to the identity modulo
$\mathfrak{m}_R$. Let $\shAut_0(\mathcal{F}_0\otimes_\mathbb{C} R)$
denote the similarly defined sheaf of groups of locally defined
automorphisms of $ \mathcal{F}_0\otimes_\mathbb{C} R$. Then,
$\shIsom_0(\mathcal{F}_0\otimes_\mathbb{C} R, \mathcal{F})$ is a
torsor under $\shAut_0(\mathcal{F}_0\otimes_\mathbb{C} R)$.

Arguing as in Lemma 6 and Corollary 7 of \cite{bgnt2} using the
exponential map
$\Diff(\mathcal{F}_0,\mathcal{F}_0)\otimes_\mathbb{C} \mathfrak{m}_R
\to \shAut_0(\mathcal{F}_0\otimes_\mathbb{C} R)$ one shows that the
sheaf of groups $\shAut_0(\mathcal{F}_0\otimes_\mathbb{C} R)$ is
soft, hence the torsor $\shIsom_0(\mathcal{F}_0\otimes_\mathbb{C} R,
\mathcal{F})$ is trivial, i.e. admits a global section.
\end{proof}

\subsubsection{} Here we obtain the main results of this paper -- classification of the
deformation groupoid of twisted form of $\mathcal{O}_G$ in terms of
the twisted DGLA of jets (cf. \eqref{defgs}).

\begin{thm}\label{equiv of 2-gpds} Let $G$ be an \'etale groupoid and $\underline{\mathcal{S}}$ - a twisted form of $\mathcal{O}_G$.
Suppose that $\mathbb{B}$ is a basis of $N_0G$ which consists of
Hausdorff contractible open sets, and let
$\mathcal{E}=\mathcal{E}_{\mathbb{B}}(G)$ be the corresponding
embedding category. Let  $R\in\ArtAlg_\mathbb{C}$. Then there exists
an equivalence of $2$-groupoids
\[
\widetilde{\Def}(\underline{\mathcal{S}})(R)\cong
\MC^2(\mathfrak{G}_\DR(\mathcal{J}_{N\mathcal{E}})_{[\lambda^{-1}\underline{\mathcal{S}}]}\otimes
\mathfrak{m}_R).
\]
\end{thm}
\begin{proof}
Note that we have the following equivalences:
 \[
 \widetilde{\Def}(\underline{\mathcal{S}})(R)\cong
    \widetilde{\Def}(\lambda^{-1}\underline{\mathcal{S}})(R)
    \cong\Def(\lambda^{-1}\underline{\mathcal{S}})(R)
\]

Here the first equivalence  is the result of Lemma \eqref{equi} and
the second  is a part of the Theorem \ref{thm: passage to E}. By the
Theorem  \ref{thm: passage to E}
$\lambda^{-1}\underline{\mathcal{S}}(N_0\mathcal{E}_\mathbb{B})$ is
nonempty. Let $L\in
\lambda^{-1}\mathcal{S}(N_0\mathcal{E}_\mathbb{B})$ be a
trivialization. Then the functor
\begin{equation}
\Xi\colon \Def(\lambda^{-1}\underline{\mathcal{S}}, L)(R) \to
\Def(\lambda^{-1}\underline{\mathcal{S}})(R)
\end{equation}
is an equivalence. By the Theorem \ref{mt1}
\[
\Def(\lambda^{-1}\underline{\mathcal{S}}, L)(R)  \cong
\MC^2(\mathfrak{G}(\mathcal{A}) \otimes
    \mathfrak{m}_R)
\]
where $\mathcal{A}=\cma(\lambda^{-1}\underline{\mathcal{S}}, L)$.
Finally, we have equivalences
\[
\MC^2(\mathfrak{G}_\DR(\mathcal{A}) \otimes \mathfrak{m}_R)\cong
\MC^2(\mathfrak{G}_\DR(\mathcal{J}_{N\mathcal{E}}(\mathcal{A}))
\otimes \mathfrak{m}_R) \cong
\MC^2(\mathfrak{G}_\DR(\mathcal{J}_{N\mathcal{E}})_{[S(\mathcal{A})]}\otimes
\mathfrak{m}_R)
\]
induced by the quasiisomorphisms \eqref{iincl} and \eqref{cotr ind
of choice} respectively (with $X= N\mathcal{E}$). Recall the
morphism of cosimplicial $\mathcal{O}^\times$-gerbes $
\mathcal{S}_\mathcal{A} \to S(\mathcal{A})$ defined in
\ref{cosimplicial splitting}. The proof the theorem will be finished
if we construct a morphism of cosimplicial gerbes
$\mathcal{S}_\mathcal{A} \to
(\lambda^{-1}\underline{\mathcal{S}})_\Delta$.

For each $n = 0, 1, 2, \ldots$ we have $L^\s{n}_0 \in
(\lambda^{-1}\mathcal{S})^\s{n}_0 =
(\lambda^{-1}\underline{\mathcal{S}})_\Delta^n$, i.e. a morphism
$L_\Delta^n \colon \mathcal{O}^\times_{N_nG}[1] \to
(\lambda^{-1}\underline{\mathcal{S}})_\Delta^n$. For $f \colon [p]
\to [q]$ in $\Delta$, we have canonical isomorphisms $L_{\Delta f}
\colon (\lambda^{-1}\underline{\mathcal{S}})_{\Delta f} =
\mathcal{A}^\s{q}_{0 f(0)}$. It is easy to see that $(L_\Delta^n,
L_{\Delta f})$ defines a morphism $\mathcal{S}_\mathcal{A} \to
(\lambda^{-1}\underline{\mathcal{S}})_\Delta$.

\end{proof}

In certain cases we can describe a solution to the deformation
problem in terms of the nerve of the groupoid without passing to the
embedding category.
\begin{thm}\label{mt3} Let $G$ be an \'etale \emph{Hausdorff} groupoid and $\underline{\mathcal{S}}$ - a twisted form of $\mathcal{O}_G$.
  Then we have a canonical equivalence of $2$-groupoids
\[
\widetilde{\Def}(\underline{\mathcal{S}}) \cong
\MC^2(\mathfrak{G}_\DR(\mathcal{J}_{NG})_{[\underline{\mathcal{S}_{\Delta}}]}\otimes
\mathfrak{m}_R).
\]
\end{thm}
\begin{proof}
Suppose that $\mathbb{B}$ is a basis of $N_0G$ which consists of
  contractible open sets, and let
$\mathcal{E}=\mathcal{E}_{\mathbb{B}}(G)$ be the corresponding
embedding category. The map of simplicial spaces $\lambda \colon
N\mathcal{E} \to NG$ induces the map of subdivisions $\real{\lambda}
\colon \real{N\mathcal{E}} \to \real{NG}$. It induces a map $
\Tot(\real{\lambda}^*) \colon
\Tot(\Gamma(\real{X};\DR(\overline{\mathcal{J}}_{\real{NG}})) \to
\Tot(\Gamma(\real{X};\DR(\overline{\mathcal{J}}_{\real{N\mathcal{E}}}))$
Let $\overline{B} \in
\Tot(\Gamma(\real{NG};\DR(\overline{\mathcal{J}}_{\real{NG}}))$ be a
cycle defined in Proposition \ref{prop:class `1-forms-gerbe} and
representing the lift of the class of $\underline{\mathcal{S}}$ in
$H^2(\Tot(\Gamma(\real{NG};\DR(\overline{\mathcal{J}}_{\real{NG}})))$.
This form depends on choices of several pieces of data described in
\ref{dlog-gerbe}. Then $\Tot(\real{\lambda}^*)\overline{B} \in
\Tot(\Gamma(\real{NG};\DR(\overline{\mathcal{J}}_{\real{N\mathcal{E}}}))$
would be the form representing the class of
$\lambda^{-1}\underline{\mathcal{S}}$ constructed using pullbacks of
the data used to construct $\overline{B}$. We therefore obtain a
morphism of DGLA $\real{\lambda}^* \colon
\mathfrak{G}_\DR(\mathcal{J}_{NG})_{\overline{B}
}\to\mathfrak{G}_\DR(\mathcal{J}_{N\mathcal{E}})_{\Tot(\real{\lambda}^*)\overline{B}
}$. It is enough to show that this morphism is a quasiisomorphism.
To see this  filter both complexes by the (total) degree of
differential forms. The map $\real{\lambda}$ respects this
filtration and therefore induces a morphism of the corresponding
spectral sequences. The $E_1$ terms of these spectral sequences are
$\Tot(\Gamma(\real{NG};\real{\DR(
{HH}^\bullet(\mathcal{J}_{NG})[1])}^\prime))$ and
$\Tot(\Gamma(\real{N\mathcal{E}};\real{\DR({
HH}^\bullet(\mathcal{J}_{N\mathcal{E}})[1])}^\prime))$ respectively,
and the second differential in these spectral sequences is given by
$\widetilde{\nabla}$. Here ${ HH}^\bullet(.)$ is the cohomology of
the complex $\overline{ C}^\bullet (.)$  Since both $NG$ and
$N\mathcal{E}$ are Hausdorff, the complexes $\DR({
HH}^\bullet(\mathcal{J}_{N\mathcal{E}}))$ $ \DR({
HH}^\bullet(\mathcal{J}_{NG}))$ are resolutions of the soft sheaves
${HH}^\bullet(\mathcal{O}_{N\mathcal{E}})$,
${HH}^\bullet(\mathcal{O}_{NG})$ respectively.

By the results of  \ref{subsubsection: equiv for ab sh}
\[
\lambda^*\colon C(\Gamma(NG;
 {HH}^\bullet(\mathcal{O}_{NG})[1])) \to
C(\Gamma(N\mathcal{E}; {HH}^\bullet(\mathcal{O}_{N\mathcal{E}})[1]))
\]
is a quasiisomorphism. Hence
\[
\real{\lambda}^*\colon C(\Gamma(\real{NG};
\real{{HH}^\bullet(\mathcal{O}_{NG})[1]}^\prime)) \to
C(\Gamma(\real{N\mathcal{E}};
\real{{HH}^\bullet(\mathcal{O}_{N\mathcal{E}})[1]}^\prime))
\]
 is a quasiisomorphism, by the Lemma \ref{lemma: subdiv quism}.
From this one concludes that
\[
\Tot(\real{\lambda}^*)\colon \Tot(\Gamma(\real{NG};
\real{{HH}^\bullet(\mathcal{O}_{NG})[1]}^\prime)) \to
\Tot(\Gamma(\real{N\mathcal{E}};
\real{{HH}^\bullet(\mathcal{O}_{N\mathcal{E}})[1]}^\prime))
\]
is a quasiisomorphism. Hence
\[
\Tot(\real{\lambda}^*) \colon \Tot(\Gamma(\real{NG};\real{\DR(
{HH}^\bullet(\mathcal{J}_{NG})[1])}^\prime)) \to
\Tot(\Gamma(\real{N\mathcal{E}};\real{\DR({HH}^\bullet(\mathcal{J}_{N\mathcal{E}})[1])}^\prime))
\]
is a quasiisomorphism. Hence $\Tot(\real{\lambda}^*)$ induces a
quasiisomorphism of the $E_1$ terms of the spectral sequence, and
therefore is a quasiisomorphism itself.
\end{proof}

\subsection{Deformations of convolution algebras}
Assume that $G$ is an \'etale groupoid with $N_0G$ Hausdorff. In the following we will treat $N_0G$ as a subset of
$N_1G$.

With an object $(\underline{\mathcal{C}},L)\in\Triv_R(G)$ one can
canonically associate the nonunital  algebra
$\conv(\underline{\mathcal{C}},L)$, called the convolution algebra,
cf. \cite{tuxu}. Let $\mathcal{A}=\cma(\underline{\mathcal{C}},L)$,
see \ref{Cosimplicial matrix algebras from algebroid stacks}.
 The underlying vector space of this algebra is  $\Gamma_c(N_1G;
\mathcal{A}_{01}^1)$. Here we use the definition of the compactly
supported sections as in \cite{crainic}. The product is defined by
the composition
\begin{multline}
\Gamma_c(N_1G; \mathcal{A}_{01}^1) \otimes \Gamma_c(N_1G;
\mathcal{A}_{01}^1) \to \Gamma_c(N_2G; \mathcal{A}_{01}^2 \otimes
\mathcal{A}_{12}^2) \to \\ \Gamma_c(N_2G; \mathcal{A}_{02}^2)  \cong
\Gamma_c(N_1G, (d_1^1)_! \mathcal{A}_{02}^2) \to \Gamma_c(N_1G,
\mathcal{A}_{01}^1).
\end{multline}

Here the first arrow maps $f \otimes g$ to $f^2_{01} \otimes
g^2_{12}$, the second is induced by the map \eqref{mat mult ijk}.
Finally the last arrow is induced by the ``summation along the
fibers'' morphism $(d_1^1)_! \mathcal{A}_{02}^2 \to {A}_{01}^1$.

  Recall that a multiplier for a nonunital $R$-algebra $A$ is a pair
$(l, r)$ of $R$-linear maps $A\to A$ satisfying
\begin{equation*}
l(ab)=l(a)b, \ r(ab)=ar(b), \ r(a)b=al(b) \text{ for } a, b, c \in A
\end{equation*}

Multipliers of a given  algebra $A$ form an algebra denoted $M(A)$
with the operations given by $\alpha \cdot(l, r)+\alpha'\cdot(l',
r')=(\alpha l +\alpha' l', \alpha r+\alpha'r')$, $\alpha$, $\alpha'
\in R$ and $(l, r)\cdot(l', r')=(l' \circ l, r \circ r')$. The
identity is given by $(\id, \id)$. For $x=(l, r) \in M(A)$ we denote
$l(a)$, $r(a)$ by $xa$, $ax$ respectively.

Similarly to the $2$-category $\Alg^2_R$  (see \ref{subsubsection:
algebroids}) we introduce the $2$-category $(\Alg^2_R)'$  with
\begin{itemize}
\item objects -- nonunital $R$-algebras
\item $1$-morphisms -- homomorphism of $R$-algebras
\item $2$-morphisms $\phi\to\psi$, where $\phi,\psi \colon  A\to B$ are two $1$-morphisms -- elements $b\in M(B)$
    such that $b\cdot\phi(a) = \psi(a)\cdot b$ for all $a\in A$.
\end{itemize}

Suppose that $(\underline{\phi},\phi_\tau)\colon
(\underline{\mathcal{C}},L) \to (\underline{\mathcal{D}},M)$ is a
$1$-morphism in $\Triv_R(G)$. Let $F= \cma
(\underline{\phi},\phi_\tau)$, see \ref{Cosimplicial matrix algebras
from algebroid stacks}. Let $\conv(\underline{\phi},\phi_\tau)
\colon \conv (\underline{\mathcal{C}},L) \to
\conv(\underline{\mathcal{D}},M)$ be the morphism induced  by
$F^1_{01}$.
  Suppose that  $b \colon (\underline{\phi},\phi_\tau) \to
(\underline{\psi},\psi_\tau)$ is a $2-morphism$, where
$(\underline{\phi},\phi_\tau),(\underline{\psi},\psi_\tau)  \colon
(\underline{\mathcal{C}},L) \to (\underline{\mathcal{D}},M)$. Then
  $\cma(b)^0$ defines a $2$-morphism in $(\Alg^2_R)'$ between $\conv(\underline{\phi},\phi_\tau)$ and
$\conv(\underline{\psi},\psi_\tau)$. We denote this $2$-morphism by
$\conv(b)$.

\begin{lemma}
{~}\begin{enumerate}
\item The assignments $(\underline{\mathcal{C}},L) \mapsto \conv(\underline{\mathcal{C}},L)$,
    $(\underline{\phi},\phi_\tau) \mapsto \conv(\underline{\phi},\phi_\tau)$, $b \mapsto \conv(b)$ define a functor
\[
\conv \colon \Triv_R(G) \to (\Alg^2_R)'
\]

\item The functors $\conv$ commute with the base change functors.
\end{enumerate}
\end{lemma}
Assume  that $(\underline{\mathcal{S}},L)\in \Triv_{\mathbb{C}}(G)$
where $\mathcal{S}$ is a twisted form of $\mathcal{O}_G$. Let  $R$
be a Artin $\mathbb{C}$-algebra. An $R$-deformation of
$\conv(\underline{\mathcal{S}},L)$ is an associative $R$-algebra
structure $\star$ on the $R$-module
$B=\conv(\underline{\mathcal{S}},L)\otimes_{\mathbb{C}}R$ with the
following properties:
\begin{enumerate}
\item The product induced on $B\otimes_R \mathbb{C} \cong \conv(\underline{\mathcal{S}},L)$ is the convolution
    product defined above.
\item $ \supp (f*g) \subset\ d^1_1((d^1_0)^{-1}\supp(f)\cap (d^1_2)^{-1}\supp(g))$.
\end{enumerate}

A $1$-morphism between two such deformations $B_1$ and $B_2$ is an
$R$-algebra homomorphism $F\colon B_1 \to B_2$ such that
\begin{enumerate}
\item The  morphism $ F\otimes_R \mathbb{C} \colon \conv(\underline{\mathcal{S}},L) \to
    \conv(\underline{\mathcal{S}},L)$ is equal to $\id$.
\item $ \supp (F(f)) \subset \supp(f)$ for any $f \in B_1$.
\end{enumerate}

Given two deformations $B_1$ and $B_2$ and two $1$-morphisms $F_1$,
$F_2 \colon B_1 \to B_2$ a $2$ morphism between them is given by a
$2$-morphism $b=(l, r)$ in $(\Alg^2_R)'$ such that
\begin{enumerate}
\item The  $2$-morphism $b\otimes_R \mathbb{C} \colon  F_1\otimes_R \mathbb{C} \to \colon  F_1\otimes_R \mathbb{C}
    $ is equal to $\id$.
\item $ \supp (l(f)) \subset \supp(f)$, $ \supp (r(f)) \subset \supp(f)$.
\end{enumerate}

Thus, given $(\underline{\mathcal{S}},L)\in \Triv_{\mathbb{C}}(G)$,
we obtain a two-subgroupoid $\Def
(\conv(\underline{\mathcal{S}},L))(R) \subset (\Alg^2_R)'$ of
deformations of $\conv(\underline{\mathcal{S}},L)$.

Let $\mathcal{A}=\cma(\underline{\mathcal{S}},L)$ and let
$\mathbf{B} \in \Def(\mathcal{A})(R)$. Notice that for any $B \in
\Def (\conv(\underline{\mathcal{S}},L))(R)$ we have a canonical
isomorphism of vector spaces
\begin{equation}\label{idntt}
i\colon B \to \Gamma_c(N_1G; \mathbf{B}_{01}^1)
\end{equation}

\begin{lemma}\label{lemma: def conv cma}
Suppose that $B \in \Def (\conv(\underline{\mathcal{S}},L))(R)$.
There exists a unique up to unique isomorphism $\mathbf{B} \in
\Def(\mathcal{A})(R)$ such that the isomorphism \eqref{idntt} is an
isomorphism of algebras.
\end{lemma}
\begin{proof}
Let $U \subset N_2G$ be a Hausdorff   open subset such that
$d^1_i|_U$, $i=0$, $1$, $2$ is a diffeomorphism. Define an
$R$-bilinear map
\[
m_U\colon   \Gamma_c(U; \mathcal{A}^2_{01}\otimes R) \otimes_R
\Gamma_c(U; \mathcal{A}^2_{12}\otimes R) \to \Gamma_c(U;
\mathcal{A}^2_{02}\otimes R)
\]
by
\[
m_U(f, g) =
(d^1_1|_U)^*(((d^1_0|_U)^{-1})^*f\star((d^1_2|_U)^{-1})^*g).
\]
Here, we view $((d^1_0|_U)^{-1})^*f$, $((d^1_2|_U)^{-1})^*g$ as
elements of $B$. It follows from the locality property of the
product $\star$ that $\supp(m_U(f, g))\subset \supp(f)\cap
\supp(g)$. Peetre's theorem \cite{peetre} implies that $m_U$ is a
bidifferential operator. If $V \subset N_2G$ is another Hausdorff
open subset, then clearly $(m_V)|_{U \cap V}= (m_U)|_{U \cap V}$.
Therefore, there exists a unique element $m \in
\shHom_R((\mathcal{A}^2_{01}\otimes R)\otimes_R
(\mathcal{A}^2_{12}\otimes R), (\mathcal{A}^2_{02}\otimes R)) $
given by a bidifferential operator, such that $m|_U=m_U$ for every
Hausdorff open subset of $U \subset N_2G$. Note that $m
\otimes_{R}\mathbb{C}$ is the map $
\mathcal{A}^2_{01}\otimes_{\mathbb{C}} \mathcal{A}^2_{12}\to
\mathcal{A}^2_{02} $.

Now define  for $i \le j$ $\mathbf{B}^n_{ij} =\mathcal{A}^n_{ij}
\otimes_{\mathbb{C}}R$.   For $i \le j \le k $ $\mathbf{B}^n_{ij}
\otimes _R \mathbf{B}^n_{jk} \to \mathbf{B}^n_{ik}$ is given by
$m^\s{n}_{ijk}$. In particular this endows $\mathbf{B}^n_{ij}$, $i
\le j$, with the structure of $\mathbf{B}^n_{ii} -\mathbf{B}^n_{jj}$
bimodule. The map $m^\s{n}_{ijk}$ induces an isomorphism of
$\mathbf{B}^n_{kk}$-bimodules $\mathbf{B}^n_{ij}
\otimes_{\mathbf{B}^n_{jj}} \mathbf{B}^n_{jk} \xrightarrow{\cong}
\mathbf{B}^n_{ik}$. For $i>j$ set $\mathbf{B}^n_{ij}=
\shHom_{\mathbf{B}^n_{ii}}(\mathbf{B}^n_{ji}, \mathbf{B}^n_{ii})$.
We then have a canonical isomorphism $ \mathbf{B}^n_{ij}
\otimes_{\mathbf{B}^n_{jj}}\mathbf{B}^n_{ji} \xrightarrow{\cong}
\mathbf{B}^n_{ii}$. Therefore we have a canonical isomorphism
\[
\mathbf{B}^n_{ij}= \mathbf{B}^n_{ij}
\otimes_{\mathbf{B}^n_{jj}}\mathbf{B}^n_{ji}
\otimes_{\mathbf{B}^n_{ii}}
\shHom_{\mathbf{B}^n_{jj}}(\mathbf{B}^n_{ji},
\mathbf{B}^m_{jj})=\shHom_{\mathbf{B}^n_{jj}}(\mathbf{B}^n_{ji},
\mathbf{B}^m_{jj}).
\]

With this definition we extend the pairing $\mathbf{B}^n_{ij}
\otimes _R \mathbf{B}^n_{jk} \to \mathbf{B}^n_{ik}$ to all values of
$i$,$j$, $k$. For example for $i\ge j \le k$ this pairing is defined
as the inverse of the isomorphism
\[
\mathbf{B}^n_{ik}=\mathbf{B}^n_{ij}
\otimes_{\mathbf{B}^n_{jj}}\mathbf{B}^n_{ji}
\otimes_{\mathbf{B}^n_{ii}}\mathbf{B}^n_{ik} \xrightarrow{\cong}
\mathbf{B}^n_{ij} \otimes_{\mathbf{B}^n_{jj}}\mathbf{B}^n_{jk}.
\]
We leave the definition of this pairing in the remaining cases to
the reader. Choice of an $R$-linear differential isomorphism
$\mathbf{B}^1_{10}\cong \mathcal{A}^1_{10}\otimes_{\mathbb{C}}R$
induces isomorphisms $\mathbf{B}^n_{ij}\cong
\mathcal{A}^n_{ij}\otimes_{\mathbb{C}}R$ for all $n$ and $i >j$. We
thus obtain an object $\mathbf{B} \in \Def(\mathcal{A})(R)$. It is
clear from the construction that the map \eqref{idntt} is an
isomorphism of algebras. We leave the proof of the uniqueness to the
reader.

\end{proof}
 We denote the cosimplicial matrix
algebra $\mathbf{B}$ constructed in Lemma \ref{lemma: def conv cma}
by $\mat(B)$. By similar arguments using Peetre's theorem one
obtains the following two lemmas.

\begin{lemma}\label{lemma: def conv cma 1-mor}
Let $B_k\in\Def(\conv(\underline{\mathcal{S}},L))(R) $, $k=1$, $2$,
and let $i_k$ be the corresponding isomorphisms defined in
\eqref{idntt}. Let $F \colon B_1 \to B_2$ be a $1$-morphism. Then,
there exists a unique $1$-morphism $\phi \colon \mat(B_1) \to
\mat(B_2)$ such that $\phi^1\circ i_1=i_2 \circ F$.
\end{lemma}

We will denote the $1$-morphism $\phi$ constructed in Lemma
\ref{lemma: def conv cma 1-mor} by $\mat(F)$.

\begin{lemma}\label{lemma: def conv cma 2-mor}
Let $B_k\in\Def(\conv(\underline{\mathcal{S}},L))(R) $, $k=1$, $2$,
and let $i_k$ be the corresponding isomorphisms defined in
\eqref{idntt}. Let $F_1$, $F_2 \colon B_1 \to B_2$ be two
$1$-morphisms. Let $b\colon  F_1 \to F_2$ be a $2$-morphism. Then,
there exists a unique $2$-morphism  $\beta \colon \mat(F_1) \to
\mat(F_2)$ such that $b \cdot a =\beta^1_{00} \cdot i_2(a)$, $a
\cdot b =i_2(a)\cdot \beta^1_{11}$ for every $a \in B$.
\end{lemma}

We will denote the $2$-morphism $\beta$ constructed in Lemma
\ref{lemma: def conv cma 2-mor} by $\mat(b)$.

\begin{lemma}
{~}\begin{enumerate}
\item The assignments $B \mapsto \mat(B)$, $F \mapsto \mat(F)$, $b \mapsto \mat(b)$ define a functor
\begin{equation}\label{the functor mat}
\mat\colon \Def(\conv(\underline{\mathcal{S}},L))(R) \to
\Def(\cma(\underline{\mathcal{S}},L))(R)
\end{equation}

\item The functors $\mat$ commute with the base change functors.
\end{enumerate}
\end{lemma}

\begin{prop}\label{mateq}
The functor \eqref{the functor mat} induces an equivalence \[\Def
(\conv(\underline{\mathcal{S}},L))(R) \cong
 \Def(\cma(\underline{\mathcal{S}},L))(R).\]
\end{prop}

\begin{thm}
Assume that $G$ is a Hausdorff \'etale groupoid. Then, there exists a canonical equivalence of categories
\[
\Def (\conv(\underline{\mathcal{S}},L))(R) \cong
\MC^2(\mathfrak{G}_\DR(\mathcal{J}_{NG})_{[\underline{\mathcal{S}_{\Delta}}]}
\otimes \mathfrak{m}_R)
\]
\end{thm}
\begin{proof}
Let $\mathcal{A}=\cma(\underline{\mathcal{S}},L)$. Then, $\Def
(\conv(\underline{\mathcal{S}},L))(R) \cong
 \MC^2(\mathfrak{G}(\mathcal{A}) \otimes \mathfrak{m}_R)$ by Proposition \ref{mateq} and Theorem \ref{mt1}. Then, as in the proof of Theorem \ref{equiv of 2-gpds}, we have the equivalences
\[
 \MC^2(\mathfrak{G}(\mathcal{A}) \otimes \mathfrak{m}_R)\cong \MC^2(\mathfrak{G}(\mathcal{J}_{N\mathcal{E}}(\mathcal{A})) \otimes \mathfrak{m}_R)
\cong
\MC^2(\mathfrak{G}_\DR(\mathcal{J}_{N\mathcal{E}})_{[\underline{\mathcal{S}_{\Delta}}]}\otimes
\mathfrak{m}_R)
\]
induced by the quasiisomorphisms \eqref{iincl} and \eqref{cotr ind
of choice} respectively.
\end{proof}

\end{document}